\RequirePackage{fix-cm}
\documentclass{svjour3}                     % onecolumn (standard format)
\smartqed  % flush right qed marks, e.g. at end of proof
\usepackage{graphicx}
\usepackage{pgfplots}

\usepackage{latexsym,mathrsfs}
\usepackage{amsmath,amssymb} 
\usepackage{enumerate,verbatim}
\usepackage{amsfonts}
\usepackage{graphicx}
\usepackage{algorithm}
\usepackage{algorithmic}
\usepackage{url}
\usepackage{color}
\usepackage{hyperref}
\usepackage{pifont}% http://ctan.org/pkg/pifont
\newcommand{\cmark}{\ding{51}}%
\newcommand{\xmark}{\ding{55}}%

\newtheorem{assumption}{Assumption}
\DeclareMathOperator{\argmin}{argmin}

\definecolor{brightpink}{rgb}{1.0, 0.0, 0.5}

\newcommand{\revise}[1]{{{\color{black} #1}}}

\begin{document}

\title{Algorithms for Nonnegative Matrix Factorization with the  Kullback-Leibler Divergence} 
\titlerunning{Algorithms for KL NMF}    
\author{Le Thi Khanh Hien         \and
        Nicolas Gillis  %etc.
}


\institute{L. T. K. Hien and N. Gillis \at Department of Mathematics and Operational Research,  
Facult\'e Polytechnique, Universit\'e de Mons,
Rue de Houdain 9, 7000 Mons, Belgium\\
              \email{thikhanhhien.le@umons.ac.be, nicolas.gillis@umons.ac.be}           %  \\
}

\date{Received: date / Accepted: date}

\maketitle


\begin{abstract}
Nonnegative matrix factorization (NMF) is a standard linear dimensionality reduction technique for nonnegative data sets. 
In order to measure the discrepancy between the input data and the low-rank approximation, the Kullback-Leibler (KL) divergence  is one of the most widely used objective function for NMF. 
It corresponds to the maximum likehood estimator when the underlying statistics of the observed data sample follows a Poisson distribution, and KL NMF is particularly meaningful for count data sets, such as documents or images. 
%In this paper, we consider NMF where the objective function minimizes the  between the input data and the low-rank approximation. 
In this paper, %we focus on KL NMF, and our contribution is threefold. 
%First, 
we first collect important properties of the KL objective function that are essential to study the convergence of KL NMF algorithms. 
Second, together with reviewing existing algorithms for solving KL NMF, we propose three new algorithms that guarantee the non-increasingness of the objective function. We also provide a global convergence guarantee
for one of our proposed algorithms. 
Finally, we conduct extensive numerical experiments to provide a comprehensive picture of the performances of the KL NMF algorithms.  
%on different types of data sets and under differ. 
\keywords{nonnegative matrix factorization \and 
  Kullback-Leibler divergence \and 
  Poisson distribution \and 
  algorithms}
\end{abstract}

% REQUIRED
%\begin{AMS}
%  68Q25, 68R10, 68U05
%\end{AMS}

\section{Introduction} 

Given a nonnegative matrix data $V\in\mathbb{R}_{+}^{m\times n}$ and a positive integer $r\leq \min(m,n)$, nonnegative matrix factorization (NMF)  is the problem of finding %an  approximate rank $r$ factorization of $V$. In particular, NMF finds 
$W\in\mathbb{R}_{+}^{m\times r}$ and $H\in\mathbb{R}_{+}^{r\times n}$ such that $V\approx WH$. The quality of the approximation is measured using an objective function, which typically has the form  
\[ 
D\left(V|WH\right) %=D\left(V|\hat{V}\right) 
\triangleq\sum_{i=1}^{m}\sum_{j=1}^{n}d\left(V_{ij}|(WH)_{ij}\right),
\]
 where $d(x|y)$ is a scalar cost function such that $d(x|y) \geq 0$ for all $x,y \geq 0$ and $d(x|y) = 0$ if and only if $x=y$. 
 NMF is then written as the following problem 
\begin{equation*}
\min_{W\in\mathbb{R}^{m\times r}_+,H\in\mathbb{R}^{r\times m}_+} D\left(V|WH\right).%\label{eq:NMF}
\end{equation*}
The most widely used class of the scalar cost functions $d(x|y)$ is the $\beta$-divergence~\cite{Fevotte2011} %defined as follow:  
$$d_{\beta}(x|y)=\begin{cases}
\frac{1}{\beta(\beta-1)}\big(x^{\beta}+(\beta-1)y^{\beta}-\beta xy^{\beta-1}\big) &{\rm if}\:\beta\in\mathbb{R}\setminus\left\{ 0,1\right\}, \\
x\log\frac{x}{y}-x+y  & {\rm if}\:\beta=1,\\
\frac{x}{y}-\log\frac{x}{y}-1  &{\rm if}\:\beta=0.
\end{cases}$$
In this paper, we are interested in the Kullback-Leibler (KL) divergence (also known as the I-divergence), which is the $\beta$-divergence with $\beta=1$.  The KL NMF problem can be rewritten as follows:  
\begin{equation}
\label{KLNMF}
\min_{W\in\mathbb{R}^{m\times r}_+,H\in\mathbb{R}^{r\times m}_+} D_{\rm KL} (V|WH), 
\end{equation}
where 
\begin{align*}
D_{\rm KL} (V|WH) := 
\sum_{i=1}^m \sum_{j=1}^n\big((WH)_{ij}- V_{ij}\log(WH)_{ij}\big)  +\sum_{i=1}^m \sum_{j=1}^n \big(V_{ij}\log V_{ij}-V_{ij} \big). 
\end{align*} 
%Without lost of generality, we assume $V$ has no zero column or zero row (otherwise, we can set the corresponding column of $H$ or the corresponding row of $W$ are 0 and remove the zero columns/rows from the problem).  --> later 
With the convention that $0 \times \log 0 =0$ and $a \times \log 0 = -\infty$ for $a>0$, the objective function in \eqref{KLNMF} is well-defined and it is an extended-value function, that is, $D_{\rm KL}\in [0,+\infty]$. 
KL NMF \eqref{KLNMF} is well-posed, that is, a solution always exists (see Section~\ref{sec:existence_sol} for more details),  
but the solution is in general 
non-unique even when removing the scaling and permutation ambiguities of the low-rank approximation $WH$; see~\cite{xiao2019uniq} and the references therein.  

\vspace{-0.1in}
\subsection{Motivation} \label{sec:motiv} Since the seminal paper of Lee and Seung \cite{Lee99}, NMF has been shown to be a very powerful model to extract perceptually meaningful features from high-dimensional data sets. Applications include facial feature extraction~\cite{Lee99,Guillamet}, recovery and document classification~\cite{Lee99,DING2008,SHAHNAZ2006}, unmixing hyperspectral images~\cite{Bioucas-Dias,Maetal2014}; see also~\cite{cichocki2009nonnegative,gillis2014,xiao2019uniq} and the references therein. 
The most widely used objective function for NMF is the Frobenius norm $\|V - WH\|_F^2$ which corresponds to the $\beta$-divergence with $\beta = 2$. For nonnegative data sets, the Frobenius norm is however not the theoretically most reasonable choice. In fact, it corresponds to the maximum likelihood estimator in the presence of additive i.i.d.\ Gaussian noise. Under this noise distribution, observing negative entries in a data set has a positive probability. 
Moreover, many nonnegative data sets are sparse in which case Gaussian noise is clearly not appropriate~\cite{Lee99,Fevotte2009_prob,chi2012tensors}. 
%Using the KL divergence instead of the Frobenius norm is motivated by statistical considerations. 

Let us assume that $V_{ij}$ is a sample of the random variable $\tilde{V}_{ij}$ following the Poisson distribution of parameter $(WH)_{ij}$, that is,  
\[ 
\mathbb{P}\big(\tilde{V}_{ij}=k\big) 
= \frac{(WH)_{ij}^k \, e^{-(WH)_{ij}}}{k!} \text{ for } k=0,1,2, \ldots 
\] 
%the entries of $V$ 
%When the entries of $V$ are assumed to be the realization of a Poisson distribution of parameter $(WH)_{ij}$, that is, 
Then the maximum likelihood estimator of $W$ and $H$, given $V$, is the solution of \eqref{KLNMF}. 
The Poisson distribution is particularly well suited for integer-valued data sets, such a documents represented by vector of word counts~\cite{chi2012tensors,leskovec2014mining}, 
or images which can be interpreted as a photon counting process~\cite{hasinoff2014photon}.  
\revise{KL NMF has also been used successfully in bioinformatics~\cite{lin2007convergence}, e.g., to  
  cluster samples of RNA sequencing gene expression data~\cite{dey2017visualizing}.}  
It is worth noting that KL NMF also makes sense when $V$ contains non-integer entries. In that case, $V_{ij}$ can be interpreted as the average of several samples of the random variable $\tilde{V}_{ij}$. 
%there is still an ambiguity with real-valued data  interpretation as the Poisson process generates integers. 

In the literature, NMF with the Frobenius norm (Fro NMF) has been thoroughly studied and well-documented. 
Among algorithms for Fro NMF, the block coordinate (BC) methods, which update one block of the variables at a time, have the best performance in practice, 
see for example  
\cite{cichocki2009nonnegative,Hsieh2011,Xu2013,gillis2014,Ang2018,Hien2019} and the references therein. 
Note that one ``block'' here can be a full matrix, a column, a row, or even just a scalar component of the matrices $W$ or $H$. 
The objective function of Fro NMF has several nice properties 
that allow us to apply some advanced development of BC methods for solving composite block-wise convex optimization problems, which subsume Fro NMF as a special case, to derive very efficient algorithms with some rigorous convergence guarantee~\cite{Hien2019,Xu2013}. 
One of the most important properties is that the gradients of the objective with respect to $W$ and $H$, that is, $W\mapsto \nabla_W D_{\rm{Fro}}$ and $H\mapsto \nabla_H D_{\rm{Fro}}$, 
are Lipschitz-continuous over $W \in \mathbb{R}^{m\times r}_+$, $H \in \mathbb{R}^{r\times n}_+$. 
Although the objective of KL NMF, similarly to $D_{\rm{Fro}}$, is block-wise convex (that is, the function   $W\mapsto D_{\rm{KL}}$ and $H\mapsto D_{\rm{KL}}$ are convex), $D_{\rm KL}$ is even not differentiable at $W$ or $H$ when $(WH)_{ij}=0$ for some $i,j$. Since $D_{\rm{KL}}$ does not possess the nice smooth properties of Fro NMF, 
the extension of the analysis of BC methods from Fro NMF to KL NMF is restricted. 
Proposing a good algorithm for solving KL NMF is therefore a more difficult task compared to the Fro NMF. 
In fact, there are much fewer papers studying algorithms for KL NMF in the literature; in particular, algorithms with convergence guarantee are scarce. To the best of our knowledge, the multiplicative updates (MU) with some modification of KL NMF (see Section \ref{sec:MU} for more details) 
is the only algorithm that has a subsequential convergence guarantee. Moreover, the MU are the most widely used algorithm for KL NMF, while it is well known that, for Fro NMF, the MU are slow and should not be used; see~\cite{gillis2014} and the references therein.  
 These observations motivate this work which analyses in details the properties and algorithms for KL NMF~\eqref{KLNMF}. 

\vspace{-0.1in}
\subsection{Contribution and outline}
Our main contribution is threefold.

\begin{enumerate}

\item In Section \ref{sec:KLprop}, we present important properties of KL NMF~\eqref{KLNMF} for the convergence analysis of KL-NMF algorithms.  
%under the multi-block minimization framework. 
Some are well-known while others have not yet been highlighted in the literature, to the best of our knowledge.  
% in the community and several are not. 

\item While existing algorithms for~\eqref{KLNMF} are briefly reviewed in Section~\ref{sec:sota}, we present two new algorithms in Section~\ref{sec:newmember}. They are (i) 
 a block mirror descent method (BMD), for which we prove the global convergence of its generated sequence to a stationary point of a slightly perturbed version of KL NMF, and (ii) a scalar Newton-type algorithm which monotonically decreases the objective function. To the best of our knowledge, BMD is the first algorithm for KL NMF that has a global convergence guarantee.  
We also propose a hybridization between the scalar Newton algorithm and the MU for which the objective function is guaranteed to be 
non-increasing. 

\item In Section \ref{sec:experiment}, 
we perform extensive numerical experiments on synthetic as well as real data sets to compare the performance of the algorithms. 
To the best of our knowledge, this is the first time such a comparison is performed. It provides a good picture on the performance of the algorithms in different scenarios. 

\end{enumerate}

%The conclusion presented in Section \ref{sec:conclusion} might be surprising: although some algorithms perform much better in some scenarios, the MU, which were the first proposed algorithm for KL NMF, is the more reliable algorithm, always performing among the best in most scenarios. 

  We hope that the paper will be a good reference for whomever is using KL NMF, or is interested in KL NMF algorithms.

%\subsection{Organization}
%This paper is organized as follows. %We first provide some well-known applications of the KL NMF problem in Section \ref{sec:KLNMP_app}. They are the practical answer to why good algorithms for solving the KL NMF are important, necessary and demanded. 
 %we analyse the KL objective function, several properties of the non-convex minimization problem \eqref{KLNMF} and its perturbed version are presented. 
 %Based on the properties of the KL objective provided in Section \ref{sec:KLprop}, 
 %we derive some new algorithms for solving \eqref{KLNMF} in Section . de the aper in .  
%\paragraph{Notations:} For a positive integer $n$, we use $[n]$ to denote the set $\{1,\ldots,n\}.$ We denote  $e_{m \times n}$ as a $m \times n$ matrix with all elements equal to 1. For a given matrix $W$, we use $W_{i,:}$ to denote its $i$-th row and use $W_{:,j}$ to denote its $j$-th column.
%\subsection{Contribution} 
%\section{Applications of KL-NMF}
%\label{sec:KLNMP_app}

\vspace{-0.1in}
\section{Some properties of KL NMF} 
\label{sec:KLprop}

In the remainder of the paper, for simplicity, we denote $D(V|WH) = D_{\rm{KL}}(V|WH)$ since we only consider the KL divergence. 
The objective function of KL NMF \eqref{KLNMF} is not finite at every point of its constraint set since $D(V|WH)=\infty$ when $V_{ij}>0$ and $(WH)_{ij}=0$. This fact makes the convergence analysis of 
algorithms for  \eqref{KLNMF} very challenging. Hence, in parallel with considering KL NMF~\eqref{KLNMF}, we propose to study the following  perturbed version 
\begin{align} \label{KLNMF_perturbed}
\min_{W, H} 
& \quad D(V|WH) \\ 
   \text{such that } &  
  W_{ik}\geq \varepsilon,   
  H_{kj} \geq \varepsilon \text{ for } 
  i \in [m], 
  j \in [n], 
  k \in [r],   \nonumber 
\end{align}
where $\varepsilon \geq 0$, and $[n]$ denotes the set $\{1,\ldots,n\}$. Problem \eqref{KLNMF_perturbed} is equivalent to KL NMF \eqref{KLNMF} when $\varepsilon=0$.  When $\varepsilon>0$, 
the objective of \eqref{KLNMF_perturbed} is finite at every point of its constraint set. 
Moreover, a solution of \eqref{KLNMF_perturbed} has all its  entries strictly positive when $\varepsilon > 0$. 
 However, for $\varepsilon$ sufficiently small 
 (we recommend to use the machine precision), then such entries can be considered as zeros, which will not influence the objective function of KL NMF much; see Proposition~\ref{prop:sol_existence} below. 
 
% if an entry equals $\varepsilon$ with $\varepsilon$ being as small as a computer precision then, under a practical point of view, this entry is considered as 0.

\vspace{-0.1in}
\subsection{Existence of solutions}
\label{sec:existence_sol}

The following proposition proves the existence of solutions of Problem~\eqref{KLNMF_perturbed} and provides a connection between the optimal value of \eqref{KLNMF} and its perturbed problem~\eqref{KLNMF_perturbed} with $\varepsilon>0$. The proof is given in Appendix~\ref{app:proofprop1}. 
\begin{proposition} 
\label{prop:sol_existence}
(A) Given $\varepsilon\geq 0$ and a nonnegative matrix $V$, Problem~\eqref{KLNMF_perturbed} attains its minimum, that is, it has at least one optimal solution. 

(B) Let $ \nu=\sum_{i=1}^m\sum_{j=1}^n V_{ij}$. Denote $D^*(V,\varepsilon)$ be the optimal value of 
Problem~\eqref{KLNMF_perturbed}. We have 
$D^*(V,\varepsilon) \leq D^*(V,0) + (\min\{ n+ mr, m+nr\} \sqrt{\nu}+mn\varepsilon)\varepsilon.
$
\end{proposition}

Proposition~\ref{prop:sol_existence}(A) is known for the case $\varepsilon=0$; see for example~\cite[Proposition 2.1]{FS2006}, while Proposition \ref{prop:sol_existence}(B) is new. 
Given $\delta \geq 0$, Proposition \ref{prop:sol_existence}(B) shows that by choosing $\varepsilon$ such that 
$
 \min\{n+mr,m+nr\}\sqrt{\nu} \varepsilon +  mn \varepsilon^2 =\delta, 
$
 an optimal solution of~\eqref{KLNMF_perturbed} is a $\delta$-optimal solution of KL NMF~\eqref{KLNMF}. 
  \revise{Proposition~\ref{prop:sol_existence}(B) shows that, 
  for $\varepsilon$ sufficiently small, the objective function of \eqref{KLNMF_perturbed} is not significantly larger than that of~\eqref{KLNMF}. However, it says nothing on the corresponding optimal solutions.  To the best of our knowledge, no sensitivity analysis of the optimal solutions exist in the NMF literature. In fact, %since NMF is non-convex (and NP-hard~\cite{vavasis2009complexity}), 
  there could exist several isolated optimal solutions (see, e.g., \cite{gillis2017introduction}, for some examples) which makes a sensitivity analysis difficult in the general case because the optimal solution can change drastically for an arbitrarily small perturbation (leading to an infinite condition number of the problem).  
  This is an interesting direction of further research.  }

\subsection{KKT and stationary points} 
\label{sec:KKT}

A pair $(W^*,H^*)$  is a Karush--Kuhn--Tucker (KKT) 
point of~\eqref{KLNMF_perturbed} if it satisfies the KKT optimality condition of  \eqref{KLNMF_perturbed}, that is, if the objective function $f(W,H) := D(V,WH)$ is differentiable at $(W^*,H^*)$, 
and for all $i\in [m]$ and $k\in [r]$:   $W_{ik}^*\geq\varepsilon$ and 
\begin{equation}
\label{KKTpoint}
\begin{aligned}  
 \frac{\partial f(W^*,H^*)}{\partial  W_{ik}} = \sum_{j=1}^{n}H_{kj}^*-\sum_{\{j | V_{ij}>0\}}^{n}V_{ij}\frac{H_{kj}^*}{(W^*H^*)_{ij}} & \geq 0,  (\ref{KKTpoint}a)\\
 (W^*_{ik}-\varepsilon)\frac{\partial f(W^*,H^*)}{\partial  W_{ik}} & =0,     (\ref{KKTpoint}b)\\
\end{aligned}
\end{equation}
and similarly for $H^*$, by symmetry. A pair $(W^*,H^*)$ is a stationary point of~\eqref{KLNMF_perturbed} if it is a feasible point of~\eqref{KLNMF_perturbed} that lies in the domain of $f$ and for all $i\in [m]$ and $k\in [r]$, it satisfies 
\begin{equation}
\label{eq:Stationary1}
\frac{\partial f(W^*,H^*)}{\partial  W_{ik}}(W_{ik}-W_{ik}^*)\geq 0  \text{ for any } W_{ik}\geq \varepsilon, 
\end{equation}
and similarly for $H^*$. 
For Problem~\eqref{KLNMF_perturbed}, it turns out that KKT points and stationary points coincide. 
%Let $(W^*,H^*)$ be a KKT point. 
Since $W_{ik}^*-\varepsilon\geq 0$, (\ref{KKTpoint}a) and  (\ref{KKTpoint}b) hold if and only if \eqref{eq:Stationary1} holds. Indeed, let $(W^*,H^*)$ satisfy~\eqref{eq:Stationary1} then choosing $W_{ik}$ 
in~\eqref{eq:Stationary1} to be $W_{ik}^*+1$ gives  (\ref{KKTpoint}a) while choosing $W_{ik}$ in~\eqref{eq:Stationary1} to be $\varepsilon$ or $2W_{ik}^*-\varepsilon$  gives (\ref{KKTpoint}b). The reverse direction is obvious. The same reasoning applies to $H^*$.  

The following proposition provides an interesting property of KKT points of  KL NMF \eqref{KLNMF}, see~\cite[Theorem 1]{Ho2008} for its proof.  Note that this property does not hold for its perturbed variant~\eqref{KLNMF_perturbed} with $\varepsilon>0$. 
\noindent 
\begin{proposition}
\label{prop:scaleproperty}
If $(W^*,H^*)$  is a KKT point of~\eqref{KLNMF}, then $W^*H^*$ preserves the row sums and the column sums of $V$, that is, 
\[ 
V e =(W^*H^*)e 
\quad 
\text{ and } 
\quad 
e^\top V = e^\top (W^*H^*), 
\] 
 where  $e$ denotes the vector of all ones of appropriate dimension. % $m \times n$ matrix with all elements equal to 1.
\end{proposition}

Let us define the following notion of a scaled pair $(W,H)$. 
\begin{definition}
\label{def:scalepoint}
We say $(W,H)$ is scaled if the optimal solution of the problem 
$ \min_{\alpha\in \mathbb{R}} D(V,\alpha WH)
$ 
is equal to 1. 
\end{definition} 
Hence $WH$ is scaled implies that one cannot multiply $WH$ by a  constant $\alpha$ (that is, one cannot scale $WH$) to reduce the error $KL(X,WH)$. 

By Proposition~\ref{prop:scaleproperty}, if $(W^*,H^*)$  is a KKT point of~\eqref{KLNMF} then $\sum_{i=1}^m \sum_{j=1}^n V_{ij} = \sum_{i=1}^m \sum_{j=1}^n (W^*H^*)_{ij}$ and hence is scaled. Moreover, we have the following result. 
\begin{proposition}
\label{prop:scalepoint}
A pair $(W,H)$ is scaled if and only if 
$$\sum_{i=1}^m \sum_{j=1}^n V_{ij} = \sum_{i=1}^m \sum_{j=1}^n (WH)_{ij}.$$
\end{proposition}
\begin{proof}
We have 
\begin{align*} 
D(V,\alpha WH)=\alpha\sum_{i,j}(WH)_{ij}-\sum_{i,j}V_{ij}\log(WH)_{ij}\\
\quad -\sum_{i,j}V_{ij}\log \alpha + \sum_{i,j} (V_{ij} \log V_{ij} - V_{ij}). 
\end{align*} 
The result follows from the equation $\nabla_\alpha D(V,\alpha WH)=0$.  \hfill{\ensuremath{\square}} 
\end{proof}
Proposition~\ref{prop:scalepoint}, although simple to prove, is not present explicitly in the literature. %, to the best of our knowledge. 
  However, it is rather interesting from a practical point of view: any \revise{feasible} solution $(W,H)$ to KL NMF can be improved simply by scaling it. 
  \revise{In fact, combining Definition~\ref{def:scalepoint} and 
  Proposition~\ref{prop:scalepoint}, we can compute 
  \[
  \alpha^*  
  = \argmin_{\alpha\in \mathbb{R}} D(V,\alpha WH) 
  = \frac{e^\top V e}{e^\top (WH) e}, 
  \] 
  and scale $W \; \leftarrow \; \alpha^* \, W $ 
to improve the feasible solution $(W,H)$. 
This could be used within any algorithm for KL NMF. }

\vspace{-0.1in}
\subsection{Relative smoothness}
\label{sec:RM}

In Section~\ref{sec:BMD}, we will propose a new algorithm that globally converges to a stationary point of~\eqref{KLNMF_perturbed} with $\varepsilon>0$, namely a block mirror descent method. %It means that it convergence to a stationary point.%, regardless of the initialization. 
An instrumental property of~\eqref{KLNMF_perturbed} to prove such a result is relative smoothness. Let us describe this property in details.  

The objective $D(V|WH)$ is convex in each block variable $W$ and $H$, but it is not jointly convex in $(W,H)$. 
Furthermore, $D(V|WH)$ does not possess the Lipschitz smoothness property, that is, the derivative of $D(V|WH)$ with respect to $W$ or $H$ is not Lipschitz continuous. 
%Functions having this property will be referred to as $L$-smooth. 
Recently, the authors in \cite{Bauschke2017} and \cite{Lu2018} introduce the notion of relative smoothness that is a generalization of the Lipschitz smoothness. 
\begin{definition}\cite[Definition 1.1]{Lu2018}
\label{def:relativesmooth}
Let $\kappa(\cdot)$ be any
given differentiable convex function defined on a convex set $Q$. The \revise{convex} function $ g(\cdot)$ is $L$-smooth relative to $\kappa(\cdot)$ on $Q$ if for any $x, y \in {\rm int} Q$,
there is a scalar $L$ for which
\begin{equation}
\label{eq:RM} g(y) \leq g(x) + \langle \nabla g(x), y - x\rangle +  L \mathcal B_\kappa(y, x),
\end{equation}
where  
\begin{equation}
\mathcal B_\kappa(y,x):=\kappa(y)-\kappa(x)-\langle \nabla \kappa(x), y-x \rangle, \text{ for all } \, x,y\in Q.
\end{equation} 
\end{definition}
 %Note that 
%\begin{equation}
%\label{eq:fixcolumn}
%D(T|WH)=\sum_{i=1}^n D(T_{:,i} | WH_{:,i}), 
%\end{equation}
%where $T_{:,i}$ is the $i$-th column of $T$. 

%It turns out that $D(V|WH)$ %when restricted to one column $h$ of $H$ 
%is a relative smooth function w.r.t.\ to $H$ and $W$ separately.  

The following proposition shows that, when restricted to a column of $H$, the KL objective function is a relative smooth function. 
Since $D(V|WH)=D(V^\top |H^\top W^\top )$, this result  implies that 
the KL objective function when restricted to a row of $W$ is also relative smooth. 
\begin{proposition}\cite[Lemma 7]{Bauschke2017}
\label{prop:relativesmooth}
Let $v \in \mathbb{R}^m_+$ and $W \in \mathbb{R}^{m \times r}_+$, and 
\begin{equation}
\label{Dh}
 D(v|Wh):=\sum_{i=1}^m \Big ((Wh)_i - v_i \log (Wh)_i + v_i \log v_i  - v_i\Big). 
\end{equation} 
Then the function $h\mapsto D(v|Wh)$ is relative smooth to $\kappa(h)=-\sum_{j=1}^r\log h_{j} $ with the relative smooth constant $L=\|v\|_{1}$.
\end{proposition}

\vspace{-0.1in}
\subsection{Self-concordant properties} 
\label{sec:self-con}

In Section~\ref{sec:SN}, we will propose a new monotone algorithm  to solve~\eqref{KLNMF_perturbed}, namely a scalar Newton-type algorithm.  
An instrumental property of~\eqref{KLNMF_perturbed} to prove the  monotonicity is the self-concordant properties of its objective function.   
Let us first define a self-concordant function. We adopt the definition in \cite[Chapter 5]{NesterovLecture2018}. 
 
 Consider a closed convex function $g: \mathbb E \to \mathbb R$ with an \emph{open domain}. Fixing a point $x\in {\rm dom} (g)$ and a direction $d\in \mathbb E$, let $\varphi_x(\tau)=g(x+\tau d)$. We define 
$
 \mathbf D g(x)[d] = \varphi_x'(0)$, $\mathbf D^2 g(x)[d,d]=\varphi_x''(0)$,  and $ \mathbf D^3 g(x)[d,d,d]=\varphi_x'''(0).
$

\begin{definition}
 We say the function $g$ belongs to the class $\mathcal{F}_{M}$ of self-concordant functions with parameter $M\geq0$, if 
 \[ \big|\mathbf D^3 g(x)[d,d,d]\big|\leq 2 M \|d\|^3_{\nabla^2 g(x)},
 \]
 where $\|d\|^2_{\nabla^2 g(x)}=\langle \nabla^2 g(x)d,d\rangle=\varphi_x''(0)$.
  The function $g$ is called standard self-concordant when $M=1$.
\end{definition}
The scalar function $g(x)=-\log(x)$ is a standard self-concordant function; see \cite[Example 5.1.1]{NesterovLecture2018}. 
The following proposition provides some useful properties to determine the self-concordant constant of a function.  
\begin{proposition}\cite[Theorems 5.1.1 and 5.1.2]{NesterovLecture2018} 
\label{prop:selfconcordant} 
~

(i) Let $g_1, g_2$ be self-concordant functions with constants $M_1, M_2$. Then the function $g(x)=\alpha g_1(x) + \beta g_2(x)$, where $\alpha, \beta$ are positive constants, is self-concordant with constant 
$M_g$$=$$\max\big\{\frac{1}{\sqrt{\alpha}} M_1, \frac{1}{\sqrt{\beta}}M_2 \big\}$. 

(ii) If $g(\cdot)$ is self-concordant with constant $M_g$, then $\phi(x)=g(\mathcal A(x))$, where $\mathcal A(x)=Ax+b$ is a linear operator, is also self-concordant with constant $M_\phi=M_g$. 
\end{proposition} 
Using Proposition \ref{prop:selfconcordant}, 
we can prove that the objective of the perturbed KL NMF problem~\eqref{KLNMF_perturbed} is 
 self-concordant with respect to a single entry of $W$ and   $H$. 
\begin{proposition}
\label{prop:self-conc-NMF} 
Given $V \in \mathbb{R}^{m \times n}_+$, $W \in \mathbb{R}^{m \times r}_+$ and $H \in \mathbb{R}^{r \times n}_+$, 
the scalar function 
\begin{align*}
 H_{kj} \mapsto D(V|WH) = \sum_{i=1}^m \Big(\sum_{a=1}^n W_{ia}H_{aj}-V_{ij}\log \sum_{a=1}^n W_{ia}H_{aj}\Big) + {\rm constant} 
 \end{align*}
%(where the other elements of $W$ and $H$ are fixed)% and the function $$ H_{:,j}\mapsto D(T|WH)=\sum_{i=1}^m \big(W_{i,:}H_{:,j} - T_{ij} \log(W_{i,:}H_{:,j})\big) + {\rm const}$$
%(where the other rows of $W$ and columns of $H$ are fixed)
 is a self-concordant function with constant $
\mathbf c_{H_{kj}}=\max_{\{i  |  V_{ij}>0\}}\Big\{\frac{1}{\sqrt{V_{ij}}}\Big\}.$
\end{proposition}

\vspace{-0.1in}
\section{Existing algorithms}% for solving the KL NMF problem}
\label{sec:sota}

In this section, we briefly review the most efficient algorithms for KL NMF.

\vspace{-0.1in}
\subsection{Multiplicative updates}
\label{sec:MU}
%\subsection{Majorized functions and BSUM}
%\label{sec:majorized}

Let us consider the linear regression problem: 
 Given $v \in \mathbb{R}^m_+$ and $W \in \mathbb{R}^{m \times r}_+$,  
  \[
 {\rm minimize}_{h \in \mathbb{R}^{r}_+} \; D_{\text{KL}}(v,Wh) .  
\]  
For $W$ fixed in KL NMF, this is the subproblem to be solved for each column of $H$. The multiplicative updates (MU) for solving this problem are given by 
 \[
 h \; \leftarrow \; h \circ \left(\frac{W^\top \frac{[v]}{[Wh]}}{W^\top ee^\top} \right),  
 \] 
 where $\circ$ and $\frac{[\cdot]}{[\cdot]}$ are the component-wise product and division between two matrices, respectively. 
 They  were derived by 
 Richardson~\cite{richardson1972bayesian} and 
 Lucy~\cite{lucy1974iterative}, who used them for image restoration, and rectification and deconvolution in statistical astronomy, respectively. 
 This algorithm was later referred to as the Richardson-Lucy algorithm, and is guaranteed to decrease the KL NMF objective function.  
 
In the context of NMF, Lee and Seung~\cite{Lee99} derived these updates again, in a matrix form, to update $W$ and $H$ alternatively in KL NMF. 
The MU can be easily generalized to the perturbed KL NMF 
problem~\eqref{KLNMF_perturbed}~\cite{takahashi2014global},  
and are given by 
\begin{equation} \label{eq:MUforH}
H  
\leftarrow 
\max\left(\varepsilon ,  
H  \circ 
\frac{
\big[ W^\top \frac{[V]}{[WH]}  \big]
}
{
\big[  W^\top ee^\top  \big] 
}
\right), 
W  
\leftarrow  
\max\left(\varepsilon , 
W  \circ 
\frac{
\big[ \frac{[V]}{[WH]} H  \big]
}
{
\big[  ee^\top H  \big]
} 
\right).  
\end{equation}   
 Note that, for $\varepsilon = 0$, the $\max(\varepsilon,\cdot)$ can be removed since the entries of $V$, $W$ and $H$ are nonnegative, which corresponds to the MU used by Lee and Seung.

The MU can be derived using the majorization-minimization (MM) framework, which is a two-step approach: 
\begin{enumerate}
\item majorization: find a majorizer, that is, 
a function that is equal to the objective function at the current iterate while being larger everywhere else on the feasible domain,  and 
\item minimization: minimize the majorizer to obtain the next iterate.  
\end{enumerate}  
We refer the interested reader to~\cite{fevotte2009nonnegative} for all the details. Moreover, the MU can be shown to belong to a specific MM framework, namely, the block successive upper-bound minimization (BSUM)  framework~\cite{hong2015unified,Razaviyayn2013}. 
For completeness, we describe the BSUM framework and its convergence guarantees in Appendix~\ref{app:bsum}. 
This allows to provide convergence guarantees for the MU. Let us anaylze two cases separately.

%Because of this connection, the MU for~\eqref  
% $\varepsilon=0$ and $\varepsilon>0$.% it is important noting that Theorem \ref{thrm:BSUM_convergence} is only applicable to the former case. 

\subsubsection{Case $\varepsilon=0$}  

In this case, 
the MU are not well-defined if $(WH)_{ij}=0$ for some $(i,j)$. Furthermore, the MU would encounter a zero locking phenomenon, that is, the MU cannot modify an entry of $W$ or $H$ when it is equal to 0. 
This phenomenon can be fixed by choosing an initial pair with  strictly positive entries.  
Moreover, the objective function is not directionally differentiable when $(WH)_{ij}=0$ and $V_{ij}>0$ for some $i,j$. 
Hence, the convergence to stationary points obtained in \cite[Theorem~2]{Razaviyayn2013} for BSUM does not apply. 
In fact, \cite{GZ2005}~provided a numerical evidence that the generated sequence may not converge to a KKT point, and, 
\cite[Section 6.2]{chi2012tensors} gave an example that MU may converge to a non KKT point. 
Although the convergence of the generated sequence by MU is not guaranteed, it is worth noting that MU in this case possesses an interesting scale-invariant property. 
\begin{proposition}
\label{prop:scaleMU} 
%\ngc{Don't we need conditions, like the MU are well defined: $WH > 0$ for $V > 0$? and/or no zero rows/columns of $V$?}
Let $\varepsilon=0$ and %\ngi{let $V,W,H$ satisfy...} 
denote $H^+$ (resp.\ $W^+$) the update of $H$ (resp.\ $W$) after applying one MU~\eqref{eq:MUforH} of $H$ (resp.~of $W$) on $(W,H)$. Suppose the MU for $H$ (resp.\ $W$) is well-defined, that is, $(WH)_{ij}>0$ for all $i,j$ and $W$ has no zero column (resp.\ $H$ has no zero row). %\ngc{This condition is unclear: do you mean $(WH)_{ij}>0$ when $V_{ij} > 0$? or $(WH)_{ij}>0$ for all $(i,j)$?} 
Then the MU of $H$ preserve the column sum of $V$, that is, 
%\[
$e^\top V = e^\top (WH^+) $, 
%\] 
while the MU of $W$ preserve the row sum of $V$, that is,  
$ Ve =  (W^+ H) e$. 
%and the MU update of $W$ preserves the row sum of $V$, that is  $\sum_{j=1}^n(W^+H)_{ij}=\sum_{j=1}^n V_{ij}$, $i\in [m]$.
\end{proposition}
\begin{proof}
We prove the result for $H$, the result for $W$ can be obtained by symmetry.  
We have for $j \in [n]$ that 
\begin{align*}
\sum_{i=1}^m(WH^+)_{ij}&=\sum_{i=1}^m\sum_{k=1}^r W_{ik}H^+_{kj}=\sum_{i=1}^m\sum_{k=1}^r  W_{ik}H_{kj}\frac{\sum_{l=1}^m W_{lk}V_{lj}/(WH)_{lj}}{\sum_{l=1}^m W_{lk}}\\
&=\sum_{k=1}^r  \sum_{i=1}^m W_{ik}H_{kj}\frac{\sum_{l=1}^m W_{lk}V_{lj}/(WH)_{lj}}{\sum_{l=1}^m W_{lk}}=\sum_{k=1}^r H_{kj}\sum_{l=1}^m  \frac{W_{lk}V_{lj}}{(WH)_{lj}} \\
&=\sum_{l=1}^m \sum_{k=1}^r V_{lj} \frac{W_{lk}H_{kj}}{(WH)_{lj}}=\sum_{l=1}^m V_{lj}. \qquad\qquad\qquad\qquad\qquad\qquad\qquad\quad\qed 
\end{align*}  
\end{proof}
Proposition \ref{prop:scaleMU} shows that any iterate of the MU, except the first one, is scaled (Definition \ref{def:scalepoint}).  
%an update of MU makes $(W^+,H^+)$ scaled since $\sum_{i=1}^m \sum_{j=1}^n(W^+H^+)_{ij} = \sum_{i=1}^m \sum_{j=1}^n V_{ij}$, see  and Proposition \ref{prop:scalepoint}. 
Hence, sometimes in the literature, the MU are used to scale 
the pair $(W,H)$ within KL NMF algorithms.

\subsubsection{Case $\varepsilon>0$}

For this case, \cite[Theorem 2]{Razaviyayn2013} implies that every limit point of the generated sequence of MU for solving Problem~\eqref{KLNMF_perturbed} with $\varepsilon>0$ is a stationary point  of Problem~\eqref{KLNMF_perturbed} (which is also a KKT point, see Section \ref{sec:KKT}).  %the first, the second, the fourth and the fifth items of Theorem \ref{thrm:BSUM_convergence} are satisfied. The third item is proved by a similar method to the proof for the compactness of the set $\mathcal A_4$ (see the proof of Proposition \ref{prop:sol_existence}). Therefore, we can apply Theorem \ref{thrm:BSUM_convergence} to derive that every limit point of the generated sequence of MU for solving Problem~\eqref{KLNMF_perturbed} with $\varepsilon>0$ is a stationary point  of Problem~\eqref{KLNMF_perturbed} (which is also a KKT point, see Section \ref{sec:KKT}). 
It is worth noting that the sub-sequential convergence of MU in this case can also be proved by using Zangwill's convergence theory as in \cite{takahashi2014global}.

\vspace{-0.1in}
\subsection{ADMM} \label{sec:ADMM}

The alternating direction method of multiplier (ADMM) is a standard technique to tackle low-rank matrix approximation problems~\cite{huang2016flexible}, and it was used to solve KL NMF~\eqref{KLNMF} in~\cite{Sun2014}. 
The first step is to reformulate~\eqref{KLNMF} as 
\begin{equation}
\label{eq:ADMM}
\begin{split}
\min_{W,H,W_+,H_+} &\quad D(V | \hat V) \\
\text{such that } & \hat V=WH, W=W_+, H=H_+, W_+ \geq 0, H_+ \geq 0, 
\end{split}
\end{equation}
where $\hat V$, $W_+$ and $H_+$ are auxiliary variables. 
%The next step is to construct the augmented Lagrangian, and  
ADMM alternately minimizes the augmented Lagrangian of~\eqref{eq:ADMM} over the variables $(\hat V,W,H,W_+,H_+)$, and updates the dual variables at each iteration. 
We refer the reader %to Appendix~\ref{app:algopseudo} for the pseudocode, and 
to~\cite{Sun2014} for more details. 
It is important noting that the objective is not monotonically decreasing under the updates of ADMM. Also, convergence to stationary points is not guaranteed. %in the non-convex case when the gradient of the objective is not Lipschitz continuous is not guaranteed. \ngc{Can you check this last sentence?} Moreover, the update of the dual variable is non-trivial as the step size is difficult to tune. 
In fact, we will see in the numerical experiments that ADMM does not converge in some cases.

\vspace{-0.1in}
\subsection{Primal-dual approach}
\label{sec:PD}

In~\cite{Yanez2017}, Yanez and Bach proposed a first-order primal-dual (PD) algorithm for KL NMF~\eqref{KLNMF}.
%We recall that when we restrict the objective of \eqref{KLNMF} to a column of $H$ or a row of $W$ we obtain a convex function, see \eqref{Dh}. 
PD employs the primal-dual method proposed in \cite{ChambolleP11} to tackle the convex subproblems in the columns of $H$ and rows of $W$; see \eqref{Dh}. 
To improve the performance of the primal-dual method, PD uses an automatic heuristic selection for the step sizes; see~\cite{Yanez2017} for the details. %the pseudocode in  Appendix~\ref{app:algopseudo}. 
PD is a heuristic algorithm, and it does not guarantee the monotonicity of the objective.% function.  

\vspace{-0.1in}
\subsection{A cyclic coordinate descent method}
\label{sec:CCD}

In~\cite{Hsieh2011}, 
Hsieh and Dhillon proposes to solve KL NMF~\eqref{KLNMF} using a cyclic coordinate descent (CCD) method. 
CCD alternately updates the scalars $W_{ik}$ or $H_{kj}$. %, one at a time. %\ngc{Is it cyclic? Need to check --For Fro NMF, they select the variable to update.}  
% while fixing the latest values of the other scalar elements of $W$ and $H$ and approximately 
  The subproblems in one variable are solved by a Newton method, in which a full Newton step is used, without line search. %pseudocode in  Appendix~\ref{app:algopseudo}.   
 Again, there is no convergence guarantee for this algorithm, 
although the Newton method used to approximately solve the subproblems in one variable may have some convergence guarantee, if properly tuned. 
In fact, we realized that there is a gap in the convergence proof of the Newton method without line search to the global minimum of the scalar subproblems~\cite[Theorem 1]{Hsieh2011}: Equation~(29) is not correctly computed which makes the proof incorrect.
% In fact, we will see in the numerical experiments that, because of that, CCD sometimes runs into numerical problems because variables are set to negative values. 
 This observation motivated us to introduce a new scalar Newton-type algorithm in Section~\ref{sec:SN}. 
\revise{Note that a similar algorithm was independently developed in~\cite{lin2020optimization}, along with an R package. Another R package is available that implements both 
CCD and MU~\cite{unpublishedCarbonetto}.}  
%\ngc{Couldn't we simply fix CCD by using $\max(\varepsilon , .)$ in its update? I think it would be worth it given that it performs very well in many cases... What do you think?} 

% 

\vspace{-0.1in}
\subsection{Two other algorithmic approaches}

%This section provides some other approaches to solve the KL NMF problem~\eqref{KLNMF} in the literature. 
In \cite{Li2012}, Li, Lebanon and Park proposed a Taylor series expansion to express Bregman divergences in term of Euclidean distances, leading to a heuristic scalar coordinate descent algorithm for solving NMF problems with Bregman divergences. 
The algorithm using this approach for solving KL NMF underperforms the CCD method from~\cite{Hsieh2011}, and hence we will not compare it in the numerical experiment section. 

More recently, Kim, Kim and Klabjan~\cite{kim2019scale} proposed a scale invariant power iteration (SCI-PI) algorithm to solve a class of scale invariant problems, and apply it to KL NMF~\eqref{KLNMF}. 
To establish convergence results for SCI-PI (see \cite[Theorem 9, 11]{kim2019scale}), the objective function of the scale invariant problem needs to be twice continuously differentiable on an open set containing the set $\{x: \|x\|=1\}$. However, the objective function of the corresponding sub-problem (see \cite[Lemma 14]{kim2019scale}) when applying SCI-PI to KL NMF violates this condition. Therefore, the theory of SCI-PI in \cite{kim2019scale} does not apply to KL NMF.

\vspace{-0.1in}
\section{New members of the algorithm collection}
\label{sec:newmember}

In this section, we present two new algorithms for KL NMF, a block mirror descent method in Section~\ref{sec:BMD}, 
and a new scalar Newton-type algorithm in Section~\ref{sec:SN}.  

\vspace{-0.1in}
\subsection{Block Mirror Descent Method }
\label{sec:BMD}

Let us first give details on the analysis of the block mirror descent (BMD) method (Section~\ref{sec:bmdtheory}), and then apply the result to solve the perturbed KL NMF problem~\eqref{KLNMF_perturbed} (Section~\ref{sec:bmdklnmd}).

\subsubsection{BMD method} \label{sec:bmdtheory}

The standard gradient descent (GD) scheme for solving the smooth optimization problem $\min_{x\in Q} g(x)$, with $g$ being $L$-smooth,  uses the following classical update
\begin{equation}
\label{eq:GD}
x^{k+1}=\arg\min_{x\in Q} 
\left\{g(x^k)+\langle \nabla g(x^k),x\text{$-$}x^k\rangle 
+ \frac{L}{2}\|x\text{$-$}x^k\|_2^2 \right\}. 
\end{equation}
When $g(\cdot)$ is not $L$-smooth but $L$-relative smooth to $\kappa(\cdot)$ (see Definition~\ref{def:relativesmooth}), 
the GD scheme can be generalized by replacing the Euclidean norm in \eqref{eq:GD} by the Bregman divergence $\mathcal B$, leading to the following mirror descent step 
\begin{equation}
\label{eq:MD}
x^{k+1}=\arg\min_{x\in Q} \{g(x^k)+\langle \nabla g(x^k),x-x^k\rangle + L\mathcal B_\kappa(x,x^k)\}.
\end{equation}
Let us now present BMD, that uses the above generalized GD scheme, but updating the variable block by block. For that, let us consider a problem of the form 
\begin{equation} \label{eq:compositev2}
\min_{ x \in \mathcal X} f(x), 
\end{equation}  
where $f$ is continuously differentiable on $\mathcal X$, $x := (x_1,x_2,\dots,x_s)$ and  $\mathcal X = \mathcal X_1 \times \mathcal X_2 \times \dots \times \mathcal X_s$ with $\mathcal X_i$ ($i=1,2,\dots,s$) being closed convex sets. 
Let us also make the following assumption. 
\begin{assumption}
\label{assump:RelativeSmooth}
~

(i) For all $\bar{x} \in \mathcal X$, the function 
$
x_i \mapsto f(\bar{ x}_1,\ldots,\bar{ x}_{i-1},x_i,\bar{ x}_{i+1},\ldots,\bar{ x}_s) 
$ 
is convex and relative smooth to $\kappa_i(x_i)$ with constant $L_i(\bar{ x}_1,\ldots,\bar{ x}_{i-1},\bar{ x}_{i+1},\ldots,\bar{ x}_s)$. 

(ii) There exist positive constants $\underline{L}_i$ and $\overline{L}_i$ such that 
\[
\underline{L}_i \leq L_i(\bar{ x}_1,\ldots,\bar{ x}_{i-1},\bar{ x}_{i+1},\ldots,\bar{ x}_s) \leq \overline{L}_i.
\]  
\end{assumption}
\noindent BMD for solving Problem~\eqref{eq:compositev2}  is described in Algorithm \ref{alg:BMD}.
\begin{algorithm}[ht!]
\caption{Block Mirror Descent for solving \eqref{eq:compositev2}} \label{alg:BMD}
\begin{algorithmic}[1]
\STATE Choose an initial point $x^{0}$ such that $x_i^0 \in {\rm int\, dom} \mathcal X_i$ for $i=1,\ldots,s$.
\FOR{$k=1,\ldots$}
\FOR{$i=1,\ldots,s$}
\STATE Let $f^k_i(x_i)=f\big(x_1^{k+1},\ldots,x_{i-1}^{k+1},x_i,x_{i+1}^k,\ldots,x_s^k\big)$.
\STATE Update 
\begin{equation}
\label{eq:BMDstep}
x^{k+1}_i = \arg\min_{x_i\in \mathcal X_i} \big\{f^k_i(x_i^k) +   \langle \nabla f^k_i(x_i^k),x_i-x_i^k \rangle + L_i^{k} \mathcal B_{\kappa_i}(x_i,x_i^k)\big\},
\end{equation}
where $L_i^{k}=L_i(x_1^{k+1},\ldots,x_{i-1}^{k+1},x_{i+1}^k,\ldots,x_s^k\big)$.
\ENDFOR
\ENDFOR
\end{algorithmic}
\end{algorithm} 

BMD described in Algorithm \ref{alg:BMD} is structurally identical to the block Bregman proximal gradient (BBPG) method presented in \cite{Teboulle2020} with cyclic update. 
However, BMD has an improvement over BBPG: it uses the step size $1/L_i^{k}$ which is larger than $\rho/L_i^{k}$ used in BBPG,  as $\rho \in (0,1)$. 
Our convergence analysis of Algorithm \ref{alg:BMD} (Theorem~\ref{thm:BMDconvergence} below, see its proof in Appendix~\ref{app:BMDconvergence}) is an extension of the primal gradient scheme (which is BMD for $s=1$) analysed in \cite{Lu2018}. It is worth noting that our result can be extended to composite optimization with essentially cyclic regime by using the technique in \cite[Section A.2]{Lu2018} and \cite[Section 2]{Teboulle2020}. In this paper, we only present the result of BMD with cyclic regime for~\eqref{eq:compositev2} to simplify the presentation.% and to be more accessible by the readers working on the KL NMF problem. 

% The following theorem proves the global convergence of Algorithm~\ref{alg:BMD} for solving~\eqref{eq:compositev2} under Assumption \ref{assump:RelativeSmooth}.   
\begin{theorem}
\label{thm:BMDconvergence}
Suppose Assumption \ref{assump:RelativeSmooth} is satisfied. Let $\{x^k\}$ be the sequence generated by Algorithm \ref{alg:BMD}. 
Let also $\Phi(x)= f (x)+ \sum_i I_{\mathcal X_i}(x_i)$, where $I_{\mathcal X_i}$ is the indicator function of $\mathcal X_i$.  We have 

(i) $\Phi \big(x^k\big)$ is non-increasing;

(ii) Suppose $\mathcal X_i \subset {\rm int\, dom} \, \kappa_i$.  If   $\{x^k\}$ is bounded and $\kappa_i(x_i)$ is strongly convex on bounded subsets of $\mathcal X_i$ that contain $\{x_i^k\}$, 
then every limit point of $\{x^k\}$  is a critical point of $\Phi$;

(iii)  If together with the conditions  in (ii) we assume that  $\nabla \kappa_i$ and $\nabla_i f$ are Lipschitz continuous on bounded subsets of $\mathcal X_i$ that contain $\{x_i^k\}$, 
then the whole sequence  $\{x^k\}$  converges to a critical point of $\Phi$. 

\end{theorem}
%\begin{proof}
%See .  
%\end{proof}

Let us now apply this new algorithm and convergence result to KL NMF.

\subsubsection{BMD for KL NMF} \label{sec:bmdklnmd}

Problem~\eqref{KLNMF_perturbed} has the form of Problem~\eqref{eq:compositev2}. Proposition~\ref{prop:relativesmooth} allows us to apply BMD to solve~\eqref{KLNMF_perturbed}, 
where the blocks of variables are the columns of $H$ and the rows of $W$. 
Using the notation of Proposition~\ref{prop:relativesmooth}, 
we have 
$ 
\nabla_h D(v|Wh)=\sum_{i=1}^m \big(1-\frac{v_i}{(Wh)_i}\big) W_{i,:},
$
where  $W_{i,:}$ is the $i$-th row of $W$, and 
$\mathcal B_\kappa (h,h^k)=\sum_{j=1}^r \frac{h_j}{(h^k)_j}-\log \frac{h_j}{(h^k)_j} -1.  
$
 The following proposition provides the closed-form solution for the mirror descent step, see its proof in \cite[Section 5.2]{Bauschke2017}.
\begin{proposition}
\label{prop:KLNMFrelative}
Using the notation of 
Proposition~\ref{prop:relativesmooth}, the problem 
\begin{equation}
\label{eq:KLMD}
\min_{h\geq \varepsilon} \Big \{ D(v|Wh^k) + \big\langle\nabla_h D(v|Wh)[h^k],h-h^k \big\rangle + L\mathcal{B}_\kappa(h,h^k)\Big\}
\end{equation}
has the following unique closed-form solution $h^{k+1}$: for $l \in [r]$,  
\begin{equation}
\label{eq:KLclosedform}
\begin{split}
\displaystyle
(h^{k+1})_l&=\max\left\{\frac{(h^k)_l}{1+\frac{1}{L} (h^k)_l\big(\sum\limits_{i=1}^mW_{il}-\sum\limits_{i=1}^m \frac{v_i W_{il}}{(Wh^k)_i}\big)},\varepsilon\right\}.  
\end{split}
\end{equation}
\end{proposition} 
%Let us now put Problem~\eqref{KLNMF_perturbed} in the form of Problem~\eqref{eq:compositev2} with $x_i$ representing the columns of $H$ and the rows of $W$. We 
Recall that, $h\mapsto D(v|Wh)$  is $\|v\|_{1}$-relative smooth to $\kappa(h)=-\sum_{j=1}^r\log h_{j} $. Together with Proposition \ref{prop:KLNMFrelative} which provides a closed-form solution of the update~\eqref{eq:BMDstep}, we can therefore easily apply BMD to~KL NMF.   
Assumption \ref{assump:RelativeSmooth} is satisfied, and Theorem \ref{thm:BMDconvergence} (i) implies that $\Phi \big(x^{k}\big)$ is non-increasing.  %We thus have $\Phi \big(x^{k}\big)\leq \Phi \big(x^{0}\big) $. 
Moreover, using a similar method as in the proof Proposition~\ref{prop:sol_existence} and noting that $\Phi \big(x^{k}\big)\leq \Phi \big(x^{0}\big) $, we can prove that BMD for KL NMF~\eqref{KLNMF_perturbed} with $\varepsilon>0$ generates a bounded sequence.  
Together with the assumption $x\geq \varepsilon$, we see that all conditions of Theorem \ref{thm:BMDconvergence} (iii) are satisfied. Hence, we obtain the following convergence guarantee for BMD applied 
to~\eqref{KLNMF_perturbed}.

%Considering the bounded convex sets $\{x_i:  x_i\geq \varepsilon>0, \|x_i\|\leq A_i\}$, where $A_i$ are some positive constants, we can verify that the Hessian $\nabla^2 \kappa (h)$, which is a diagonal matrix with its diagonal being $\big(1/h_1^2,\ldots,1/h_r^2 \big)$, is lower bounded on the sets.  We then can apply Theorem \ref{thm:BMDconvergence} (ii) to derive the following theorem that guarantees the subsequential convergence for the sequence  generated by BMD when solving the KL NMF problem. 
\begin{theorem}
\label{thm:KLNMFconvergence}
Algorithm \ref{alg:BMD} applied on the KL NMF problem~\eqref{KLNMF_perturbed} monotonically decreases the objective function for any $\varepsilon \geq 0$.  
Moreover, for $\varepsilon>0$, 
the generated sequence by BMD is bounded and globally converges to a stationary point of~\eqref{KLNMF_perturbed}. 
\end{theorem}
%\begin{proof} 
%This is a consequence of Proposition~\eqref{prop:KLNMFrelative} and Theorem~\eqref{thm:BMDconvergence}; see the discussion above. 
%\end{proof}
It is important noting that $\nabla \kappa (h)$ is not Lipschitz continuous on bounded sets containing 0. %$\|\nabla^2 \kappa (h)\|$ may go to $\infty$ when $h_i\to 0$, hence $\nabla \kappa (h)$ may not Lipschitz continuous on bounded sets. 
We hence cannot apply Theorem \ref{thm:BMDconvergence} (iii) to the KL NMF problem~\eqref{KLNMF}.

\vspace{-0.1in}
\subsection{A scalar Newton-type algorithm} \label{sec:SN}

%We note that neither a convergence of the objective sequence nor the convergence of the generated sequence is guaranteed when using CCD (see Section~\ref{sec:CCD}) for Problem~\eqref{KLNMF}. 

In this section, we propose a new scalar Newton-type (SN) algorithm that makes use of the self-concordant property of the objective function (see Section \ref{sec:self-con}) to guarantee the non-increasingness for the objective sequence, that consequently guarantees the convergence of the objective sequence since it is also bounded from below. 
The motivation to propose this new method comes from our observation that CCD does not come with any convergence guarantee, as explained in Section~\ref{sec:CCD}, although it performs well in many cases.  

Let us first have a brief review on Newton methods. Unconstrained minimization problem of a self-concordant function
%\begin{equation}
%\label{eq:Newtonproblem}
%\min_{x\in \mathbb E} g(x),
%\end{equation} 
can be efficiently solved by Newton methods, see \cite[Section 5.2]{NesterovLecture2018}. %In the following, we summarize some important properties of the well-known Newton methods for solving Problem  \eqref{eq:Newtonproblem}. 
Tran-Dinh et al.~\cite{Tran2015} brought the spirit of Newton methods for unconstrained optimization to composite optimization problems of the form 
\begin{equation}
\label{eq:composite_opt}
\min_{x\in \mathbb R^n } \Psi(x):=\psi(x) + \phi(x),
\end{equation}
where $\psi(x)$ is a standard self-concordant function and $\phi(x)$ is a proper, closed, convex but possibly non-smooth function. In particular, the authors propose a proximal Newton method (PNM) with the following update 
\begin{equation}
\label{eq:proximalNewton}
x^{k+1}=x^k + \alpha_k d^{k},
\end{equation}
 where $\alpha_k \in (0,1]$ is a step size,  $d^k=s^k-x^k$ and 
 \begin{align*}
 s^k=\arg\min_x\big\{ \psi(x^k) + \langle \nabla \psi(x^k), x-x^k \rangle+ \frac12 (x-x^k)^\top  \nabla^2 \psi (x^k) (x-x^k)   + \phi(x)\big\}. 
 \end{align*}
Denoting $\lambda_k=\|d^k\|_{x^k}=\big((d^k)^\top  \nabla^2  \psi (x^k) d^k \big)^{1/2}$, it follows from \cite[Theorem 6]{Tran2015} that when the stepsize $\alpha_k=(1+\lambda_k)^{-1}$ is used, then the PNM generates the sequence $x^k$ satisfying 
$\Psi(x^{k+1}) \leq \Psi(x^k)-\omega(\lambda_k)
$,
where $\omega(t)=t-\ln(1+t)$.% By \cite[Theorem 7]{Tran2015}, we further have that if Problem~\eqref{eq:composite_opt} has the unique solution $x^*$ and suppose that $\nabla^2 \psi(x) \succ 0 $ for all $x\in \mathcal W(x,1):=\{y\in {\rm dom}(\psi), \|y-x\|_x <1\}$ then $\{\lambda_k \}$ converges to $0^+$ at a quadratic rate when $\lambda_0 \leq 0.25(5-\sqrt{17})$ and $\alpha_k=1$ (that is, the proximal Newton direction with full stepsize is chosen).  

Recall that the KL objective function with respect to a scalar component of $W$ or $H$ is a self-concordant function, 
see Proposition~\ref{prop:self-conc-NMF}. We  can hence make use of the update in \eqref{eq:proximalNewton} to propose the SN method; see Algorithm \ref{alg:SN}.
\begin{algorithm}[ht!]
\caption{ SN for solving Problem~\eqref{KLNMF} } \label{alg:SN}
\begin{algorithmic}[1]
\STATE Choose initial points $W>0$, $H>0$. 
\STATE Compute the self-concordant constants $\mathbf c_{W_{ik}}$ and  $\mathbf c_{H_{kj}}$ of the function $W_{ik}\mapsto f$ and $H_{kj} \mapsto f$, see Proposition \ref{prop:self-conc-NMF}.
\REPEAT
\STATE Alternately update each scalar component of $W$ and $H$. 
We update $W_{ik}$ several times (similarly for $H_{kj}$) as follows. %\ngc{To be consistent, can you replace $W_{ik}$ by $W_{ik}$? (It is not good to use $W_{ik}$ and $H_{kj}$: it is confusing since these are not the same indices for $W$ and $H$. By the way, we should also check this throughout the paper: it seems you use $i$, $j$, $k$ in different ways... I prefer to be consistent to help the reader --For example, in (16), I would prefer $k \in [r]$.} 
\STATE Calculate 
\begin{align*}
f'_{W_{ik}}=\sum_{l=1}^n H_{kl}-\sum_{l=1}^n V_{il}\frac{H_{kl}}{(WH)_{il}}, 
 f''_{W_{ik}}=\sum_{l=1}^n V_{il}\frac{H_{kl}^2}{\big (WH)^2_{il}},
\end{align*}
 
and let $s=\max \big\{W_{ik}- \frac{f'_{W_{ik}}}{f''_{W_{ik}}},0\big\}$, $d=s-W_{ik}$, and $\lambda=\mathbf c_{W_{ik}}\sqrt{f''_{W_{ik}}}|d|$.
\IF{ $f'_{W_{ik}}\leq 0$ or$^{\dagger}$ $\lambda \leq 0.683802$}
\STATE update $W_{ik}$ by a full proximal Newton step 
$W_{ik} \leftarrow s,$
\ELSE
\STATE update $W_{ik}$ by
$W_{ik} \leftarrow W_{ik} + \frac{1}{1+\lambda} d$.
\ENDIF
\UNTIL{some criteria is satisfied} 
\end{algorithmic}
\footnotesize $\dagger$ Such $\lambda$ guarantees that $\lambda^2 +  \lambda + \log (1-\lambda) > 0$ implying that the objective is non-increasing under a full Newton step; see the proof in Appendix~\ref{app:proofSNmono}. 
\end{algorithm} 

The following proposition proves that SN monotonically decreases the objective function. The proof is provided in Appendix~\ref{app:proofSNmono}.  
\begin{proposition}
\label{prop:SN_nonincreasing}
%The objective sequence $\{f(W^k,H^k)\}$ generated by Algorithm~\ref{alg:SN} is non-increasing.
The objective function for the perturbed KL-NMF problem~\eqref{KLNMF} is non-increasing under the updates of Algorithm~\ref{alg:SN}. 
\end{proposition}
%\begin{proof}
%See . 
%\end{proof}

%Gradient 
%\[f'(H_{kj})=\sum_{l=1}^m W_{li}-\sum_{l=1}^m X_{lj}\frac{W_{li}}{(WH)_{lj}}
%\]
%Hessian 
%\[
%f''(H_{kj})=\sum_{l=1}^m X_{lj}\frac{W_{li}^2}{\big (WH)^2_{lj}}
%\]

\vspace{-0.1in}
\subsection{A hybrid SN-MU algorithm}

\label{sec:SNMU}

We recall that MU possesses an important property, namely that $(W,H)$ scaled after any MU update; see Proposition \ref{prop:scaleMU}. 
On the other hand, we note that KKT points of Problem~\eqref{KLNMF} are also scaled. However, % because SN only update one entry of $W$ or $H$ at a time, it 
SN does not possess this scale-invariant property. 
Hence we propose to combine SN 
%which is a second-order method that often produces solutions with good accuracy, NG: this is not a second-order method, globally, because it is a block coordiate descent method. 
with MU to result in a hybrid SN-MU algorithm. 
Specifically, we alternately run several SN steps before scaling the sequence by running one or several updates of MU. 
In Section~\ref{sec:experiment}, we will use 10 steps of SN followed by one step of MU for all numerical experiments. As we will see, this combination sometimes significantly improves the performance of SN.

\section{Experiments}
\label{sec:experiment}

In this section, we report comparisons of the KL NMF algorithms listed in Table~\ref{tab:algorithms}. 
\begin{table}[t]\centering
\caption{
Properties of KL NMF algorithms presented in this paper, applied on an $m$-by-$n$ matrix $V$ to be approximated with a rank-$r$ approximation $WH$. The parameter $d$ is the number of inner iterations.
}
\label{tab:algorithms}
\begin{tabular}{|c|c|c|c|c|}
\hline
Algorithms & Complexity & Convergence & Monotonicity & Reference \\
& (flops) & & & \\ 
\hline
MU & $O(mnr)$  & \cmark & \cmark & Sec.~\ref{sec:MU} \\ 
ADMM & $O(mnr) $ & \xmark &\xmark  & Sec.~\ref{sec:ADMM}\\
PD &$O(mnrd)$  & \xmark & \xmark & Sec.~\ref{sec:PD}\\
CCD  & $O(mnrd)$ & \xmark & \xmark& Sec.~\ref{sec:CCD}\\
BMD  & $O(mnr) $ & \cmark & \cmark& Sec.~\ref{sec:BMD}\\
SN  &  $O(mnrd)$ & \xmark &\cmark & Sec.~\ref{sec:SN}\\
SN-MU & $O(mnrd) $  & \xmark & \cmark & Sec.~\ref{sec:SNMU}\\
\hline
\end{tabular}
  
\end{table} 
The second column of Table~\ref{tab:algorithms} provides the complexity of one iteration to update all entries of $(W,H)$.  
The parameter $d$ in the second column of PD, CCD, SN and SN-MU is the number of inner iterations of one main iteration of updating $(W,H)$.  The third column indicates whether the corresponding algorithm has some convergence guarantee for its generated sequence. We note that, considering Problem~\eqref{KLNMF_perturbed} with $\varepsilon>0$,  MU guarantees a subsequential convergence while BMD guarantees a global convergence. The fourth column indicates if the sequence of the objective function values is non-decreasing.

\vspace{-0.1in}

\subsection{Implementation} \label{sec:implementation}

We have implemented MU and BMD in Matlab, SN in C++ and use its mex file to run it from Matlab, as for CCD provided by the authors\footnote{\url{http://www.cs.utexas.edu/~cjhsieh/nmf}} (for which we have fixed an issue on maintaining $(WH)_{ij}$, otherwise it sometimes run into numerical issues generating NaN objective function values because $(WH)_{ij}$ could take negative values).  
%\ngc{What about the other algorithms? ADMM, PD? What about the number of iterations in SN-MU?}  
We use the Matlab code provided by the authors 
for ADMM\footnote{\url{http://statweb.stanford.edu/~dlsun/admm.html}} and PD\footnote{\url{https://github.com/felipeyanez/nmf}}.  \revise{We used the best possible programming language for each algorithm. For example, if CCD was implemented on Matlab, it would be extremely slow as it loops over each variable (and Matlab is very ineffective to handle loops). On the other side, the MU  run faster on Matlab because the main computational cost resides in matrix-matrix multiplications for which Matlab is more effective than C++.}   
All tests are preformed using Matlab
R2018a on a laptop Intel CORE i7-8550U CPU @1.8GHz 16GB RAM.  The code is available at \url{https://github.com/LeThiKhanhHien/KLNMF}. %\ngc{Don't forget to put it there.} 
We choose the penalty parameter of ADMM to be equal to~1 in all of the experiments. In each run for a data set, we use the same random initialization and the same \revise{maximal} running time for all algorithms. 

 We use the Matlab commands $W=rand(m,r)$ and $H=rand(r,n)$ to generate a random initial point; and to avoid initial points with a large value $D(V,WH)$, we then scale $W$ and $H$ by $W=\sqrt{\alpha} W$, $ H=\sqrt{\alpha} H$, where 
$\alpha=\frac{\sum_{i,j} V_{ij}}{\sum_{i,j}(WH)_{ij}}$; 
see Definition~\ref{def:scalepoint} 
and Proposition~\ref{prop:scalepoint}. 
We define the relative error ${\rm rel} \, D(V,WH)$ to be the objective $D(V,WH)$ divided by $\sum_{i,j} V_{ij}\log\frac{V_{ij}}{(\sum_j V_{ij})/n}$, and \revise{denote $E(t)$ the value of ${\rm rel} \, D(V,WH) - e_{\min}$ at time $t$, 
where $e_{\min}$ is the smallest value among all ${\rm rel} \, D(V,WH)$ produced by all algorithms and all initializations within the allotted time. 
Hence $E(t)$ goes to zero for the best run among all algorithms and initializations.   } 
%We list the information of our experiments in Table~\ref{tab:synthetic}. 

%\begin{table}[t]\centering
%\caption{KL NMF algorithms}
%\label{tab:synthetic}
%\begin{tabular}{|c|c|c|}
%\hline
%Fig. & Description & Running time  \\
%\hline
%\multicolumn{3}{|c|}{Synthetic Data} \\
%\hline
%Fig.~\ref{fig:lowrank_200x200} & low rank $X=200 \times 200$ & 15  \\ 
%Fig.~\ref{fig:lowrank_500x500} & low rank $X=500 \times 500$ & 20 \\
%Fig.~\ref{fig:full rank} &$O(mnrd)$  & \xmark  \\
%CCD  & $O(mnrd)$ & \xmark  \\
%BMD  & $O(mnr) $ & \cmark  \\
%SN  &  $O(mnrd)$ & \xmark \\
%SN-MU & $O(mnrd) $  & \xmark  \\
%\hline
%\end{tabular}
%\end{table}
\vspace{-0.1in}

\subsection{Experiments with synthetic data sets}

For each type of synthetic data sets, we generate 10 random matrices $V$, 
then for each random $V$, we generate 10 random initial points.  
%We use $r=20$ in the experiments on synthetic data sets with size $500 \times 500$, and $r=10$ for the others. 

\vspace{-0.1in}

\subsubsection{Low-rank synthetic data sets}

%We generate a low-rank synthetic data set $V$ by using the Matlab commands $W=rand(m,r)$, $H=rand(r,n)$ to generate $W$ and $H$ and then letting $X=WH$. 

We will consider several types of low-rank synthetic data sets depending on the parameter $\ell \in (0,1]$ which is the density of the underlying factors, denoted $W^*$ and $H^*$. We will use $\ell = 1,0.9,0.3$. 
More precisely, to generate a low-rank synthetic data set $V$, we use the Matlab commands $W^* = sprand(m,r,\ell)$ and $H^* = sprand(r,n,\ell)$, where $\ell$ is the density of non-zero elements (that is, $1-\ell$ is the percentage of zero elements), and let $V = W^*H^*$.   
We will also either keep $V = W^*H^*$ as is, which is the noiseless case, or generate each entry of $V$  following a Poisson distribution of parameters $W^*H^*$ as described in Section~\ref{sec:motiv}, which is a noisy case and is achieved with the Matlab command $V = poissrnd(W^*H^*)$.

%In order to impose a Poisson noise to a synthetic data set $X$, we use the Matlab command . 

The results of applying the different algorithms on such matrices are reported in Figure~\ref{fig:lowrank_200x200} for 200-by-200 matrices with $r=10$, and in Figure~\ref{fig:lowrank_500x500} for 500-by-500 matrices with $r=20$. We report the evolution of the \revise{median value of $E(t)$. Although this is not an ideal choice, comparing the performance in term of iterations would be worse since the cost of one iteration can be rather different for each algorithm; for example, CCD has inner iterations, which is not the case of MU.} 
 %results for $200 \times 200$ (resp. $500 \times 500$) low rank data sets  (resp. ). 
 We also report the average and standard deviation (std) of the relative errors \revise{over 200 runs for the 6 types of synthetic data sets (100 runs for each size $200\times 200$ or $500\times 500$) in Table~\ref{tab:accuracy_lowrank}, and provide a ranking over the total 1200 runs between the different algorithms in Table~\ref{tab:rank_lowrank}}: the $i$th entry of the ranking vector indicates how many times the corresponding algorithm obtained the $i$th best solution (that is, with the $i$th lowest objective function value). This table allows to see which algorithms performs on average the best on these data sets. 

\begin{figure*}[ht]
\begin{center}
\begin{tabular}{cc}
\includegraphics[width=0.43\textwidth]{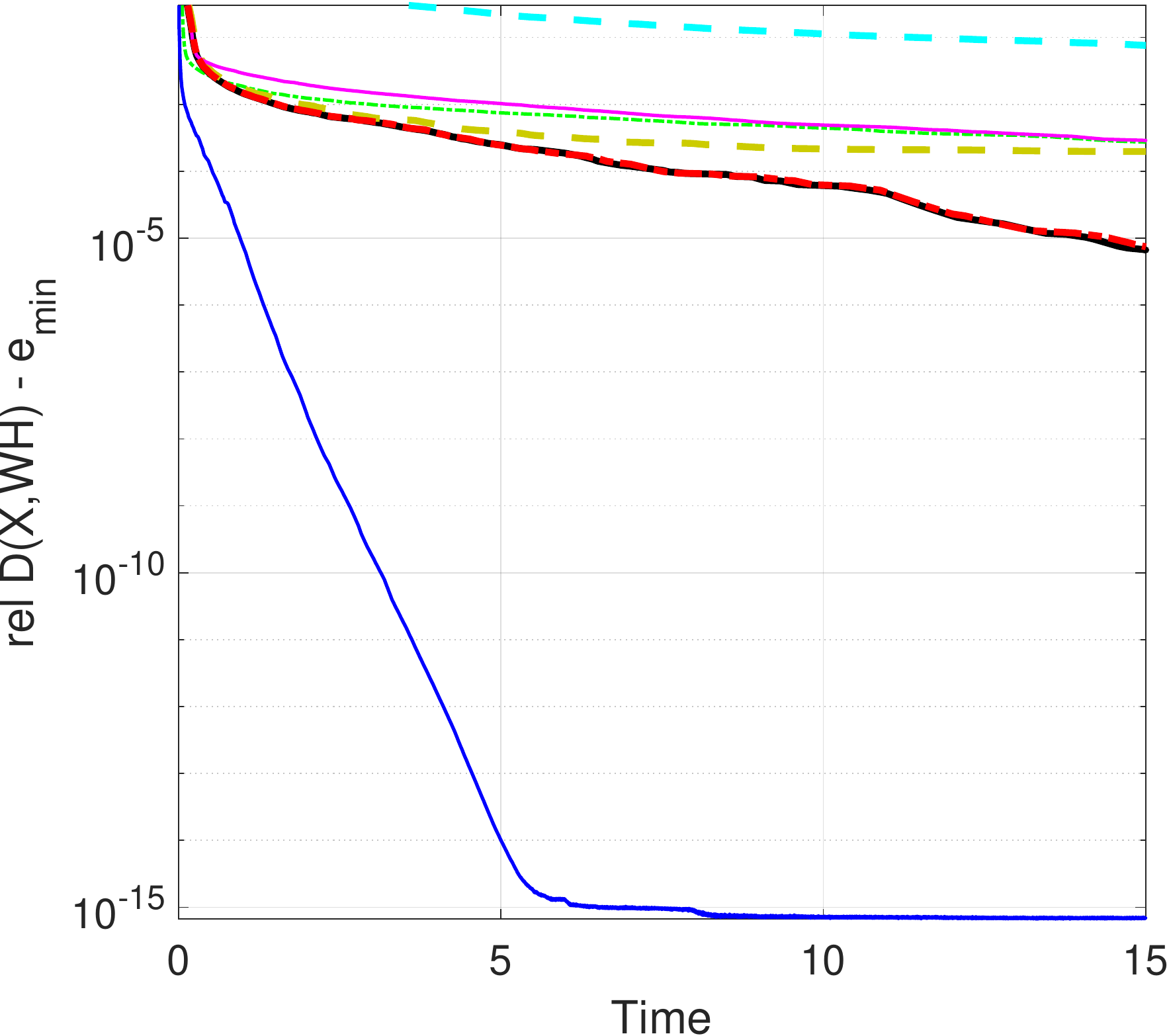}  & 
\includegraphics[width=0.425\textwidth]{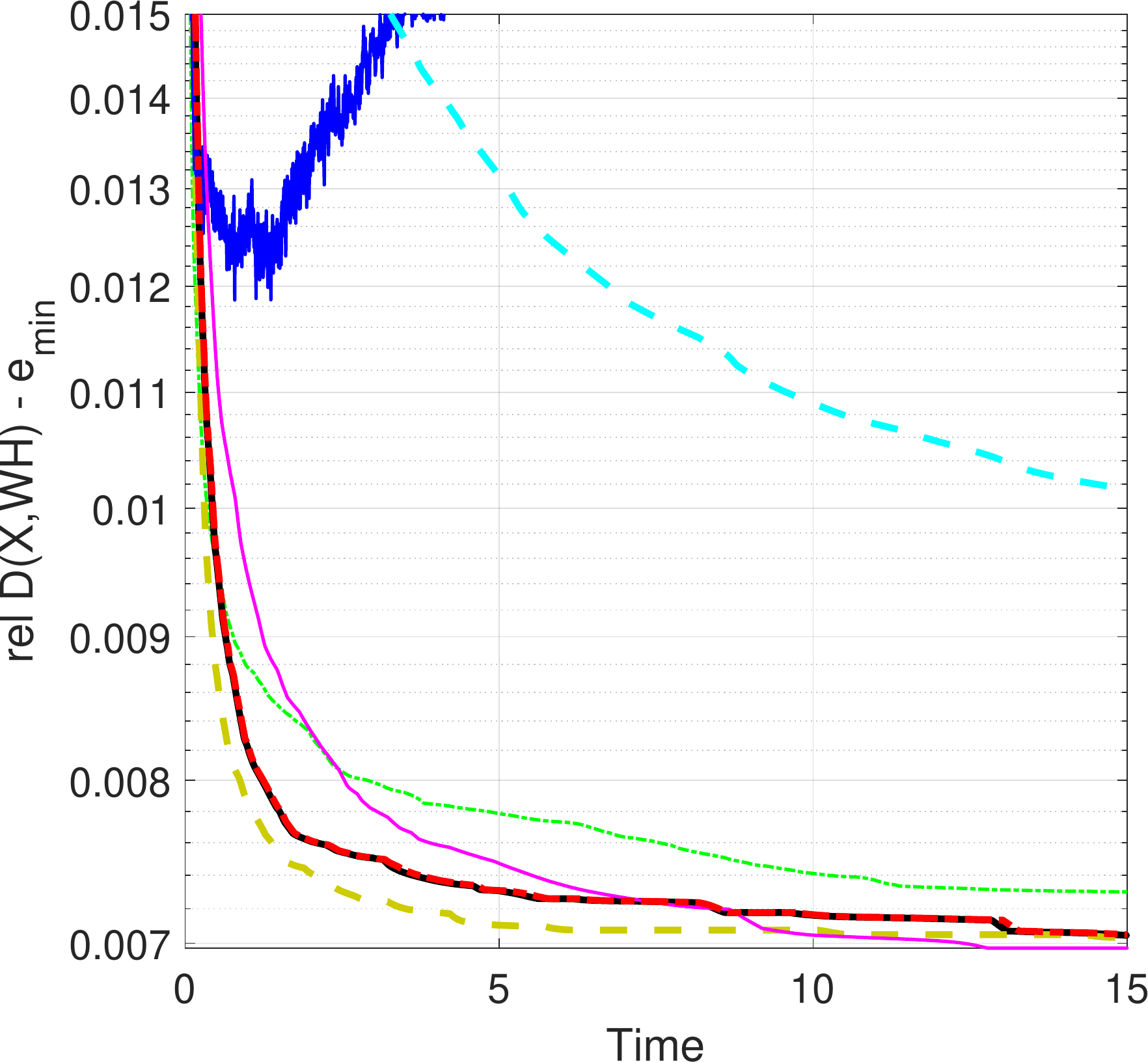} \\
\includegraphics[width=0.43\textwidth]{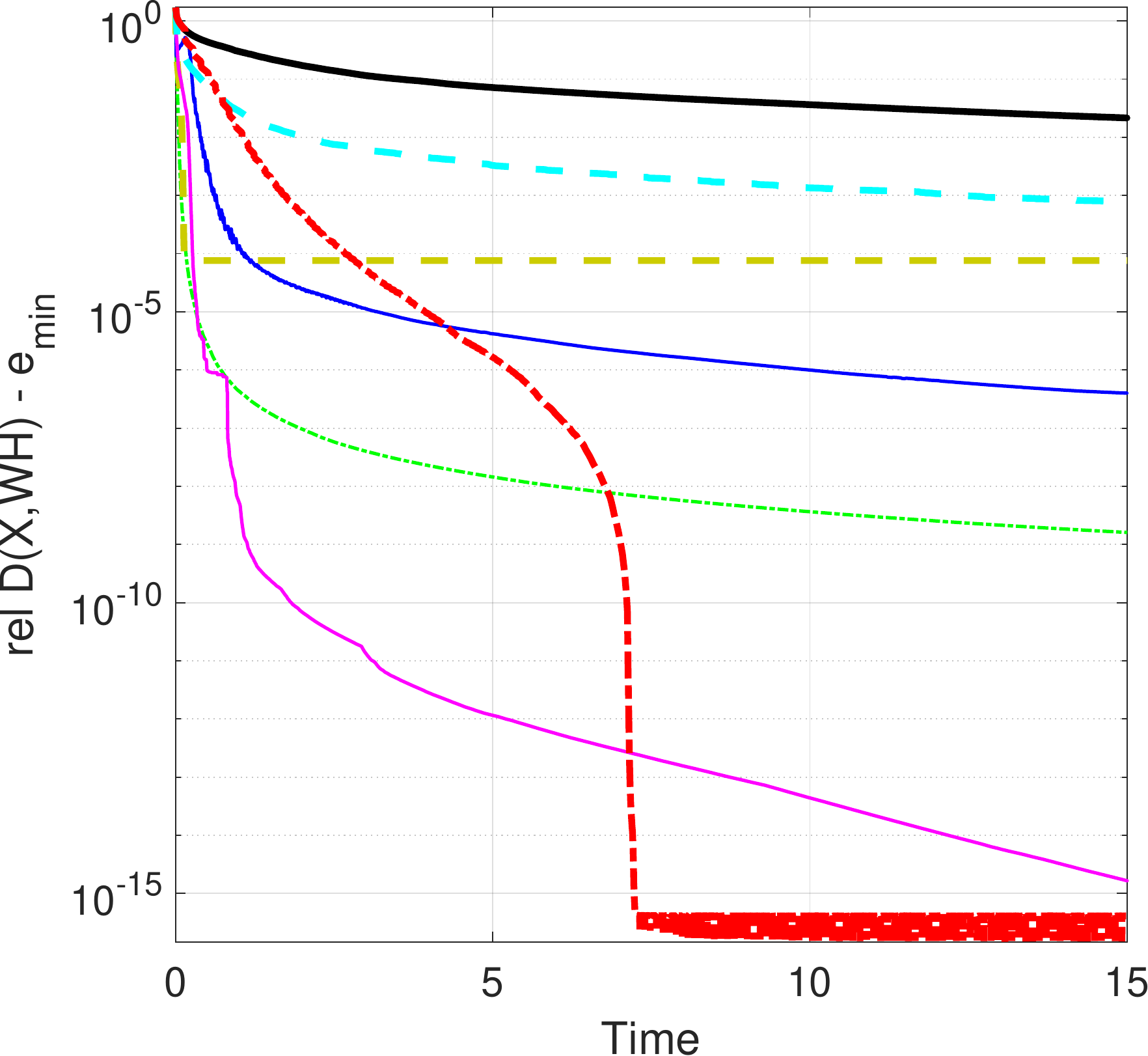} &
\includegraphics[width=0.43\textwidth]{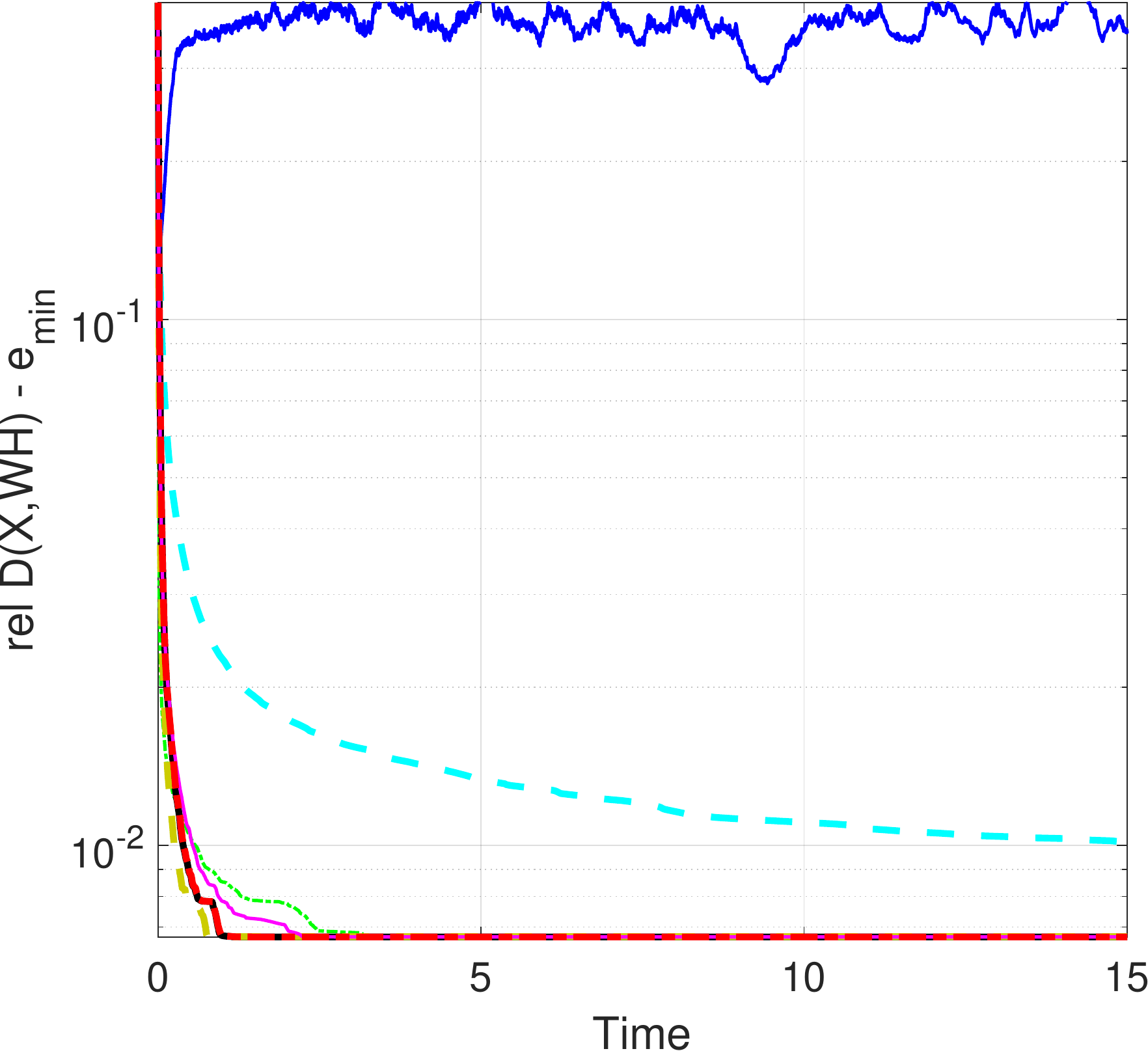} \\ 
\includegraphics[width=0.43\textwidth]{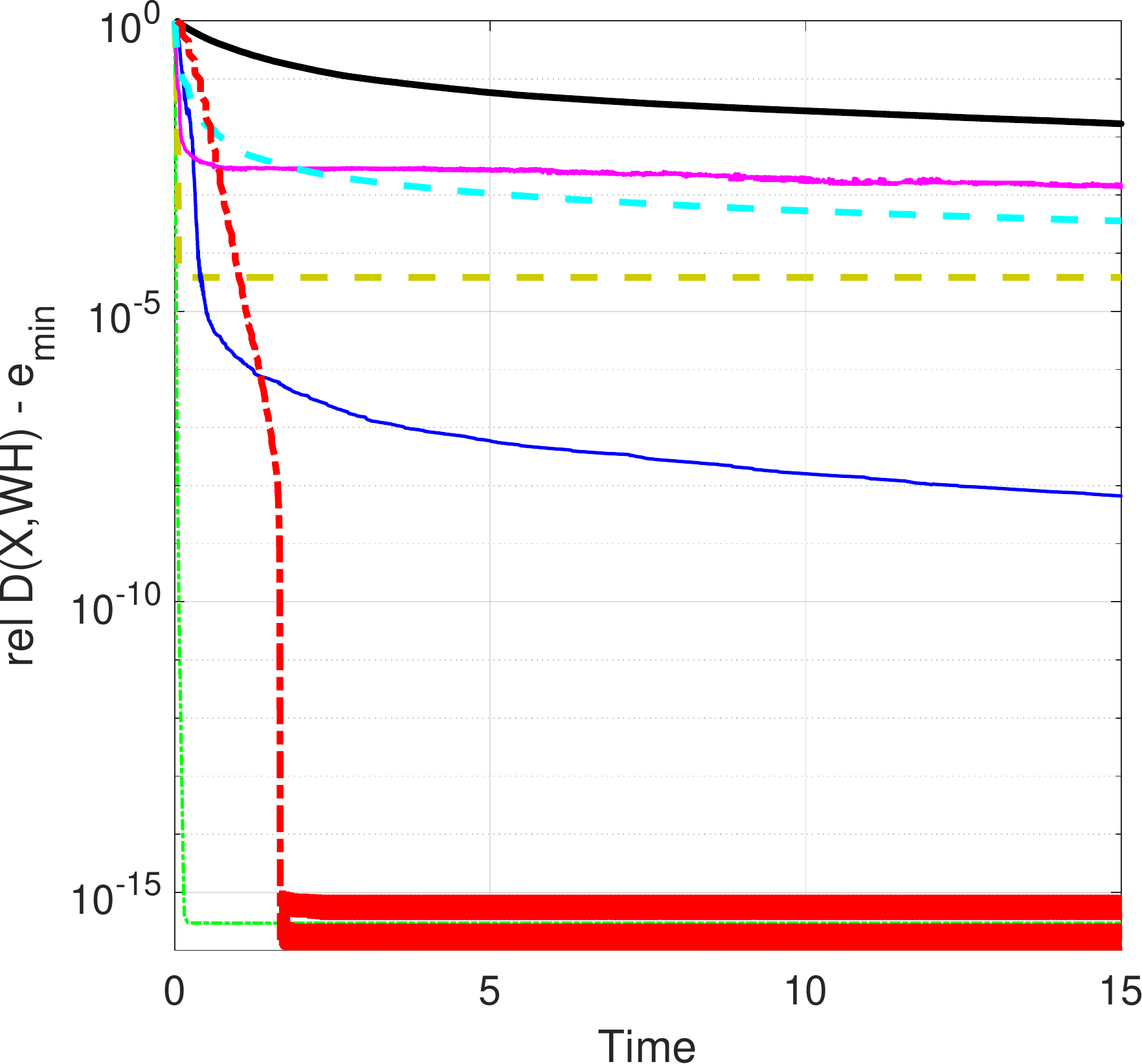} & 
\includegraphics[width=0.43\textwidth]{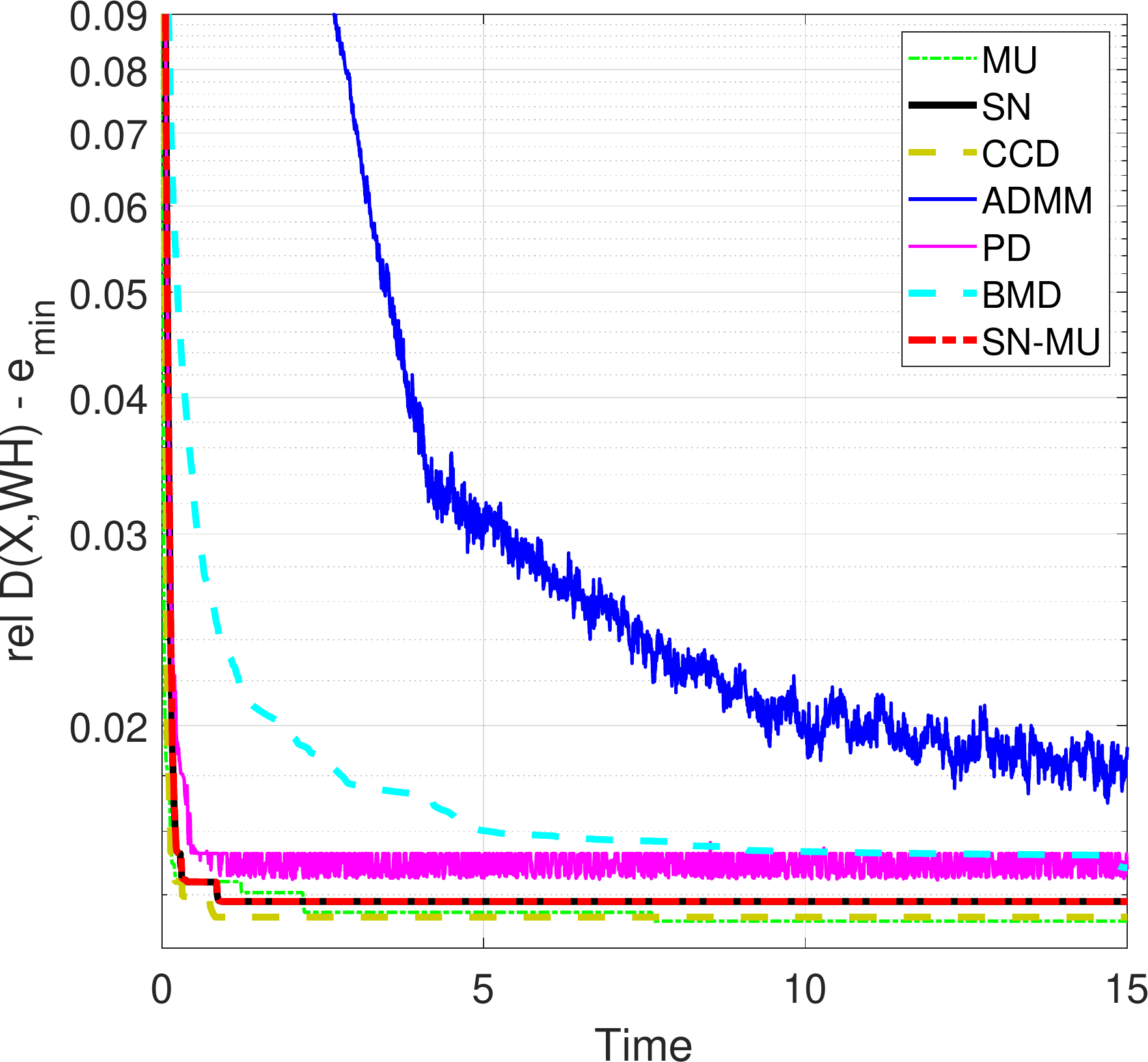}
\end{tabular}
\caption{\revise{Median value} of the error measure $E(t)$ on $200\times 200$ low-rank matrices for various KL NMF algorithms (see Table~\ref{tab:algorithms}).  
The left column corresponds to noiseless input matrices, the right column to noisy matrices using the Poisson distribution. 
The top row corresponds to dense factors $(\ell = 1)$, 
the middle row to slightly sparse factors $(\ell = 0.9)$, and the bottom row to sparse factors $(\ell = 0.3)$. 
%top left - no noise $\&$ no sparsity imposition,  top right - no noise  $\&$ 10\% sparsity imposition for $W$ and $H$, middle left - no noise  $\&$ 70\% sparsity imposition for $W$ and $H$, middle right - with Possion noise  $\&$ no sparsity imposition for $W$ and $H$, bottom left - with Possion noise  $\&$  10\% sparsity imposition for $W$ and $H$, bottom right - with Possion noise  $\&$
%70\% sparsity imposition for $W$ and $H$. 
\label{fig:lowrank_200x200}} 
\end{center}
\end{figure*}

\begin{figure*}[ht]
\begin{center}
\begin{tabular}{cc}
\includegraphics[width=0.43\textwidth]{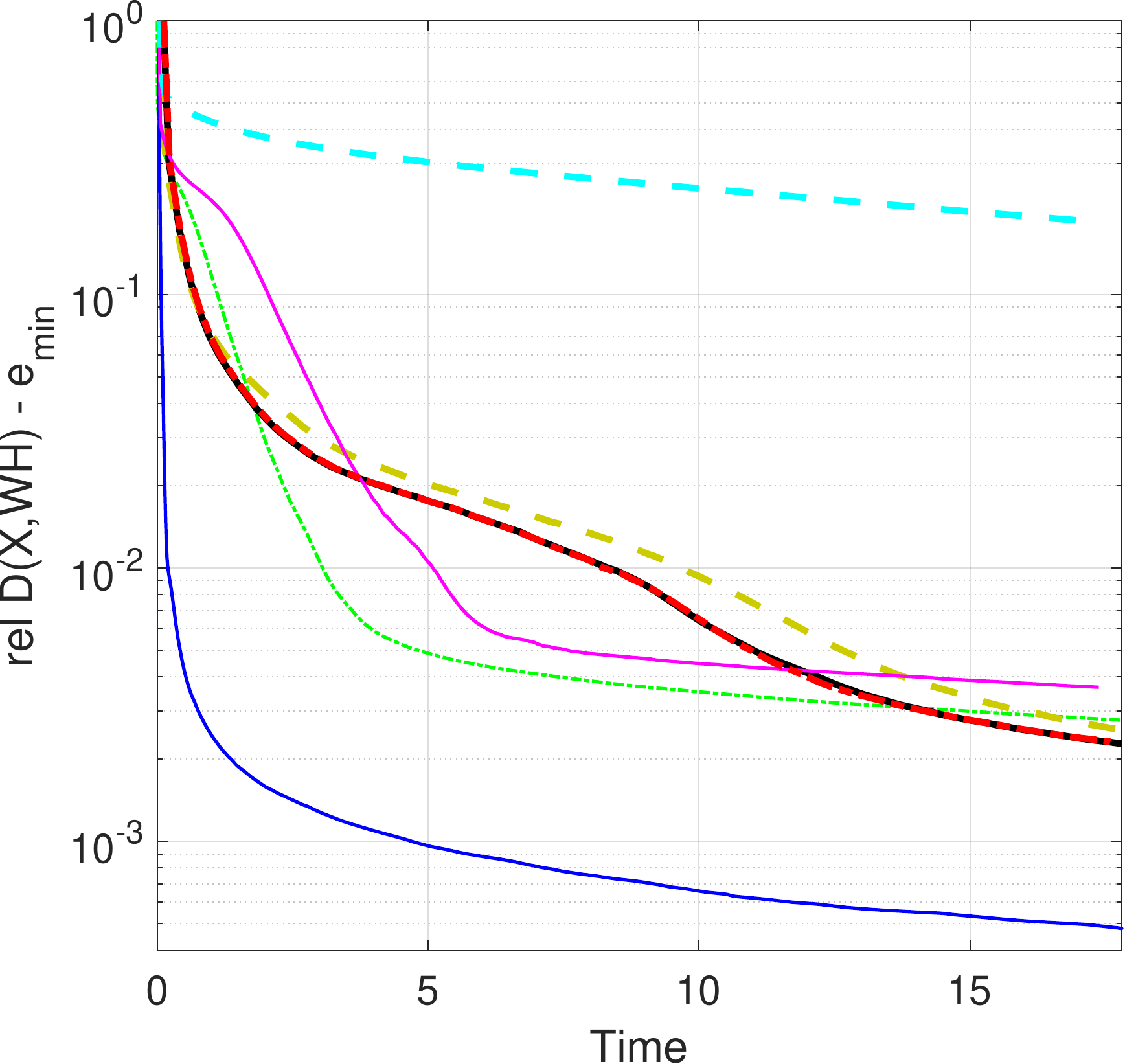}  
& \includegraphics[width=0.43\textwidth]{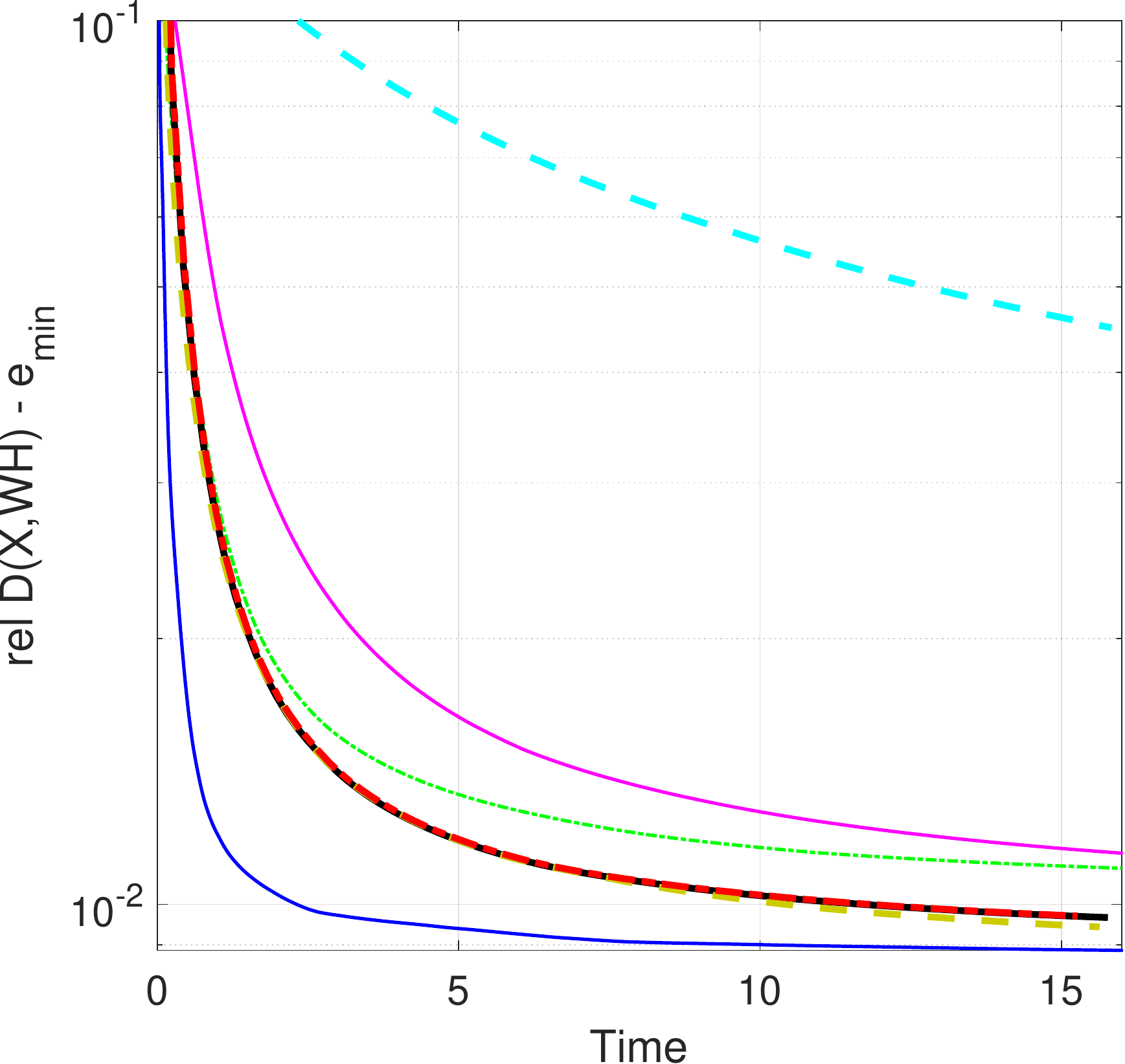} 
\\
\includegraphics[width=0.43\textwidth]{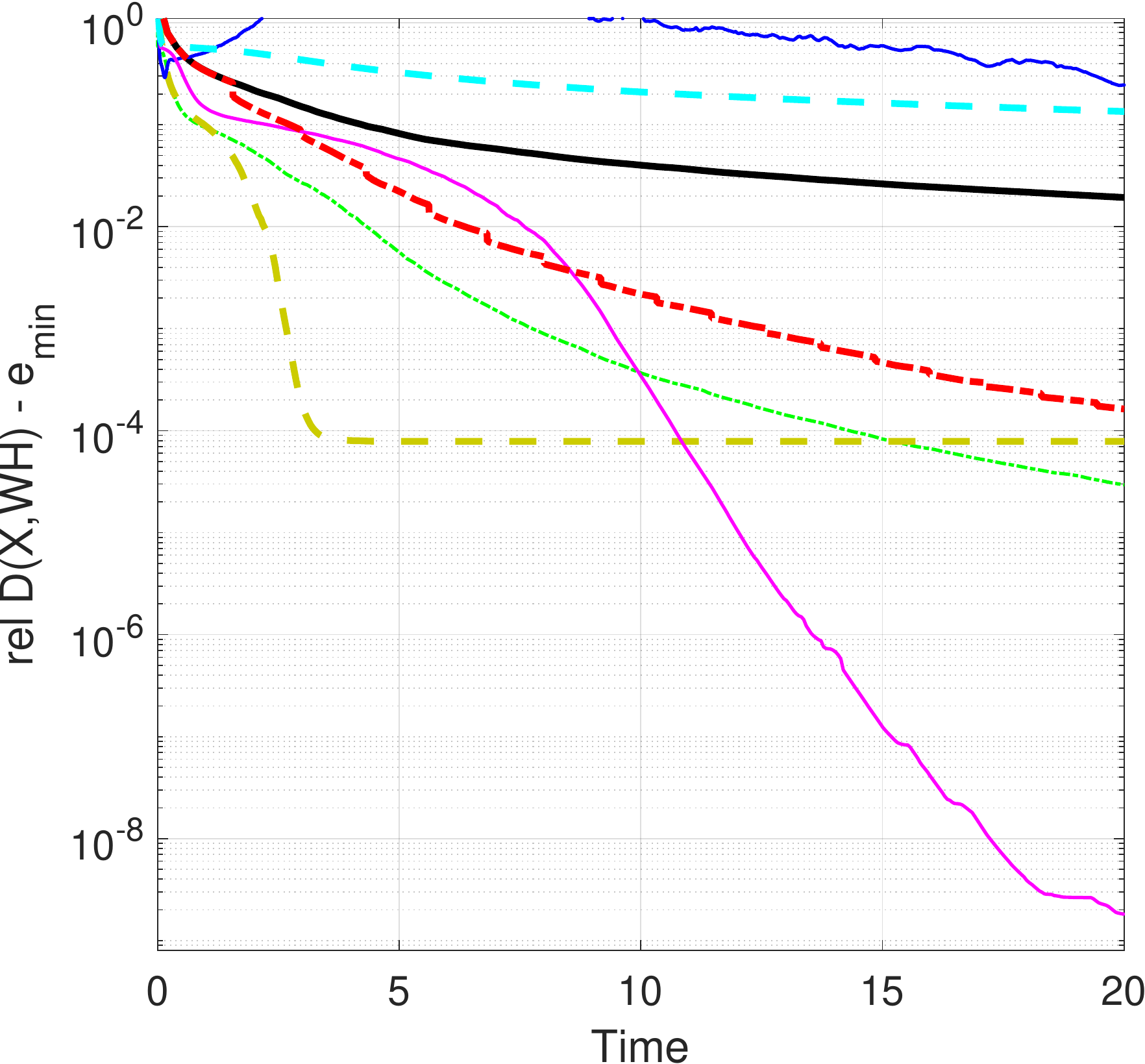} 
& 
\includegraphics[width=0.43\textwidth]{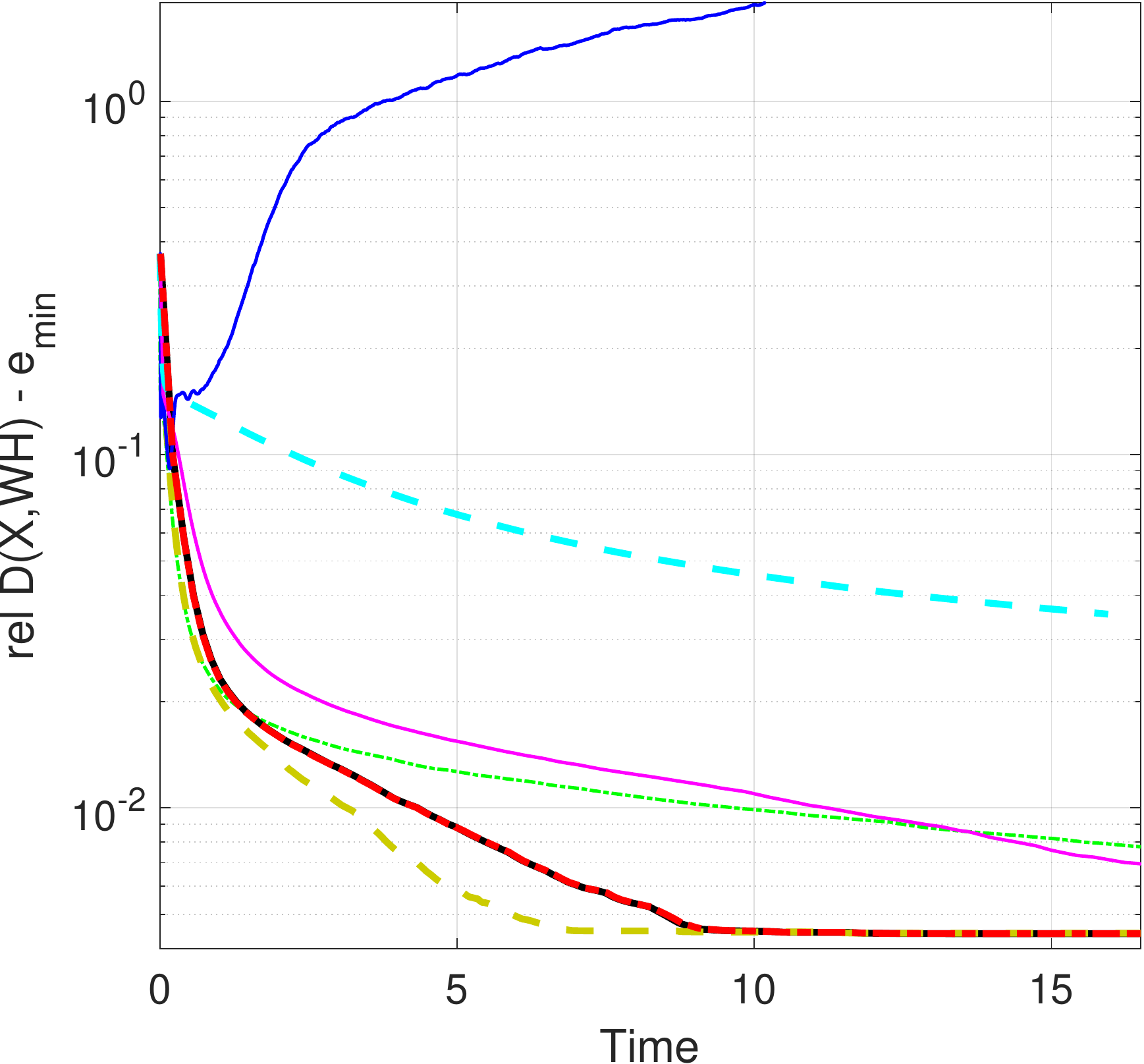} 
\\ 
\includegraphics[width=0.43\textwidth]{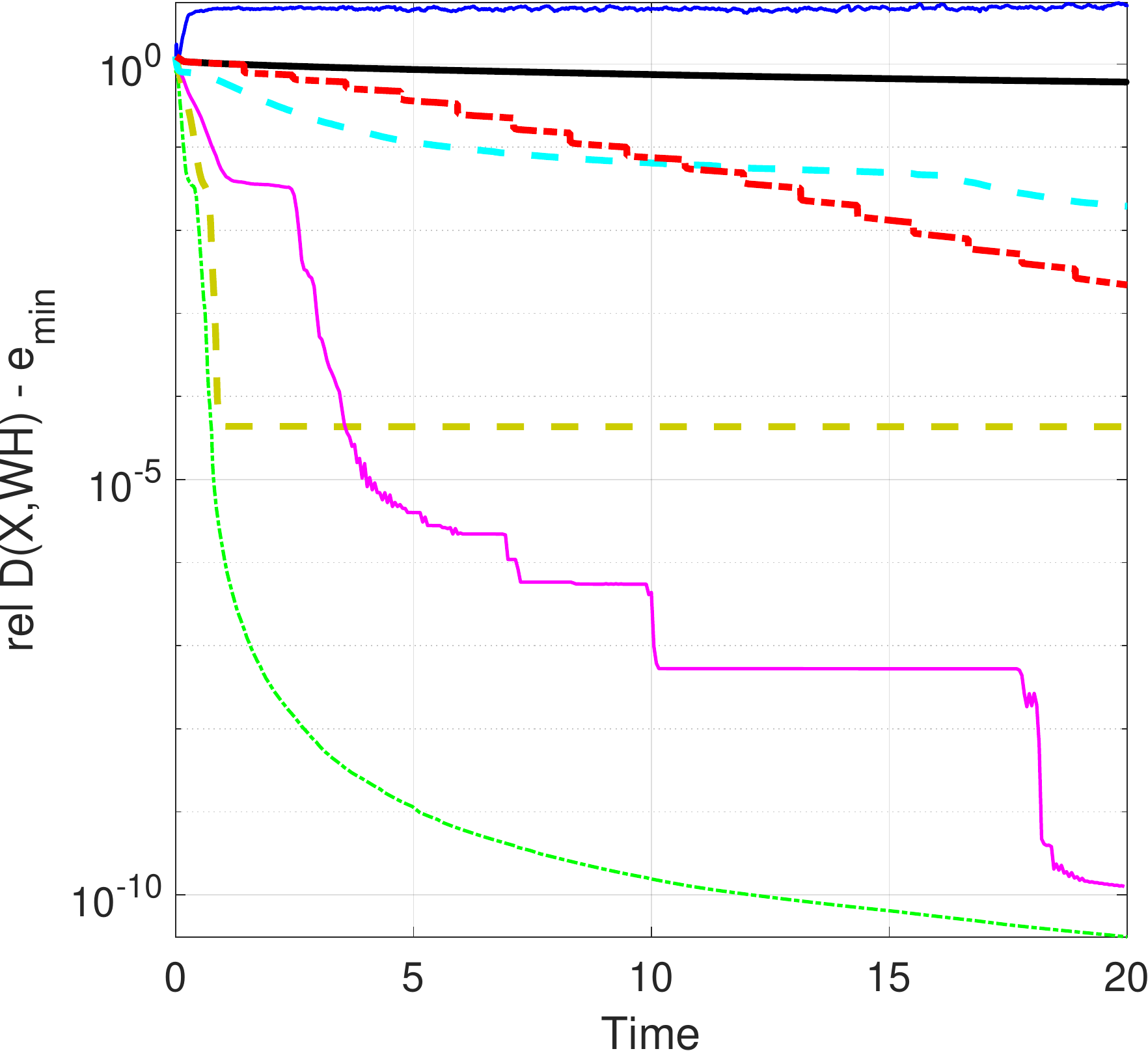} 
& 
\includegraphics[width=0.43\textwidth]{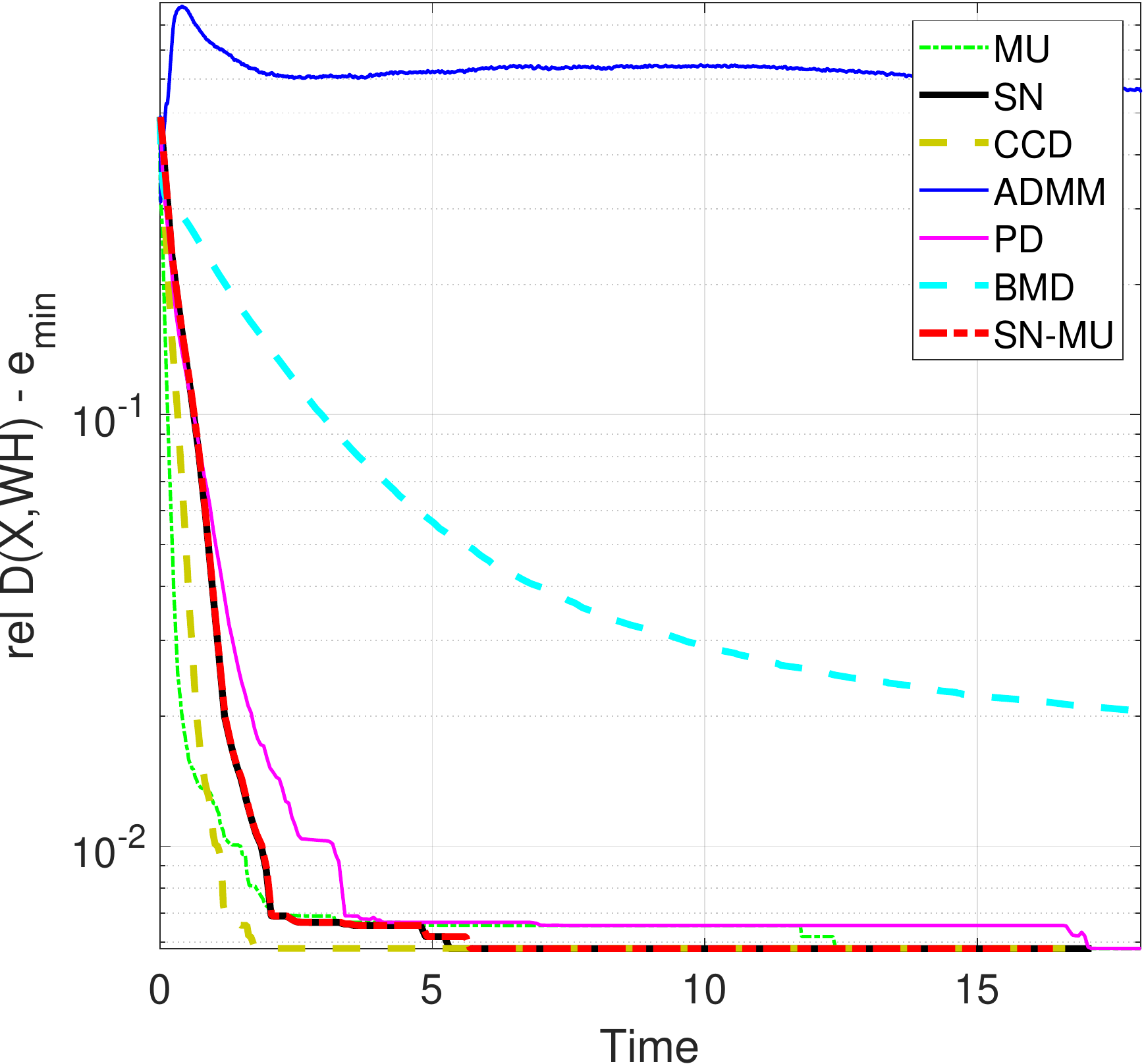}
\end{tabular}
\caption{\revise{Median value} of the error measure $E(t)$ on $500\times 500$ low-rank matrices for various KL NMF algorithms (see Table~\ref{tab:algorithms}).  
The left column corresponds to noiseless input matrices, and the right column to noisy matrices using the Poisson distribution. 
The top row corresponds to dense factors $(\ell = 1)$, 
the middle row to slightly sparse factors $(\ell = 0.9)$, and the bottom row to sparse factors $(\ell = 0.3)$.  %Average value of $E(t)$ on $500\times 500$ low-rank matrices: top left - no noise $\&$ no sparsity imposition,  top right - no noise  $\&$ 10\% sparsity imposition for $W$ and $H$, middle left - no noise  $\&$ 70\% sparsity imposition for $W$ and $H$, middle right - with Possion noise  $\&$ no sparsity imposition for $W$ and $H$, bottom left - with Possion noise  $\&$  10\% sparsity imposition for $W$ and $H$, bottom right - with Possion noise  $\&$
%70\% sparsity imposition for $W$ and $H$ . 
\label{fig:lowrank_500x500}} 
\end{center}
  \vspace{-0.13in}
\end{figure*}

 \begin{table}[t]\centering 
 \caption{\revise{Average error, standard deviation over 200 different runs for each type of low-rank synthetic data sets (100 runs for each type with the size $200\times 200$ or $500\times 500$). The middle column correspond to noiseless input matrices, the right column to noisy matrices using the Poisson distribution.  The second row corresponds to dense factors $(\ell = 1)$, 
the third row to slightly sparse factors $(\ell = 0.9)$, and the bottom row to sparse factors $(\ell = 0.3)$.}} 
   \label{tab:accuracy_lowrank}
 \begin{tabular}{|c|c|c|c|} 
 \hline 
 Algorithm &  mean $\pm$ std (noiseless matrices) &  mean $\pm$ std (noisy matrices) \\ 
 \hline 
ADMM &  $\mathbf{2.257\, 10^{-4} \pm 2.519\, 10^{-4}}$ &  $8.030\, 10^{-1} \pm 7.445\, 10^{-3}$   \\ 
CCD & $1.119\, 10^{-3} \pm 9.145\, 10^{-4}$ &   $\mathbf{8.025\, 10^{-1} \pm 7.966\, 10^{-3}}$   \\ 
SN & $9.432\, 10^{-4} \pm 9.106\, 10^{-4}$  &  $8.026\, 10^{-1} \pm 8.014\, 10^{-3}$  \\ 
 MU & $1.423\, 10^{-3} \pm 1.154\, 10^{-3}$ & $8.032\, 10^{-1} \pm 8.380\, 10^{-3}$  \\ 
 BMD & $7.962\, 10^{-2} \pm 7.325\, 10^{-2}$   &  $8.172\, 10^{-1} \pm 1.890\, 10^{-2}$   \\ 
 PD & $1.803\, 10^{-3} \pm 1.551\, 10^{-3}$  &   $8.033\, 10^{-1} \pm 8.527\, 10^{-3}$  \\ 
 SN-MU &  $9.133\, 10^{-4} \pm 8.780\, 10^{-4}$ &  $8.026\, 10^{-1} \pm 8.009\, 10^{-3}$   \\
\hline 
ADMM & $7.107\, 10^{-2} \pm 1.129\, 10^{-1}$ &   $8.497\, 10^{-1} \pm 1.149\, 10^{-2}$   \\ 
CCD & $7.807\, 10^{-5} \pm 3.642\, 10^{-6}$  &  $\mathbf{7.563\, 10^{-1} \pm 1.671\, 10^{-2}}$    \\ 
SN &  $2.013\, 10^{-2} \pm 1.335\, 10^{-2}$ & $7.564\, 10^{-1} \pm 1.661\, 10^{-2}$   \\ 
 MU & $8.795\, 10^{-6} \pm 1.088\, 10^{-5}$ &  $7.573\, 10^{-1} \pm 1.759\, 10^{-2}$ \\ 
 BMD &$5.950\, 10^{-2} \pm 5.917\, 10^{-2}$   &  $7.703\, 10^{-1} \pm 2.715\, 10^{-2}$   \\ 
 PD &  $\mathbf{6.156\, 10^{-9} \pm 1.917\, 10^{-8}}$ &  $7.567\, 10^{-1} \pm 1.702\, 10^{-2}$    \\ 
 SN-MU &   $4.515\, 10^{-5} \pm 6.725\, 10^{-5}$ &   $7.564\, 10^{-1} \pm 1.661\, 10^{-2}$  \\
  \hline 
ADMM &  $5.493\, 10^{-1} \pm 5.671\, 10^{-1}$  &  $7.292\, 10^{-1} \pm 1.869\, 10^{-1}$   \\ 
CCD &$1.340\, 10^{-3} \pm 9.142\, 10^{-3}$  &  $5.762\, 10^{-1} \pm 3.462\, 10^{-2}$   \\ 
SN &  $2.911\, 10^{-1} \pm 2.790\, 10^{-1}$  & $\mathbf{5.761\, 10^{-1} \pm 3.468\, 10^{-2}}$   \\ 
 MU &$\mathbf{3.525\, 10^{-4} \pm 4.985\, 10^{-3}}$  & $5.759\, 10^{-1} \pm 3.502\, 10^{-2}$  \\ 
 BMD &   $1.592\, 10^{-2} \pm 2.041\, 10^{-2}$ &   $5.830\, 10^{-1} \pm 3.924\, 10^{-2}$  \\ 
 PD &$9.711\, 10^{-4} \pm 6.808\, 10^{-3}$  &  $5.769\, 10^{-1} \pm 3.419\, 10^{-2}$   \\ 
 SN-MU & $1.468\, 10^{-3} \pm 6.274\, 10^{-3}$  &  $\mathbf{5.761\, 10^{-1} \pm 3.468\, 10^{-2}}$  \\
\hline 
\end{tabular} 
  \vspace{-0.1in}
 \end{table} 
 
  \begin{table}[t]\centering 
 \caption{Ranking among 1200 different runs for low-rank synthetic data sets. } 
   \label{tab:rank_lowrank}
 \begin{tabular}{|c|c|} 
 \hline 
 Algorithm &  ranking  \\ 
 \hline 
ADMM &   (301,  6, 97, 121, 55, 117, 503)   \\ 
CCD &   (52, 131, 302, 344, 328, 32, 11)   \\ 
SN &  (148, 273, 249, 89, 109, 147, 185)   \\ 
 MU &  (140, 243, 172, 198, 369, 71,  7)   \\ 
 BMD &    (1,  0,  2, 26, 184, 522, 465)   \\ 
 PD &  (222, 217, 161, 164, 122, 293, 21)   \\ 
 SN-MU &    (\textbf{343}, 328, 214, 257, 32, 18,  8)   \\
\hline 
\end{tabular} 
  \vspace{-0.1in}
 \end{table} 
 For these low-rank synthetic data sets, let us discuss the behaviour of the various algorithms:  
\begin{itemize}

\item ADMM is not stable, it diverges in many cases. 
Although the results are for ADMM with the penalty parameter $\varrho=1$, we also tried other values for $\varrho$ but the algorithm still did not converge for the other values we have tried.  
ADMM has the largest number of worst solutions (503 out of 1200). Note however that it also has a high number of best solutions (301 out of 1200), because it performs well in the simple scenario when the factors $W$ and $H$ are dense in the absence of noise.  
In summary, ADMM is unstable but, when it converges, it provides good solutions.   

\item %CCD encounters numerical error in some cases. For example, in the middle right image of Figure~\ref{fig:lowrank_200x200}, CCD terminates earlier than the other; this is because CCD produces $NaN$ values for the relative error since it sometimes generates negative entries in $W$ or $H$. However, except for these cases, 
CCD performs very well, among the best in most cases. 
When looking at Table~\ref{tab:accuracy_lowrank}, we observe that CCD has average results, having most of its solutions ranked third to fifth (out of 7). However, it has the second lowest average error right after SN-MU. 

\item SN monotonically decreases the objective function, as proved in Proposition~\ref{prop:SN_nonincreasing}. However, in some cases, it may converge rather slowly. Its ranking are well distributed, hence it performs close to the average.  

\item SN-MU improves SN and performs well, in all cases among the best algorithms. In fact, it obtained the lowest objective function values among all algorithm, 343 out of the 1200 experiments. Also, it generates only 8 out of 1200 solutions as the worst solutions.  
Also, it has the lowest relative error on average.  
%Hence this hybridization significantly improves upon SN. 

\item MU performs well, on average better than the other algorithms (it only provides 7 worst solutions, out of 1200 tests). It has a low relative error on average, ranked third, right behind SN-MU and CDD.

\item BMD converges very slowly (it only provides a solution among the third best ones in 3 cases out of 1200). 
Although it is the only algorithm with global convergence guarantee (Proposition~\ref{thm:KLNMFconvergence}), this comes at the expense of slow convergence. 

\item PD performs well, although it provides in many cases (293 out of 1200) the second worst solutions. 

\end{itemize}

In summary, SN-MU performs on average the best, followed by 
the MU, CCD, PD, and SN. 
ADMM does not always converge but can produce good solutions. 
BMD has strong convergence guarantees but converges very slowly. 
However, there is no clear winner, and, depending on the types of data sets, some algorithms might perform better than others. 

\subsubsection{Full-rank synthetic data sets}
 
We generate a full-rank synthetic data set $V$ by the Matlab command $V=rand(m,n)$. Results for full-rank synthetic $200 \times 200$ \revise{(with $r=10$)} and $500 \times 500$ \revise{(with $r=20$)} data sets  are reported in Figure~\ref{fig:fullrank}. We also report the average, the standard deviation of the relative errors and the ranking vector over 200 runs (100 runs for each size) in Table~\ref{tab:accuracy_fullrank}.

\begin{figure*}[ht]
\begin{center}
\begin{tabular}{cc}
\includegraphics[width=0.43\textwidth]{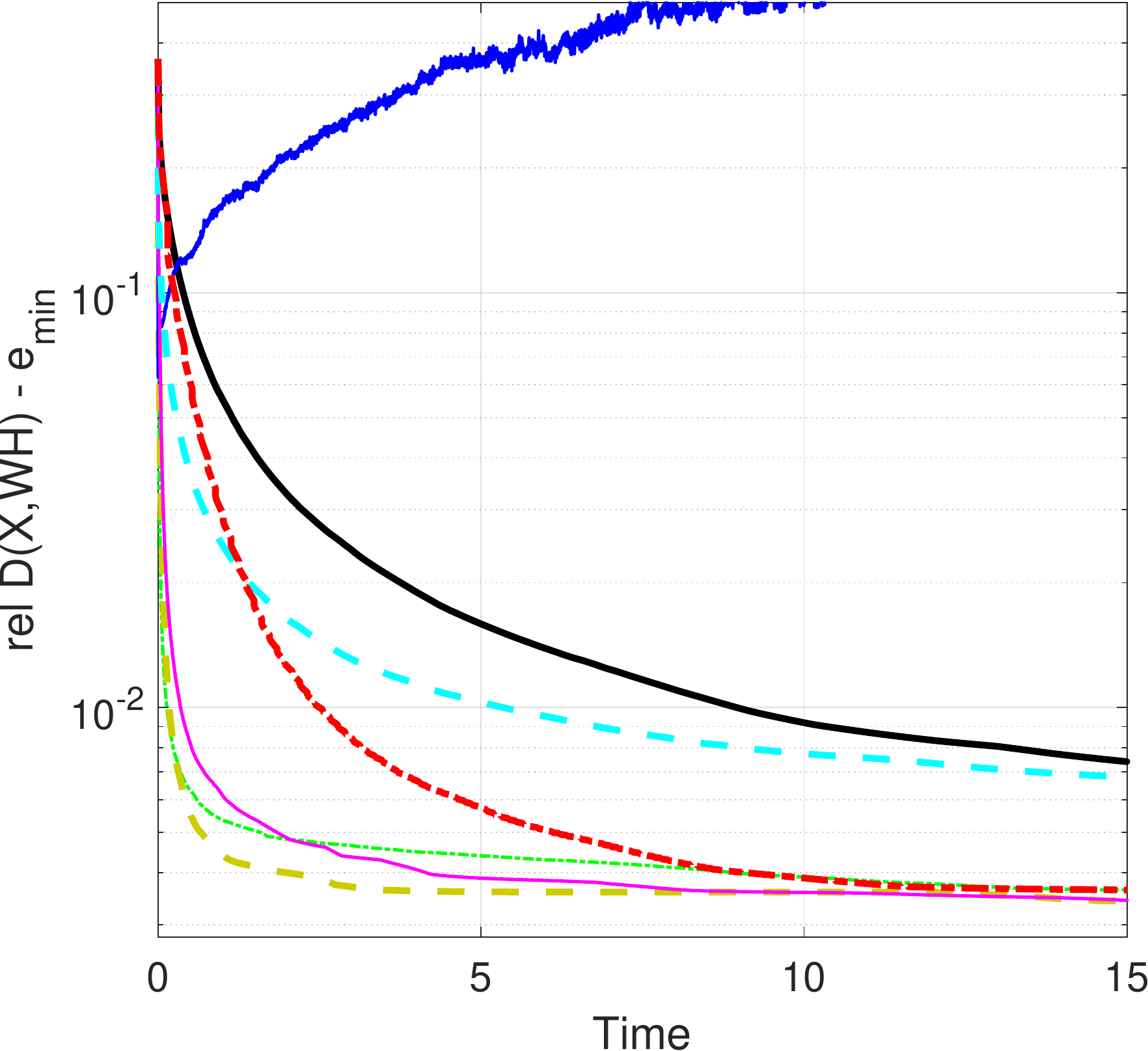}  & 
\includegraphics[width=0.43\textwidth]{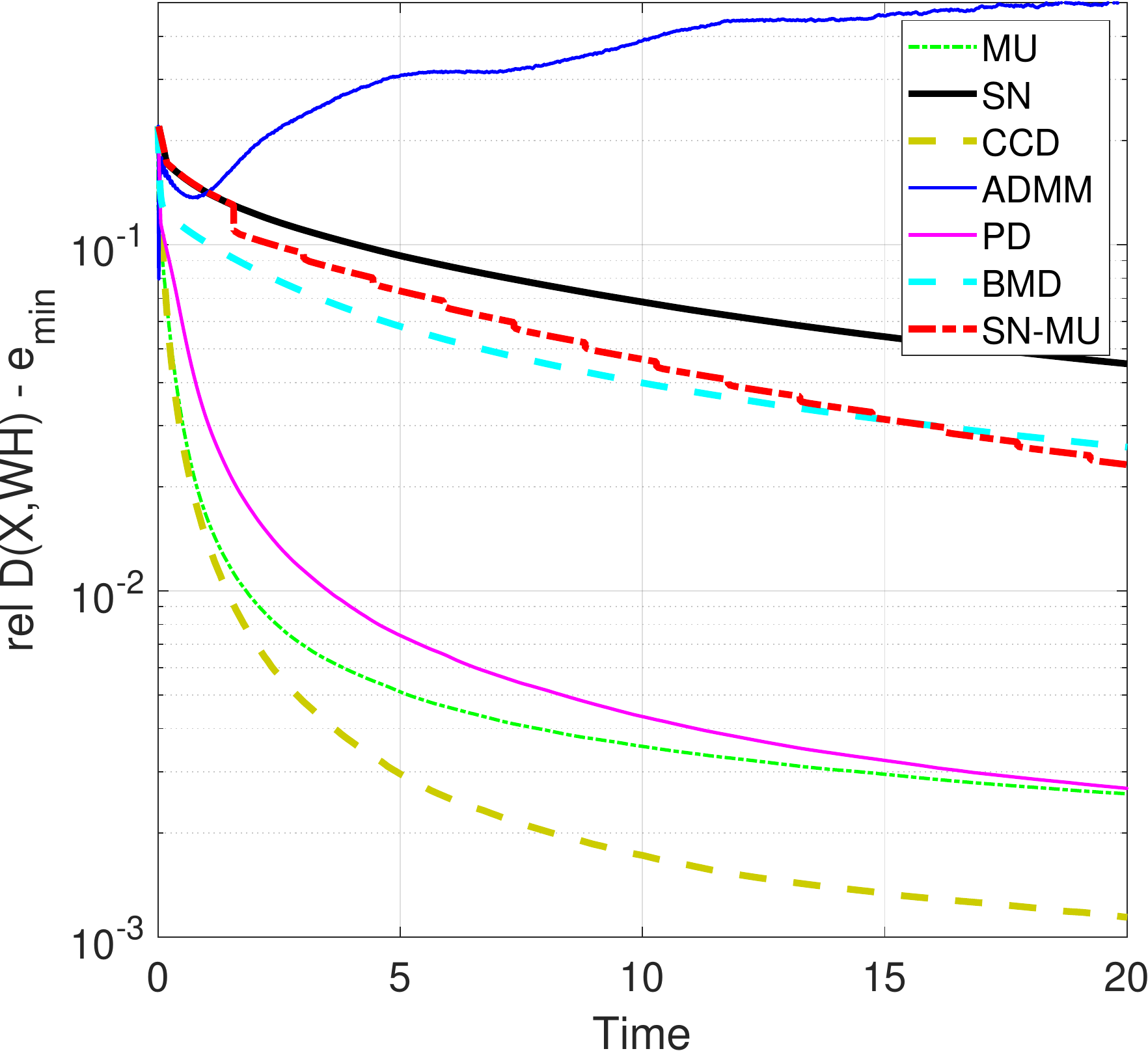} 
\end{tabular}
\caption{\revise{Median value} of $E(t)$ on Full-rank synthetic data sets:  left - $200 \times 200$,  right - $500 \times 500$.
\label{fig:fullrank}} 
\end{center}
  \vspace{-0.13in}
\end{figure*}

 \begin{table}[t]\centering 
 \caption{Average error, standard deviation and ranking among 200 different runs for full-rank synthetic data sets. } 
   \label{tab:accuracy_fullrank}
 \begin{tabular}{|c|c|c|} 
 \hline 
 Algorithm &  mean $\pm$ std & ranking  \\ 
 \hline 
ADMM &  $9.369\, 10^{-1} \pm 2.774\, 10^{-2}$ &  (0,  0,  0,  0,  0,  0, 200)   \\ 
CCD &  $\mathbf{8.709\, 10^{-1} \pm 1.418\, 10^{-2}}$ &  (\textbf{138}, 27, 23, 12,  0,  0,  0)   \\ 
SN &  $8.923\, 10^{-1} \pm 3.218\, 10^{-2}$ &  (0,  0,  0,  0, 47, 153,  0)   \\ 
 MU &  $8.717\, 10^{-1} \pm 1.480\, 10^{-2}$ &  (12, 62, 91, 35,  0,  0,  0)   \\ 
 BMD &  $8.835\, 10^{-1} \pm 2.370\, 10^{-2}$ &  (0,  0,  0,  0, 153, 47,  0)   \\ 
 PD &  $8.716\, 10^{-1} \pm 1.486\, 10^{-2}$ &  (37, 93, 58, 12,  0,  0,  0)   \\ 
 SN-MU &  $8.793\, 10^{-1} \pm 2.241\, 10^{-2}$ &  (13, 18, 28, 141,  0,  0,  0)   \\ 
\hline 
\end{tabular} 
  \vspace{-0.1in}
 \end{table} 
 
 We observe that the behavior can be quite different than in the low-rank cases. In particular, 
\begin{itemize}
\item CCD now clearly performs best in term of convergence speed and average relative error. 

\item The MU, CCD and SN-MU performs well, while BMD and SN perform  relatively poorly (they never produce the best solution). 
 
\item ADMM never converges, and produces the worst solution in all cases. 

\end{itemize}

\vspace{-0.15in}
\subsection{Experiments with real data sets}

We report in this section experiments on various widely used real data sets that are summarized in Table~\ref{tab:audio}. \revise{We use $r=10$ for all real data sets.} 
\begin{table}[t]\centering 
 \caption{Real data sets} 
   \label{tab:audio}
 \begin{tabular}{|c|c|c|} 
 \hline 
 Data set &  size & run time (seconds) \\ 
 \hline 
 \multicolumn{3}{|c|}{Audio~\cite{gillis2019distributionally}}\\
  \hline 
mary&     $129 \times 586$ & 25  \\ 
prelude JSB & $129 \times 2 \,582$  & 45   \\ 
ShanHur sunrise & $129 \times 4 \,102$    & 75 \\ 
voice cell & $129 \times 2 \, 181$  &   45\\ 
\hline 
 \multicolumn{3}{|c|}{Images~\cite{Lee99,hoyer2004non}} \\
 \hline
 cbclim  &$361 \times 2\,429$ & 50 \\
  ORLfaces&  $10\, 304 \times 400$ & 95 \\
%  Umistim  &$10304 \times 565$ & 95 \\
   \hline 
 \multicolumn{3}{|c|}{Documents~\cite{zhong2005generative}} \\
 \hline
  classic&  $7\, 094 \times 41\, 681$ & 500 \\
hitech&  $2 \,301 \times 10 \,080$ & 500 \\
reviews&  $4\, 069 \times 18\, 483$ & 500 \\
sports&  $8 \,580 \times 14\, 870$ & 500 \\
ohscal&  $11\, 162 \times 11 \,465$ & 500 \\
la1&  $3\, 204 \times 31\, 472$ & 500 \\
\hline
\end{tabular} 
  \vspace{-0.1in}
 \end{table} 
 
 \vspace{-0.1in}
\subsubsection{Audio data sets} 

For each audio data set, we generate 30 random initial points. We report the evolution of the  median of $E(t)$ in Figure~\ref{fig:audio}, and report the average, the standard deviation of the relative errors and the ranking vector over 120 runs (30 runs for each audio data set) in Table~\ref{tab:accuracy_audio}. 

\begin{figure*}[ht]
\begin{center}
\begin{tabular}{cc}
\includegraphics[width=0.43\textwidth]{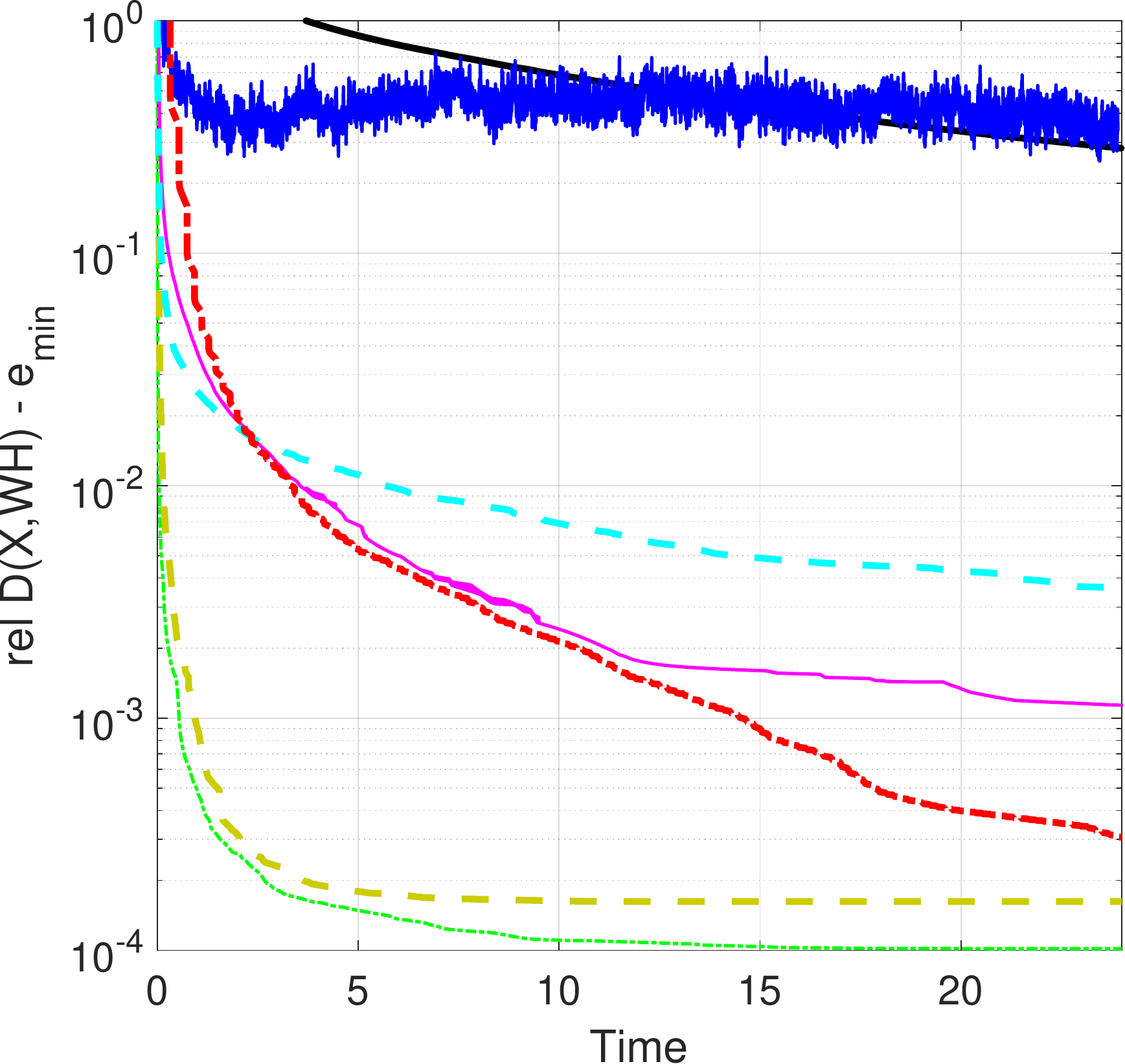}  & 
\includegraphics[width=0.43\textwidth]{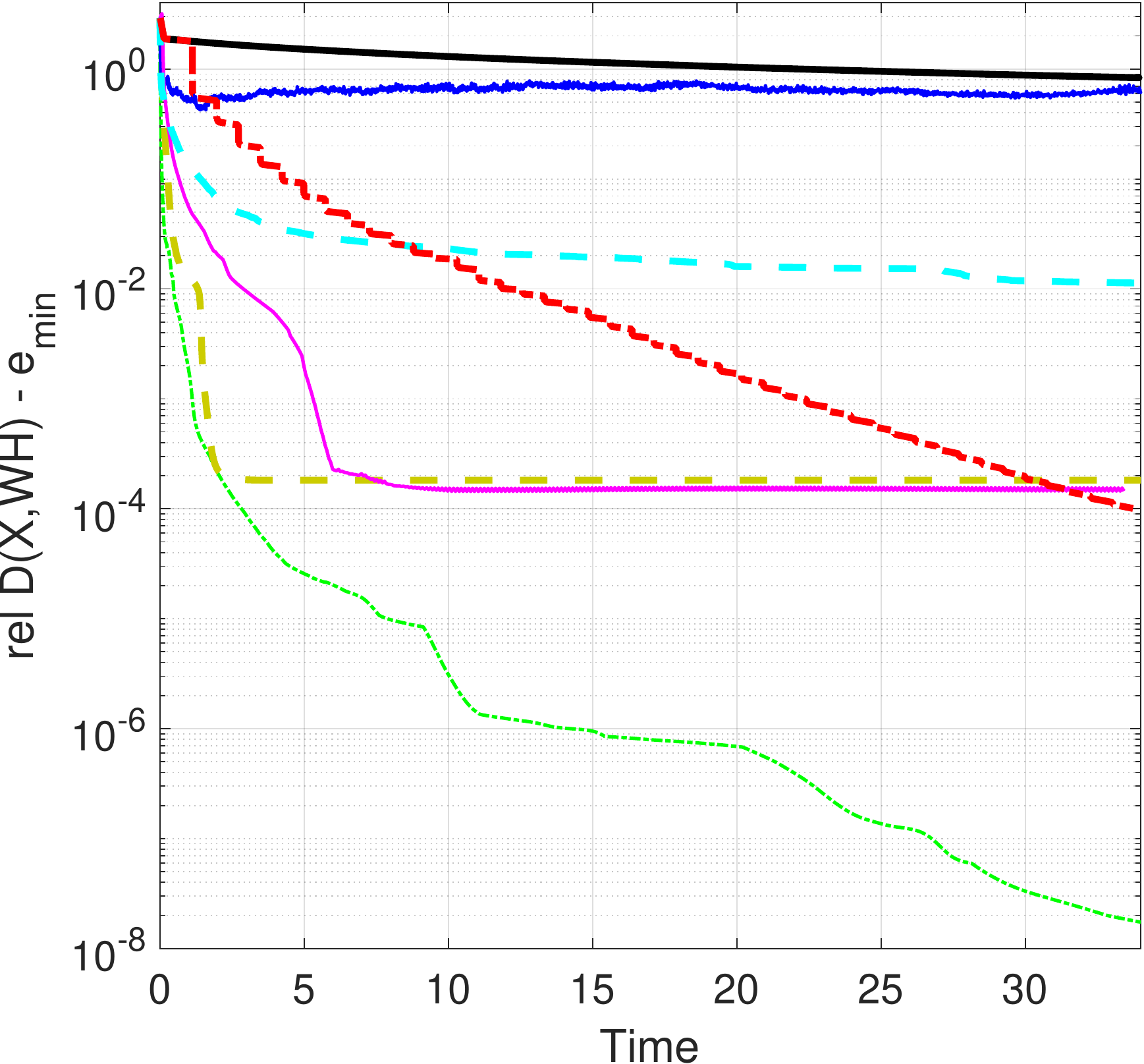} \\
\includegraphics[width=0.43\textwidth]{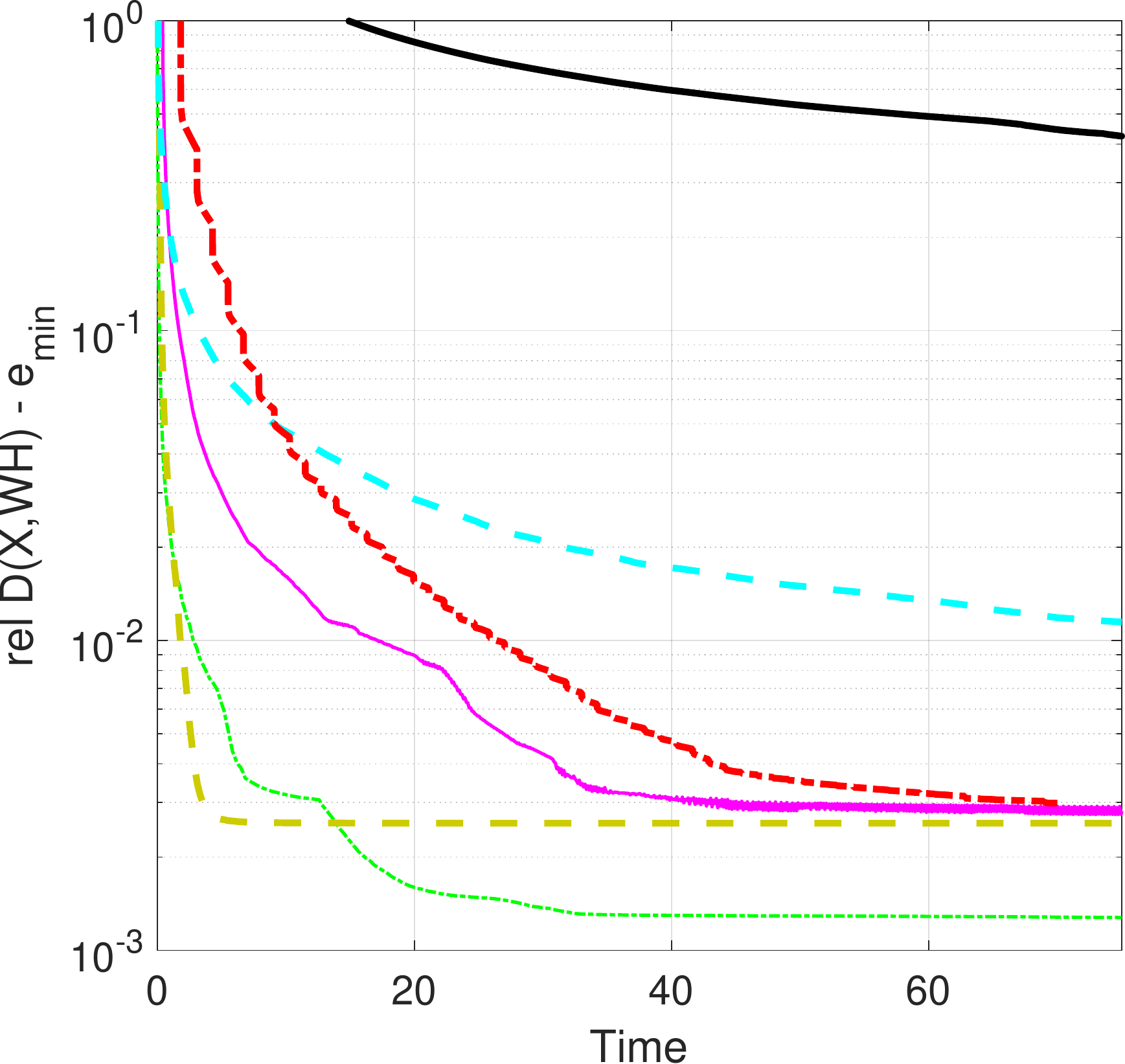}  & 
\includegraphics[width=0.43\textwidth]{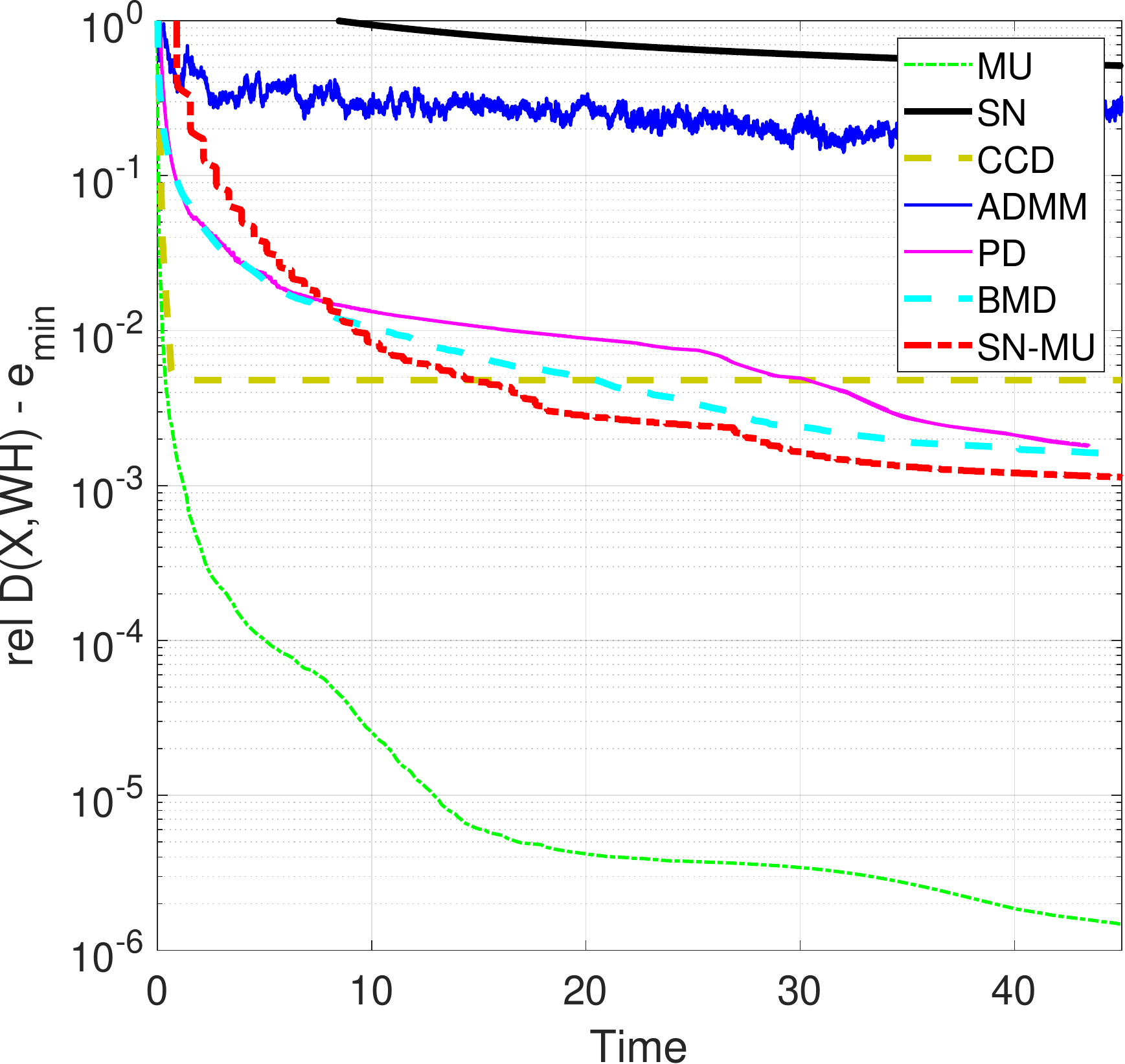} 
\end{tabular}
\caption{Median value of $E(t)$ on audio data sets:  top left - mary,  top right - prelude JSB, bottom left - ShanHur sunrise, bottom right - voice cell.
\label{fig:audio}} 
\end{center}
\vspace{-0.13in}
\end{figure*} 

\begin{table}[t]\centering 
 \caption{Average error, standard deviation and ranking among 120 different runs for audio data sets.  \label{tab:accuracy_audio}  } 
 \begin{tabular}{|c|c|c|} 
 \hline 
 Algorithm &  mean $\pm$ std & ranking  \\ 
 \hline 
ADMM &  $2.533\,10^{-1} \pm 1.916\,10^{-1}$ &  (0,  0,  0,  0,  1, 110,  9)   \\ 
CCD &  $7.588\,10^{-2} \pm 4.853\,10^{-2}$ &  (5, 42, 33, 23, 17,  0,  0)   \\ 
SN &  $5.213\,10^{-1} \pm 2.002\,10^{-1}$ &  (0,  0,  0,  0,  0,  9, 111)   \\ 
 MU &  $\mathbf{7.468\,10^{-2} \pm 4.849\,10^{-2}}$ &  (\textbf{89},  7, 13,  7,  4,  0,  0)   \\ 
 BMD &  $7.898\,10^{-2} \pm 4.946\,10^{-2}$ &  (0, 12, 21, 25, 61,  1,  0)   \\ 
 PD &  $7.642\,10^{-2} \pm 4.841\,10^{-2}$ &  (9, 20, 23, 46, 22,  0,  0)   \\ 
 SN-MU &  $7.498\,10^{-2} \pm 4.815\,10^{-2}$ &  (17, 39, 30, 19, 15,  0,  0)   \\ 
\hline 
\end{tabular}
\vspace{-0.13in} 
 \end{table} 
 
 As for full-rank synthethic data sets, ADMM diverges while SN and BMD converges slowly. 
 However, for these data sets, MU outperforms the other algorithms, followed by  CCD, SN-MU, and PD.

\vspace{-0.15in}
\subsubsection{Image data sets} 

As for audio data sets, we generate 30 random initial points. We report the result in Figure~\ref{fig:image} and Table~\ref{tab:accuracy_image}. 

\begin{figure*}[ht]
\begin{center}
\begin{tabular}{cc}
\includegraphics[width=0.43\textwidth]{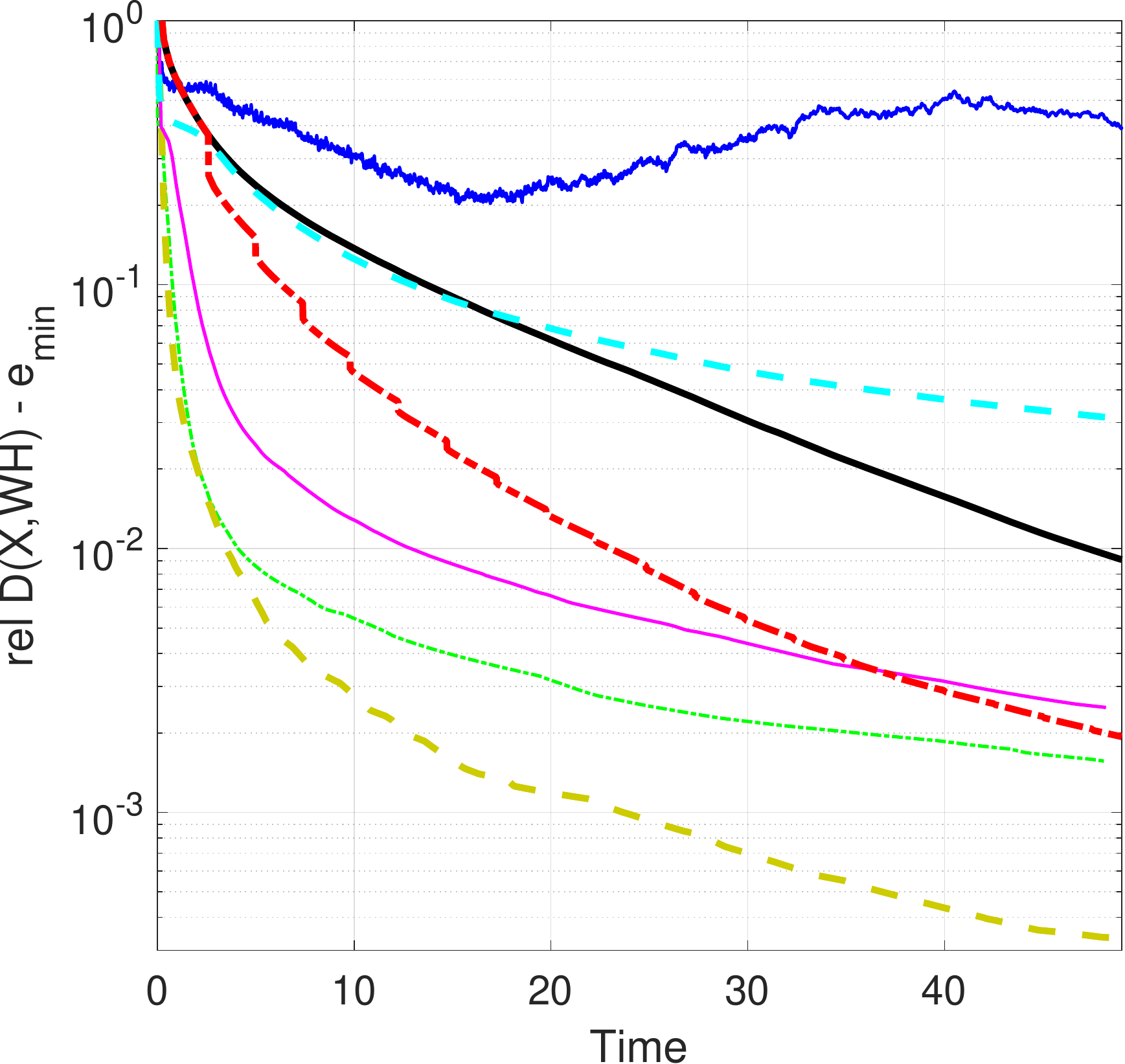}  & 
\includegraphics[width=0.43\textwidth]{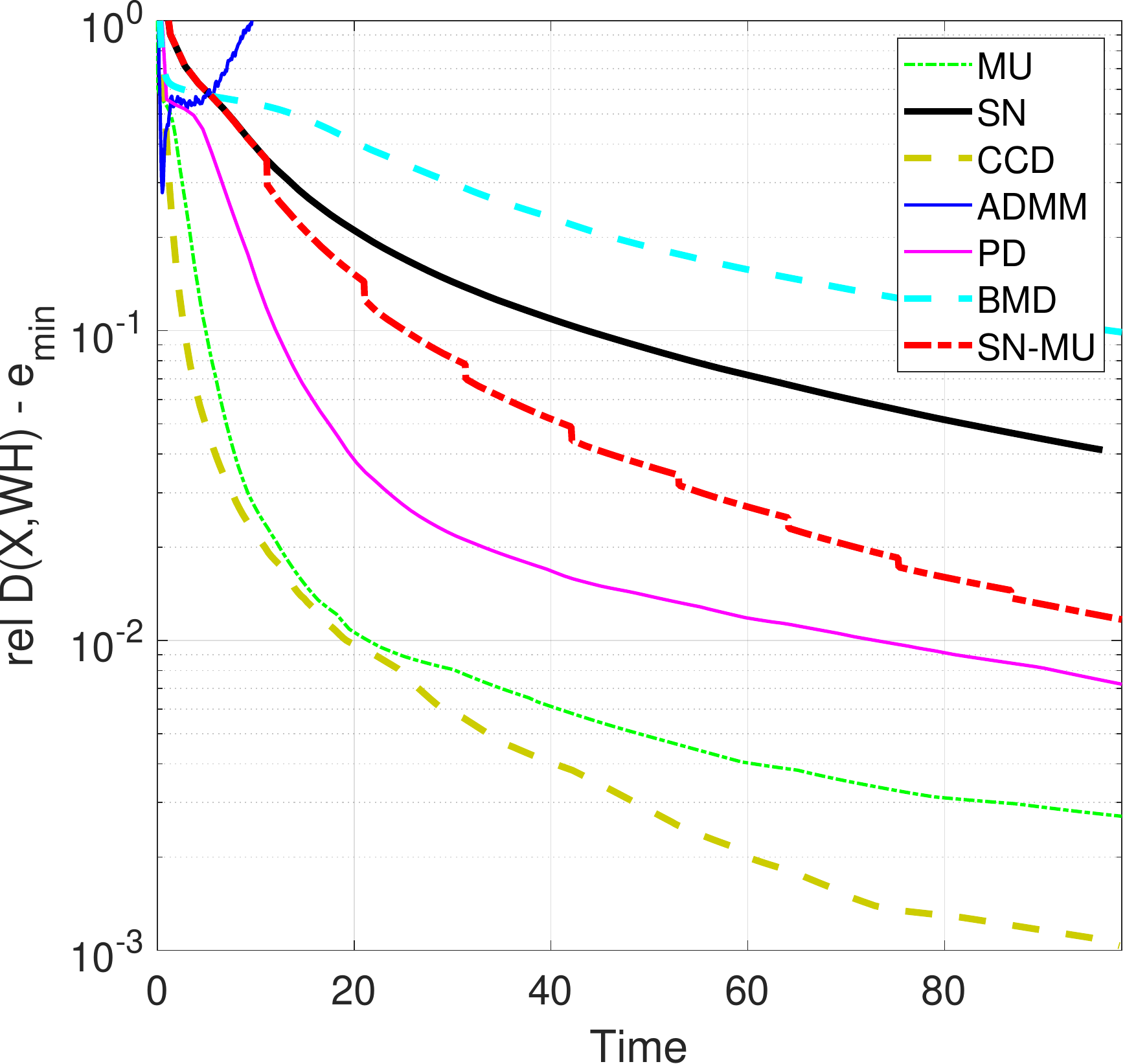} 
\end{tabular}
\caption{\revise{Median value} of $E(t)$ on image data sets:  left - cbclim,  right - ORLfaces.
\label{fig:image}} 
\end{center}
\vspace{-0.13in}
\end{figure*} 

\begin{table}[t]\centering 
 \caption{Average error, standard deviation and ranking among 60 different runs for the 2 image data sets. } 
   \label{tab:accuracy_image}
 \begin{tabular}{|c|c|c|} 
 \hline 
 Algorithm &  mean $\pm$ std & ranking  \\ 
 \hline 
ADMM &  $5.492 \, 10^{-1} \pm 2.167 \, 10^{-1}$ &  (0,  0,  0,  0,  0,  0, 60)   \\ 
CCD &  $\mathbf{3.339 \, 10^{-1} \pm 1.529 \, 10^{-1}}$ &  (\textbf{51},  8,  0,  1,  0,  0,  0)   \\ 
SN &  $3.569 \, 10^{-1} \pm 1.679 \, 10^{-1}$ &  (0,  0,  0,  0, 60,  0,  0)   \\ 
 MU &  $3.354 \, 10^{-1} \pm 1.533 \, 10^{-1}$ &  (8, 37, 10,  5,  0,  0,  0)   \\ 
 BMD &  $3.977 \, 10^{-1} \pm 1.867 \, 10^{-1}$ &  (0,  0,  0,  0,  0, 60,  0)   \\ 
 PD &  $3.379 \, 10^{-1} \pm 1.550 \, 10^{-1}$ &   (0,  4, 34, 22,  0,  0,  0)   \\ 
 SN-MU &  $3.387 \, 10^{-1} \pm 1.565 \, 10^{-1}$ &  (1, 11, 16, 32,  0,  0,  0)   \\ 
\hline 
\end{tabular} 
\vspace{-0.1in}
 \end{table} 

As for full-rank and audio data sets, 
ADMM diverges, and SN and BMD converge slowly. 
 CCD outperforms the other algorithms followed by MU, SN-MU and PD (in that order).

\subsubsection{Document data sets}

%We remark that, for PD, we use the Matlab command $normest$ to estimate the operator norm of a sparse matrix. 
For each document data set, we generate 10 random initial points and record the final relative errors (the reason of using only 10 initializations is that these data sets are rather large, and the computational time is high--we used 500 seconds for each run as shown on Table~\ref{tab:audio}). 
We report the average, the standard deviation of the final relative errors and the ranking vector over 60 runs (10 runs for each document data set) in Table~\ref{tab:accuracy_document}. 

\begin{table}[t]\centering 
 \caption{Average error, standard deviation and ranking among 60 different runs on the document data sets. } 
   \label{tab:accuracy_document}
 \begin{tabular}{|c|c|c|} 
 \hline 
 Algorithm &  mean $\pm$ std & ranking  \\ 
 \hline 
ADMM &  $2.333\, 10^{0}  \pm 3.276\, 10^{-1} $ &  (0,  0,  0,  0,  0, 10, 50)   \\ 
CCD &  $\mathbf{5.366 \, 10^{-1}  \pm 3.483\, 10^{-2}}$ &  (\textbf{37},  3, 16,  4,  0,  0,  0)   \\ 
SN &  $5.388\, 10^{-1}  \pm 3.443\, 10^{-2}$ &  (3, 13, 30, 14,  0,  0,  0)   \\ 
 MU &  $5.402\, 10^{-1}  \pm 3.394\, 10^{-2}$ &  (4,  9,  7, 40,  0,  0,  0)   \\ 
 BMD &  $6.057\, 10^{-1}  \pm 2.476\, 10^{-2}$ &  (0,  0,  0,  0, 60,  0,  0)   \\ 
 PD &  $1.783\, 10^{0} \pm 2.270\, 10^{0} $ &  (0,  0,  0,  0,  0, 50, 10)   \\ 
 SN-MU &  $5.370\, 10^{-1}  \pm 3.472\, 10^{-2} $ &  (16, 35,  7,  2,  0,  0,  0)   \\ 
\hline 
\end{tabular} 
\vspace{-0.13in}
 \end{table} 
 We observe that, CCD performs best, followed by SN-MU, in terms of the average relative error. 
 
 \paragraph{Performance profiles.} 
\revise{ 
Figure~\ref{fig:performance} reports the performance profiles for {the experiments} for synthetic (left) and real (right) data sets. It displays the performance of each algorithm as a function of $\rho \geq 0$.  
For a given value of $\rho$, we define the performance of an algorithm as follows 
\begin{equation} \label{eq:perf}
{\rm performance} (\rho) 
=
 \frac{\sharp\big\{ \text{solution } (W,H) 
 \mid 
 {\rm rel} \, D(V,WH) -  {\rm rel} \, D(V,W^*H^*) \leq \rho   \big\}}{\sharp {\rm runs}}, 
\end{equation}   
where a solution  $(W,H)$ is the final pair obtained within the total allotted time by the algorithm, and $(W^*,H^*)$ is the best solution obtained among all algorithms using the \emph{same} initialization. Hence, for example, 
performance(0) aggregates the values of the rankings at the first position provided in 
Tables~\ref{tab:rank_lowrank}-\ref{tab:accuracy_fullrank} for the synthetic data sets, and in 
Tables~\ref{tab:accuracy_audio}-\ref{tab:accuracy_document} 
for the real data sets.  

Performance profiles allow us to compare the algorithms meaningfully over different instances~\cite{dolan2002benchmarking}, that is, different matrices and initializations in our case. 
Looking at the curves for $\rho$ equal zero, we observe the percentage of the time each algorithm was able to obtain the best solution. The right of the curve, as $\rho$ increases, reports the robustness of an algorithm, that is, the percentage of times it was able to obtain a solution close to the best solutions found among all algorithms. In all cases, the higher the curve the better. 
\begin{figure*}[ht]
\center
\begin{tabular}{cc}
\includegraphics[width=0.475\textwidth]{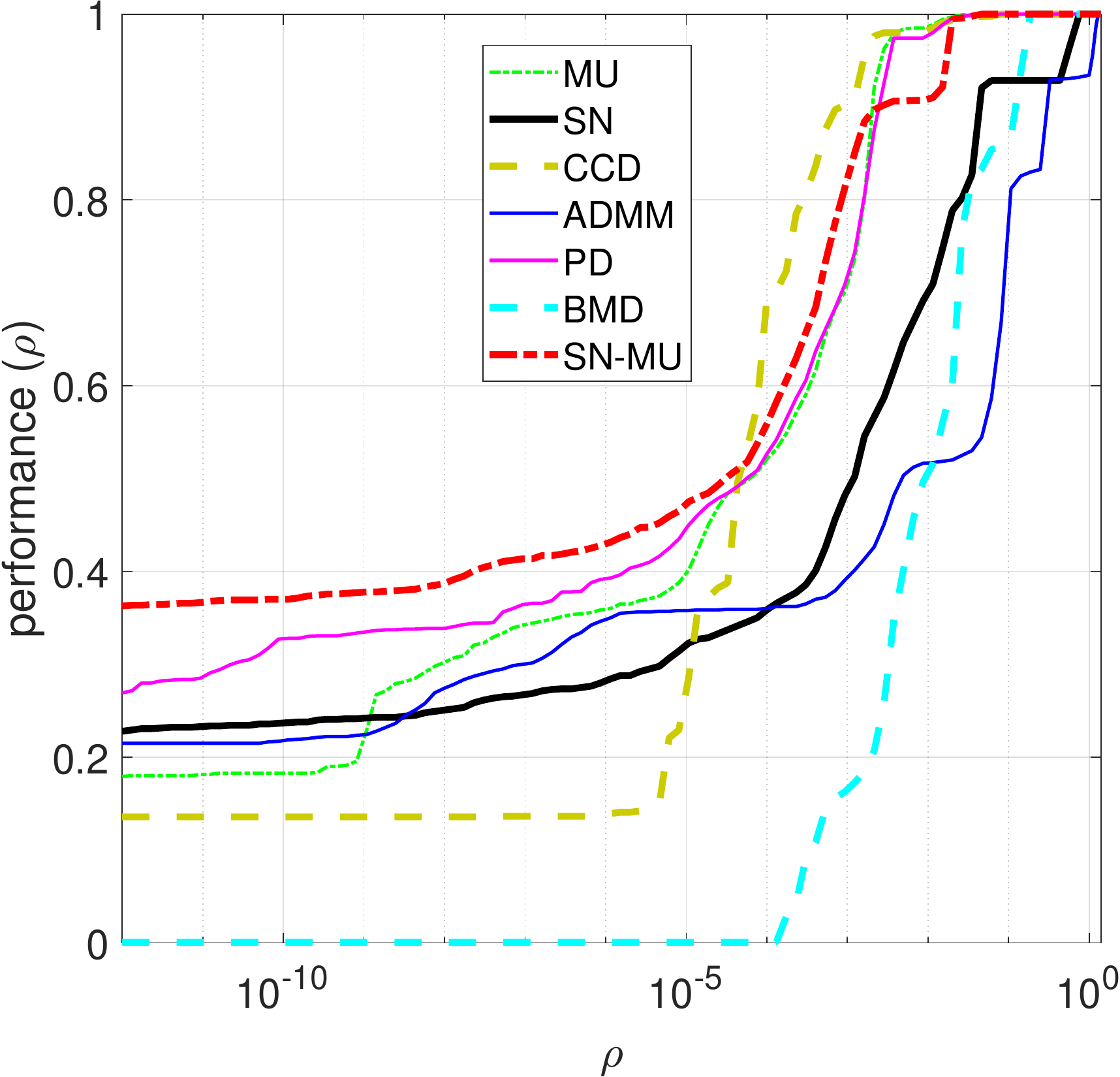} 
& 
\includegraphics[width=0.465\textwidth]{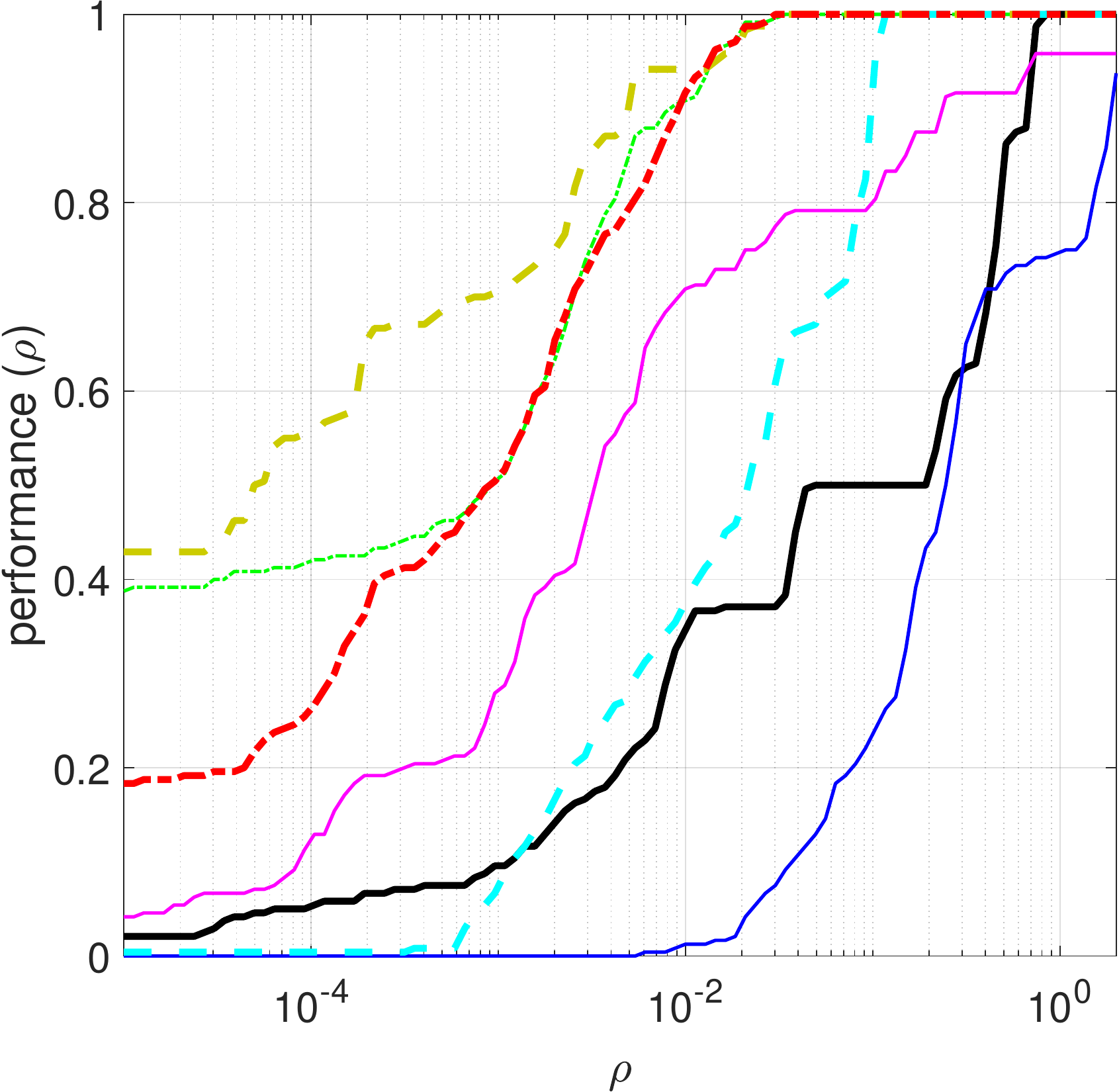}
\end{tabular} 
\caption{Performance profile for all experiments; see~\eqref{eq:perf}. On the left: synthetic data set, on the right: real data sets. 
\label{fig:performance}} 
\end{figure*}

Figure~\ref{fig:performance} confirms our observations: CCD, MU and SN-MU perform the best. On synthetic data sets, SN-MU provides the best solutions in most cases (left part of the left figure), while on real data sets, MU and CCD perform better (left part of the right figure). 
In terms of robustness, the three algorithms are comparable: their curves get closer together as $\rho$ increases.   
Note that, as we have observed as well, PD performs on par with  CCD, MU and SN-MU on synthetic data sets, while it is less effective for real data sets as it does not reach a performance of 1 even for $\rho = 1$.   
}

\subsection{Conclusions of the numerical experiments}

Surprisingly, for KL NMF, the behaviour of the algorithms can be highly dependent on the input data. 
For example, CCD performs best for images and documents, while MU performs best for audio data sets. 
To the best of our knowledge, this has not been reported in the literature. 
As far as we know, most papers focus only on a few numerical examples, introducing an undesirable bias towards certain algorithms. 
It is interesting to note that for Fro NMF, such different behaviors depending on the input data has not been reported in the literature, despite numerous studies.  
The reason is most likely that KL NMF is a more difficult optimization problem, for which the subproblem in $W$ and $H$, although convex, does not admit an $L$-Lipschitz gradient. 

The main take-home message of our experiments is that CCD, SN-MU and MU appear to be the most reliable algorithms for KL NMF, performing among the best in most scenarios. 
%However, one has to be careful when using CCD that can run into numerical problems, while MU and SN-MU have the advantage to be  guaranteed to monotonically decrease the objective function. 

\vspace{-0.1in}
\section{Conclusion} \label{sec:conclusion}

In this paper, we have presented important properties of KL NMF that are useful to analyze algorithms (Section~\ref{sec:KLprop}).  
Then, we have reviewed existing algorithms, and proposed three new algorithms: 
a block mirror descent (BMD) method  with global convergence guarantees, 
a scalar Newton-type (SN) algorithm which monotonically decreases the objective function, and an hybridization between SN and MU. 
Finally, in Section~\ref{sec:experiment}, we performed extensive numerical experiments on synthetic and real data sets. 
Although no KL NMF algorithms clearly outperforms the others, it appears that the CCD, MU and SN-MU provide the best results on average. %CCD also performs very well, the best in many cases, but it may encounter convergence issues. 

\revise{
\section*{Declarations}

\paragraph{Funding:}  This work was supported by the European Research Council (ERC starting grant n$^\text{o}$ 679515), and by the Fonds de la Recherche Scientifique - FNRS and the Fonds Wetenschappelijk Onderzoek - Vlaanderen (FWO) under EOS Project no O005318F-RG47. 

\paragraph{Conflicts of interest/Competing interests:} Not applicable. 

\paragraph{Availability of data and material:} 
 The data sets that support the findings of this study are freely available online, and available from the corresponding author, Nicolas Gillis, upon reasonable request.  

\paragraph{Code availability:} The codes used to perform the experiments in this paper are available from \url{https://github.com/LeThiKhanhHien/KLNMF}.  

}

\revise{
\section*{Acknowledgements} 
We thank the anonymous reviewers for their insightful comments that helped us improve the paper. We also thank Peter Carbonetto for helpful feedback and useful references. 
}

\bibliographystyle{spmpsci}  
\bibliography{biblio} 

\begin{thebibliography}{10}
\providecommand{\url}[1]{{#1}}
\providecommand{\urlprefix}{URL }
\expandafter\ifx\csname urlstyle\endcsname\relax
  \providecommand{\doi}[1]{DOI~\discretionary{}{}{}#1}\else
  \providecommand{\doi}{DOI~\discretionary{}{}{}\begingroup
  \urlstyle{rm}\Url}\fi

\bibitem{Ang2018}
Ang, A.M.S., Gillis, N.: Accelerating nonnegative matrix factorization
  algorithms using extrapolation.
\newblock Neural Computation \textbf{31}(2), 417--439 (2019)

\bibitem{Attouch2010}
Attouch, H., Bolte, J., Redont, P., Soubeyran, A.: Proximal alternating
  minimization and projection methods for nonconvex problems: An approach based
  on the {K}urdyka-{{\L}}ojasiewicz inequality.
\newblock Mathematics of Operations Research \textbf{35}(2), 438--457 (2010)

\bibitem{Bauschke2017}
Bauschke, H.H., Bolte, J., Teboulle, M.: A descent lemma beyond {L}ipschitz
  gradient continuity: First-order methods revisited and applications.
\newblock Mathematics of Operations Research \textbf{42}(2), 330--348 (2017).
\newblock \doi{10.1287/moor.2016.0817}

\bibitem{Bioucas-Dias}
Bioucas-Dias, J.M., Plaza, A., Dobigeon, N., Parente, M., Du, Q., Gader, P.,
  Chanussot, J.: Hyperspectral unmixing overview: Geometrical, statistical, and
  sparse regression-based approaches.
\newblock IEEE Journal of Selected Topics in Applied Earth Observations and
  Remote Sensing \textbf{5}(2), 354--379 (2012)

\bibitem{Bolte2014}
Bolte, J., Sabach, S., Teboulle, M.: Proximal alternating linearized
  minimization for nonconvex and nonsmooth problems.
\newblock Mathematical Programming \textbf{146}(1), 459--494 (2014)

\bibitem{unpublishedCarbonetto}
Carbonetto, P., Luo, K., Dey, K., Hsiao, J., Stephens, M.: {fastTopics: fast
  algorithms for fitting topic models and non-negative matrix factorizations to
  count data} (2021).
\newblock R package version 0.4-11,
  \url{https://github.com/stephenslab/fastTopics}

\bibitem{ChambolleP11}
Chambolle, A., Pock, T.: A first-order primal-dual algorithm for convex
  problems with applications to imaging.
\newblock Journal of Mathematical Imaging and Vision \textbf{40}(1), 120--145
  (2011)

\bibitem{chi2012tensors}
Chi, E.C., Kolda, T.G.: On tensors, sparsity, and nonnegative factorizations.
\newblock SIAM Journal on Matrix Analysis and Applications \textbf{33}(4),
  1272--1299 (2012)

\bibitem{cichocki2009nonnegative}
Cichocki, A., Zdunek, R., Phan, A.H., Amari, S.: Nonnegative matrix and tensor
  factorizations: applications to exploratory multi-way data analysis and blind
  source separation.
\newblock John Wiley \& Sons (2009)

\bibitem{dey2017visualizing}
Dey, K.K., Hsiao, C.J., Stephens, M.: Visualizing the structure of rna-seq
  expression data using grade of membership models.
\newblock PLoS genetics \textbf{13}(3), e1006599 (2017)

\bibitem{DING2008}
Ding, C., Li, T., Peng, W.: On the equivalence between non-negative matrix
  factorization and probabilistic latent semantic indexing.
\newblock Computational Statistics \& Data Analysis \textbf{52}(8), 3913 --
  3927 (2008)

\bibitem{dolan2002benchmarking}
Dolan, E.D., Mor{\'e}, J.J.: Benchmarking optimization software with
  performance profiles.
\newblock Mathematical programming \textbf{91}(2), 201--213 (2002)

\bibitem{fevotte2009nonnegative}
F{\'e}votte, C., Bertin, N., Durrieu, J.L.: Nonnegative matrix factorization
  with the {Itakura-Saito} divergence: With application to music analysis.
\newblock Neural computation \textbf{21}(3), 793--830 (2009)

\bibitem{Fevotte2009_prob}
{F\'{e}votte}, C., {Cemgil}, A.T.: Nonnegative matrix factorizations as
  probabilistic inference in composite models.
\newblock In: 2009 17th European Signal Processing Conference, pp. 1913--1917
  (2009)

\bibitem{Fevotte2011}
F\'{e}votte, C., Idier, J.: Algorithms for nonnegative matrix factorization
  with the $\beta$-divergence.
\newblock Neural Computation \textbf{23}(9), 2421--2456 (2011).
\newblock \doi{10.1162/NECO_a_00168}

\bibitem{FS2006}
Finesso, L., Spreij, P.: Nonnegative matrix factorization and {I}-divergence
  alternating minimization.
\newblock Linear Algebra and its Applications \textbf{416}(2), 270 -- 287
  (2006).
\newblock \doi{https://doi.org/10.1016/j.laa.2005.11.012}

\bibitem{xiao2019uniq}
Fu, X., Huang, K., Sidiropoulos, N.D., Ma, W.K.: Nonnegative matrix
  factorization for signal and data analytics: Identifiability, algorithms, and
  applications.
\newblock IEEE Signal Processing Magazine \textbf{36}(2), 59--80 (2019)

\bibitem{gillis2014}
Gillis, N.: The why and how of nonnegative matrix factorization.
\newblock In: J.~Suykens, M.~Signoretto, A.~Argyriou (eds.) Regularization,
  Optimization, Kernels, and Support Vector Machines, Machine Learning and
  Pattern Recognition, chap.~12, pp. 257--291. Chapman \& Hall/CRC, Boca Raton,
  Florida (2014)

\bibitem{gillis2017introduction}
Gillis, N.: Introduction to nonnegative matrix factorization.
\newblock SIAG/OPT Views and News \textbf{25}(1), 7--16 (2017)

\bibitem{gillis2019distributionally}
Gillis, N., Hien, L.T.K., Leplat, V., Tan, V.Y.: Distributionally robust and
  multi-objective nonnegative matrix factorization.
\newblock arXiv preprint arXiv:1901.10757  (2019)

\bibitem{GZ2005}
Gonzalez, E.F., Zhang, Y.: Accelerating the {Lee-Seung} algorithm for
  nonnegative matrix factorization (2005)

\bibitem{Guillamet}
Guillamet, D., Vitri{\`a}, J.: Non-negative matrix factorization for face
  recognition.
\newblock In: M.T. Escrig, F.~Toledo, E.~Golobardes (eds.) Topics in Artificial
  Intelligence, pp. 336--344. Springer Berlin Heidelberg, Berlin, Heidelberg
  (2002)

\bibitem{hasinoff2014photon}
Hasinoff, S.W.: Photon, {Poisson} noise (2014).
\newblock
  \urlprefix\url{http://people.csail.mit.edu/hasinoff/pubs/hasinoff-photon-2011-preprint.pdf}

\bibitem{Hien2019}
Hien, L.T.K., Gillis, N., Patrinos, P.: Inertial block proximal method for
  non-convex non-smooth optimization.
\newblock In: Thirty-seventh International Conference on Machine Learning ICML
  2020 (2020)

\bibitem{Ho2008}
Ho, N.D., Dooren, P.V.: Non-negative matrix factorization with fixed row and
  column sums.
\newblock Linear Algebra and its Applications \textbf{429}(5), 1020 -- 1025
  (2008).
\newblock Special Issue devoted to selected papers presented at the 13th
  Conference of the International Linear Algebra Society

\bibitem{hong2015unified}
Hong, M., Razaviyayn, M., Luo, Z.Q., Pang, J.S.: A unified algorithmic
  framework for block-structured optimization involving big data: With
  applications in machine learning and signal processing.
\newblock IEEE Signal Processing Magazine \textbf{33}(1), 57--77 (2015)

\bibitem{hoyer2004non}
Hoyer, P.O.: Non-negative matrix factorization with sparseness constraints.
\newblock Journal of Machine Learning Research \textbf{5}(Nov), 1457--1469
  (2004)

\bibitem{Hsieh2011}
Hsieh, C.J., Dhillon, I.S.: Fast coordinate descent methods with variable
  selection for non-negative matrix factorization.
\newblock In: Proceedings of the 17th ACM SIGKDD international conference on
  knowledge discovery and data mining, pp. 1064--1072 (2011)

\bibitem{huang2016flexible}
Huang, K., Sidiropoulos, N.D., Liavas, A.P.: A flexible and efficient
  algorithmic framework for constrained matrix and tensor factorization.
\newblock IEEE Transactions on Signal Processing \textbf{64}(19), 5052--5065
  (2016)

\bibitem{kim2019scale}
Kim, C., Kim, Y., Klabjan, D.: Scale invariant power iteration (2019).
\newblock ArXiv:1905.09882

\bibitem{Lee99}
Lee, D.D., Seung, H.S.: Learning the parts of objects by nonnegative matrix
  factorization.
\newblock Nature \textbf{401}, 788--791 (1999)

\bibitem{lee2001algorithms}
Lee, D.D., Seung, H.S.: Algorithms for non-negative matrix factorization.
\newblock In: Advances in neural information processing systems, pp. 556--562
  (2001)

\bibitem{leskovec2014mining}
Leskovec, J., Rajaraman, A., Ullman, J.D.: Mining of Massive Datasets, second
  edn.
\newblock Cambridge University Press (2014).
\newblock \urlprefix\url{http://mmds.org}

\bibitem{Li2012}
Li, L., Lebanon, G., Park, H.: Fast {B}regman divergence {NMF} using {Taylor}
  expansion and coordinate descent.
\newblock In: Proceedings of the 18th ACM SIGKDD International Conference on
  Knowledge Discovery and Data Mining, KDD ’12, p. 307–315. Association for
  Computing Machinery, New York, NY, USA (2012)

\bibitem{lin2007convergence}
Lin, C.J.: On the convergence of multiplicative update algorithms for
  nonnegative matrix factorization.
\newblock IEEE Transactions on Neural Networks \textbf{18}(6), 1589--1596
  (2007)

\bibitem{lin2020optimization}
Lin, X., Boutros, P.C.: Optimization and expansion of non-negative matrix
  factorization.
\newblock BMC bioinformatics \textbf{21}(1), 1--10 (2020)

\bibitem{Lu2018}
Lu, H., Freund, R.M., Nesterov, Y.: Relatively smooth convex optimization by
  first-order methods, and applications.
\newblock SIAM Journal on Optimization \textbf{28}, 333--354 (2018)

\bibitem{lucy1974iterative}
Lucy, L.B.: An iterative technique for the rectification of observed
  distributions.
\newblock The astronomical journal \textbf{79}, 745 (1974)

\bibitem{Maetal2014}
Ma, W., Bioucas-Dias, J.M., Chan, T., Gillis, N., Gader, P., Plaza, A.J.,
  Ambikapathi, A., Chi, C.: A signal processing perspective on hyperspectral
  unmixing: Insights from remote sensing.
\newblock IEEE Signal Processing Magazine \textbf{31}(1), 67--81 (2014)

\bibitem{NesterovLecture2018}
Nesterov, Y.: Lectures on Convex Optimization.
\newblock Springer (2018)

\bibitem{Razaviyayn2013}
Razaviyayn, M., Hong, M., Luo, Z.: A unified convergence analysis of block
  successive minimization methods for nonsmooth optimization.
\newblock SIAM Journal on Optimization \textbf{23}(2), 1126--1153 (2013)

\bibitem{richardson1972bayesian}
Richardson, W.H.: {Bayesian}-based iterative method of image restoration.
\newblock JoSA \textbf{62}(1), 55--59 (1972)

\bibitem{SHAHNAZ2006}
Shahnaz, F., Berry, M.W., Pauca, V., Plemmons, R.J.: Document clustering using
  nonnegative matrix factorization.
\newblock Information Processing \& Management \textbf{42}(2), 373 -- 386
  (2006).
\newblock
  \urlprefix\url{http://www.sciencedirect.com/science/article/pii/S0306457304001542}

\bibitem{Sun2014}
{Sun}, D.L., {F\'{e}votte}, C.: Alternating direction method of multipliers for
  non-negative matrix factorization with the beta-divergence.
\newblock In: 2014 IEEE International Conference on Acoustics, Speech and
  Signal Processing (ICASSP), pp. 6201--6205 (2014)

\bibitem{takahashi2014global}
Takahashi, N., Hibi, R.: Global convergence of modified multiplicative updates
  for nonnegative matrix factorization.
\newblock Computational Optimization and Applications \textbf{57}(2), 417--440
  (2014)

\bibitem{Teboulle2020}
Teboulle, M., Vaisbourd, Y.: Novel proximal gradient methods for nonnegative
  matrix factorization with sparsity constraints.
\newblock SIAM Journal on Imaging Sciences \textbf{13}(1), 381--421 (2020)

\bibitem{Tran2015}
Tran-Dinh, Q., Kyrillidis, A., Cevher, V.: Composite self-concordant
  minimization.
\newblock Journal of Machine Learning Research \textbf{16}(12), 371--416 (2015)

\bibitem{Tseng2008}
Tseng, P.: On accelerated proximal gradient methods for convex-concave
  optimization.
\newblock Tech. rep. (2008)

\bibitem{Xu2013}
Xu, Y., Yin, W.: A block coordinate descent method for regularized multiconvex
  optimization with applications to nonnegative tensor factorization and
  completion.
\newblock SIAM Journal on Imaging Sciences \textbf{6}(3), 1758--1789 (2013).
\newblock \doi{10.1137/120887795}

\bibitem{Yanez2017}
{Yanez}, F., {Bach}, F.: Primal-dual algorithms for non-negative matrix
  factorization with the {Kullback-Leibler} divergence.
\newblock In: 2017 IEEE International Conference on Acoustics, Speech and
  Signal Processing (ICASSP), pp. 2257--2261 (2017)

\bibitem{zhong2005generative}
Zhong, S., Ghosh, J.: Generative model-based document clustering: a comparative
  study.
\newblock Knowledge and Information Systems \textbf{8}(3), 374--384 (2005)

\end{thebibliography}
\vspace{-0.1in}
\appendix
\section{Technical proofs}
\vspace{-0.1in}
\subsection{Proof of Proposition \ref{prop:sol_existence}} \label{app:proofprop1}
\textbf{Proposition \ref{prop:sol_existence}: } (A) Given $\varepsilon\geq 0$ and a nonnegative matrix $V$, Problem~\eqref{KLNMF_perturbed} attains its minimum, that is, it has at least one optimal solution. 

(B) Let $ \nu=\sum_{i=1}^m\sum_{j=1}^n V_{ij}$. Denote $D^*(V,\varepsilon)$ be the optimal value of 
Problem~\eqref{KLNMF_perturbed}. We have 
$D^*(V,\varepsilon) \leq D^*(V,0) + (\min\{ n+ mr, m+nr\} \sqrt{\nu}+mn\varepsilon)\varepsilon.
$

\begin{proof}
Problem~\eqref{KLNMF_perturbed} can be rewritten as follows
\begin{align}
\label{KLNMF_perturbed_2} 
\min_{W,H} &  %f_\varepsilon(W,H):=
\quad D\big(V|(W+\varepsilon ee^\top)(H+\varepsilon ee^\top) \big)\\ 
 \text{ such that } &  W_{ik}\geq 0, H_{kj} \geq 0 \text{ for } i \in [m], k \in [r],j \in [n]. \nonumber 
\end{align} 
We note that 
\begin{equation}
\label{eq:D_epsilon}
(W+\varepsilon ee^\top)(H+\varepsilon ee^\top)=WH+ \varepsilon (W ee^\top + ee^\top H)+\varepsilon^2 ee^\top.
\end{equation}

Let us now prove the parts (A) and (B) 
of Proposition~\ref{prop:sol_existence} separately. 

(A) We note that 
$$
\frac{1}{\nu}D\big(V|(W+\varepsilon ee^\top)(H+\varepsilon ee^\top)\big)=D\Big(\frac{V}{\nu}\Big|\big(\frac{W}{\sqrt{\nu}}+\frac{\varepsilon}{\sqrt{\nu}} ee^\top\big)\big(\frac{H}{\sqrt{\nu}}+\frac{\varepsilon}{\sqrt{\nu}} ee^\top\big)\Big).
$$
 Then we have
 \begin{equation}
 \label{eq:Dnu}
  D^*(V,\varepsilon)=\nu D^*(V/\nu,\varepsilon/\sqrt{\nu}).
\end{equation} 
Let us now consider Problem~\eqref{KLNMF_perturbed_2} with $V$ such that $\sum_{i=1}^m\sum_{j=1}^n V_{ij}=1$. In the following, we separate the proof into 2 cases: $\varepsilon=0$, which corresponds to the original KL NMF Problem~\eqref{KLNMF}, and $\varepsilon>0$, for which the objective of Problem~\eqref{KLNMF_perturbed_2} is finite at every pair of non-negative matrices $(W,H)$. 

~

\noindent\underline{Case 1}: $\varepsilon=0$. We use the methodology in the proof of \cite[Proposition 2.1]{FS2006}.
We first observe that $WH=\sum_{i=1}^r W_{:,i} H_{i,:}$, hence without loss of generality we can consider the case when $H$ has no zero rows; otherwise we could then consider Problem~\eqref{KLNMF} with smaller rank than $r$.  Given a feasible solution $(W,H)$ of \eqref{KLNMF}, let $h_k=\sum_{j=1}^n H_{kj}$ and let ${\rm Diag}(h)$ be the diagonal matrix with its diagonal being $h$. We then have $(\tilde{W},\tilde{H})$, where $\tilde{W}= W {\rm Diag}(h)$ and $\tilde{H}= ({\rm Diag}(h))^{-1} H$, is also a feasible solution of \eqref{KLNMF} and it preserves the objective value since $\tilde{W} \tilde{H}=WH$. We observe that $\sum_{j=1}^n \tilde H_{kj}=1$. %Hence, we can restrict our search to matrices $H$ such that $He_{n\times 1} = e_{r\times 1}$. %On the other hand, we see that $(\delta e_{m\times r} , \tfrac{1}{n} e_{r\times n})$, where $\delta>0$, is a feasible solution of \eqref{KLNMF}. Let 
%\begin{equation}
%\label{eq:boundC}
%C=f(\delta e_{m\times r} , \tfrac{1}{n} e_{r\times n})= m\delta r - \log\tfrac{\delta r}{n}+\sum_{i=1}^m\sum_{j=1}^n V_{ij}\log V_{ij}-1.
%\end{equation}
Hence, we can restrict our search for the optimal solution to the set 
$$
\mathcal A_1=\{(W,H): W,H \geq 0, 
He = e\}. 
$$
We note that $ \lim_{x\to 0} (x - V_{ij} \log x) = +\infty$ for a given $V_{ij}>0$, and $x \mapsto (x - V_{ij} \log x)$ is a decreasing function when $x<V_{ij}$. Therefore, there exists a positive constant $\delta$ such that $(WH)_{ij} \geq \delta$ for all $i,j$ with $V_{ij}>0$; otherwise, if for all $\delta>0$ there exists $(i,j)$ with $V_{ij}>0$ that $(WH)_{ij} \leq \delta$  then we can choose arbitrary small $\delta<\min\{V_{ij}:V_{ij}>0\}$ which makes $(WH)_{ij} - V_{ij} \log (WH)_{ij}$ arbitrary large. Therefore, we can further restrict the constraint set of \eqref{KLNMF} to the set $$
\mathcal A_2=\big\{(W,H): (W,H)\in \mathcal A_1, (WH)_{ij}\geq \delta \forall (i,j) \in \{(i,j):V_{ij} >0\} \big\}.
$$
It follows from \cite[Theorem 2]{lee2001algorithms} (see also Section \ref{sec:MU}) that, given a feasible solution $(W,H)$, a multiplicative update step $W_{ik}^+= W_{ik} \big(\sum_{l=1,V_{il}>0}^n \frac{H_{kl} V_{il}}{(WH)_{il}} \big)\big/ \big(\sum_{l=1}^n H_{kl} \big) $, for $i\in [m], k\in [r]$, supposed to be well-defined, 
decreases the objective function. Hence, given $(W,H)\in \mathcal A_2$, we observe that
$$
\sum_{k=1}^m W_{ik}^+=\sum_{k=1}^m W_{ik} \sum_{l=1}^n \frac{H_{kl} V_{il}}{(WH)_{il}}=\sum_{l=1}^n \sum_{k=1}^m  \frac{W_{ik}H_{kl} V_{il}}{(WH)_{il}}=\sum_{l=1}^n V_{il}.
$$
Therefore, we can further restrict our search for the optimal solutions of \eqref{KLNMF} to
$$
\mathcal A_3=\big\{(W,H): (W,H)\in \mathcal A_1, W e =V e, (WH)_{ij}\geq \delta \forall (i,j) \in \{(i,j):V_{ij} >0\} \big\}.
$$
%Using inequality $\log x \leq x-1$ for all $x>0$, we obtain 
%\[
%\begin{split}
%D(V|WH) &\geq \sum_{i=1}^m\sum_{j=1}^n (1-V_{ij})(WH)_{ij}+ \sum_{i=1}^m\sum_{j=1}^n V_{ij}\log V_{ij} \\
%&\geq (1-v_{\max}) \sum_{i=1}^m\sum_{j=1}^n (WH)_{ij} + \sum_{i=1}^m\sum_{j=1}^n V_{ij}\log V_{ij},
%\end{split}
%\]
%where $v_{\max}=\max_{ij} V_{ij}$ and we have used $\sum_{i,j} V_{ij}=1$. Hence, it follows from $ D(V|WH)\leq C$ that 
%\begin{equation}
%\label{eq:boundWH}
%\sum_{i=1}^m\sum_{j=1}^n (WH)_{ij}\leq  \frac{ m\delta r - \log\tfrac{\delta r}{n}-1}{1-v_{\max}}.
%\end{equation}
%For all $i\in [m], k \in [r]$, let $H_{kj_k}=\max_{j\in [n]}\{H_{kj}\}$, then we have $W_{ik}H_{kj_k} \leq (WH)_{ij_k}$. On the other hand, as $\sum_{j} H_{kj}=1$ we get $H_{kj_k}\geq \tfrac{1}{n}$. Together with the bound \eqref{eq:boundWH}, we obtain 
%\begin{equation}
%\label{eq:boundWik}
%W_{ik} \leq n \frac{ m\delta r - \log\tfrac{\delta r}{n}-1}{1-v_{\max}}, \forall\, i\in [m], k\in [r].
%\end{equation}
%Hence, the set $\mathcal{A}_1$ is bounded. 
By choosing appropriate $\delta$, we note that $(W,1/n ee^\top) \in \mathcal{A}_3$ with $W_{ik}=1/m \sum_{l} V_{il}$ for $i\in [m], k\in [r]$, then the set $\mathcal{A}_3$ is non-empty. We see that $\mathcal{A}_3$ is closed and bounded, and $D(V|WH))$ is continuous on $\mathcal{A}_3$, hence, Problem~\eqref{KLNMF} has at least one solution.

~

\noindent\underline{Case 2}:  $\varepsilon>0$. 
When $\varepsilon>0$, we see $(0, 0)$ is a feasible solution of \eqref{KLNMF_perturbed_2}. Let 
$$C=f_\varepsilon(0 , 0)= mnr\varepsilon^2-\log (r\varepsilon^2) +\sum_{i=1}^m\sum_{j=1}^n V_{ij}\log V_{ij}-1.
$$
We then can restrict our search for the optimal solutions of Problem~\eqref{KLNMF_perturbed_2} to the constraint set  
$\mathcal A_4=\{(W,H): W,H \geq 0, f_\varepsilon(W,H)\leq C\}. $
Using inequality $\log x \leq x-1$ for all $x>0$, we obtain 
\begin{align*}
\begin{split}
f_\varepsilon(W,H) &\geq \sum_{i=1}^m\sum_{j=1}^n (1-V_{ij})\big((WH)_{ij}+\varepsilon (W ee^\top + ee^\top H)_{ij}+\varepsilon^2 \big )+ \sum_{i=1}^m\sum_{j=1}^n V_{ij}\log V_{ij} \\
&\geq \sum_{i=1}^m\sum_{j=1}^n (1-V_{ij})\big (\varepsilon \sum_{k}W_{ik}+\varepsilon \sum_{k}H_{kj} +\varepsilon^2\big)+ \sum_{i=1}^m\sum_{j=1}^n V_{ij}\log V_{ij}\\
& \geq (mn-1)\varepsilon^2+ (1-v_{\max})\varepsilon\big(n \sum_{ik}W_{ik} + m \sum_{kj}H_{kj}\big)	+ \sum_{i=1}^m\sum_{j=1}^n V_{ij}\log V_{ij},
\end{split}
\end{align*}
where  $v_{\max}=\max_{ij} V_{ij}$, we use $(WH)_{ij}\geq 0$ in the second inequality and $\sum_{ij}V_{ij}=1$ in the third inequality. 
Hence, from $f_\varepsilon(W,H)\leq C$, $W$ and $H$ are bounded. 
%\begin{equation}
%\begin{split}
%\sum_{ik} W_{ik} \leq \omega_1:= \frac{C_4-1-(mn-1)\varepsilon^2}{ n(1-v_{\max})\varepsilon},
%\\
%\sum_{kj} H_{kj} \leq \omega_2:=\frac{C_4-1-(mn-1)\varepsilon^2}{m (1-v_{\max})\varepsilon},
%\end{split}
%\end{equation}
%where $C_4=\max\{mnr\varepsilon^2-\log(r\varepsilon^2), mr\varepsilon-\log \tfrac{\varepsilon r}{n}\}$.
Therefore, $\mathcal A_4$ is bounded. % we can restrict the constraint set of Problem~\eqref{KLNMF_perturbed_2} to the compact set 
%$$
%\mathcal{A}_1:=\{(W,H): W\geq 0, H\geq 0, \sum_{ik} W_{ik} \leq \omega_1, \sum_{kj} H_{kj} \leq  \omega_2\}.
%$$
We also see that $\mathcal A_4$ is closed; hence, it is a compact set. The objective of \eqref{KLNMF_perturbed_2} is continuous on this set since $\varepsilon>0$. Therefore, Problem~\eqref{KLNMF_perturbed_2} has at least one solution. 

(B) Let us fix $\varepsilon>0$. Note that all points of $\mathcal{A}_3$ are feasible solutions of Problem~\eqref{KLNMF_perturbed_2}. We hence have 
\begin{equation}
\label{eq:bound_forD}
\begin{split}
  D^*(V,\varepsilon)&\leq \min_{(W,H)\in \mathcal{A}_3} f_\varepsilon(W,H)\\&\leq \min_{(W,H)\in \mathcal{A}_3} D(V|WH) 
 + \max_{(W,H)\in \mathcal{A}_3} (f_\varepsilon(W,H)-D(V|WH))\\
&= D^*(V,0) + \max_{(W,H)\in \mathcal{A}_3} (f_\varepsilon(W,H)-D(V|WH)),
\end{split}
\end{equation}
where in the first inequality we use the fact that the optimal value over bigger set does not exceed the optimal value over smaller set. 
On the other hand, it follows from \eqref{eq:D_epsilon} and the case $\sum_{i,j} V_{ij}=1$ that, for $(W,H)\in \mathcal A_3$, we have
\begin{align*}
\begin{split}
f_\varepsilon(W,H)-D(V|WH)&= \sum_{i=1}^m\sum_{j=1}^n \big(\varepsilon \sum_{k=1}^r W_{ik} + \varepsilon \sum_{k=1}^r H_{kj} + \varepsilon^2\big)\\
&\qquad\qquad -\sum_{i=1}^m\sum_{j=1}^n V_{ij} \log\frac{((W+\varepsilon  ee^\top)(H+\varepsilon ee^\top))_{ij}}{(WH)_{ij}}\\
&\leq n  \varepsilon +  mr \varepsilon +  mn \varepsilon^2.
\end{split}
\end{align*}
Hence $D^*(V,\varepsilon) \leq D^*(V,0) +  n  \varepsilon +  mr \varepsilon +  mn \varepsilon^2$.  By exchanging the role of $W$ and $H^\top $ and noting that $D(V|WH)=D(V^\top |H^\top  W^\top )$, we can prove a similar bound 
$D^*(V,\varepsilon) \leq D^*(V,0) +  m  \varepsilon +  nr \varepsilon +  mn \varepsilon^2$.
Together with \eqref{eq:Dnu}, we obtain Result (B). 
\end{proof}
\subsection{BSUM framework for the MU} \label{app:bsum}

MU is a block successive upper-bound minimization algorithm (BSUM) in which each column of $H$ and each row of $W$ are updated by minimizing majorized functions of the KL objective. Let us first introduce BSUM,  then derive MU for solving~\eqref{KLNMF_perturbed}. 

BSUM was proposed in  \cite{Razaviyayn2013} to solve the minimization problem in~\eqref{eq:compositev2}. 
%\begin{equation}
%\label{eq:composite}
%\min_{x\in \mathcal X} f(x),
%\end{equation}
% where the joint variable $x$ can be decomposed as $x=(x_1,\ldots,x_s)$, where we use the bold $x$ to denote the joint variable, and $x_i$ to denote the block variable ($i=1\ldots,s$). 
% The constraint set is $\mathcal X= \mathcal X_1 \times \ldots\times \mathcal{X}_s$ where $\mathcal X_i$ are closed convex sets. 
% For notation succinctness, we use $f$ to denote the objective of \eqref{KLNMF_perturbed} and use it in a flexible manner; for example, $f$ can be a function of 2 block variables that are $W$ and $H$, $f$ can also be a function of $mr+rn$ variables that are the scalars $W_{ik}$ and $H_{kj}$, this will be clear from the context. 
% The sets $\mathcal X_i$ of the perturbed KL NMF~\eqref{KLNMF_perturbed} problem is the corresponding constraint set for the block variable $x_i$, that is, $\mathcal X_i=\{x_i: x_i\geq \varepsilon\}$. 
 %Considering~\eqref{eq:compositev2}, 
% BSUM alternately updates each block $x_i$ by fixing the latest values of block $j\ne i$ and minimizing a majorized function of the objective. 
Putting in the notations of \eqref{eq:compositev2}, first, let us formally define a majorized function.  
\begin{definition}%\cite[Definition 1]{Lee99}
\label{def:majorize}
A function $u_i : \mathcal X_i \times \mathcal X \to \mathbb R$ is said to majorize $f(x)$ if 
\begin{equation}
\label{eq:majorize}
\begin{split}
&u_i (x_i , x) = f(x), \forall \, x \in \mathcal X, 
\\ 
&u_i (y_i , x) \geq f(x_1,\ldots,x_{i-1},y_i,x_{i+1},\ldots,x_s), \forall \, y_i \in \mathcal{X}_i,  x \in \mathcal X.
\end{split}
\end{equation}
\end{definition}
 Using the majorized functions, at iteration $k$, BSUM fixes the latest values of block $j\ne i$ and updates block $x_i$  by  
\begin{equation}
\label{eq:MM}
x_i^{k}=\arg\min_{x_i\in \mathcal X_i} u_i(x_i,x_1^k,\ldots,x_{i-1}^k,x_i^{k-1}, x_{i+1}^{k-1},\ldots,x_s^{k-1}).
\end{equation}
%It is straight-forward to derive from Definition \ref{def:majorize} that BSUM generates a sequence $\{x^k\}$ that satisfies 
From Definition \ref{def:majorize} we have
\begin{align*}
\begin{split}
f(x^k)&=u_1(x_1^k,x_1^k,\ldots,x_s^k) \geq u_1(x_1^{k+1},x_1^k,\ldots,x_s^k)\\&\geq f(x_1^{k+1},x_2^k,\ldots,x_s^k)=u_2(x_2^k,x_1^{k+1},x_2^k,\ldots,x_s^k)\\
&\geq u_2(x_2^{k+1},x_1^{k+1},x_2^k,\ldots,x_s^k) \geq f(x_1^{k+1},x_2^{k+1},\ldots,x_{s-1}^k,x_s^k) \geq  \ldots\geq f(x^{k+1}).
\end{split}
\end{align*}
In other words, BSUM produces a non-increasing sequence $\{f(x^k)\}$.  In the following, we give a majorized function for $D(t|Wh)$ defined in \eqref{Dh}. 
\begin{proposition} \cite[Lemma 3]{Lee99} 
\label{prop:MUmajorize}
Let 
\begin{align*}
u_{\rm MU}(h,h^k):=\sum_{i=1}^m\Big((Wh)_i -t_i\sum_{l=1}^r \frac{W_{il}(h^k)_l}{(Wh^k)_i}\Big(\log (W_{il}h_l)-\log \frac{W_{il}(h^k)_l}{(Wh^k)_i}\Big )+ t_i \log t_i -t_i \Big).
\end{align*}
Then $u_{\rm MU}(h,h^k)$ majorizes the function $h\mapsto D(t|Wh)$. 
\end{proposition}
When BSUM uses $u_{\rm MU}(\cdot,\cdot)$ to update a column $h$ of $H$ (similarly for a row of $W$) by 
$h^{k+1} = \arg\min_{h\geq \varepsilon} u_{\rm MU}(h,h^k)$, we obtain 
$
(h^{k+1})_l=\max\left\{\frac{(h^k)_l}{\sum_{i=1}^m W_{il}}\sum_{i=1}^m \frac{t_i W_{il}}{(Wh^k)_i}, \varepsilon\right\},
$
leading to the MU. More specifically, the updates of MU for solving Problem~\eqref{KLNMF_perturbed} are %have the following forms
\begin{equation}
\label{KLMU}
W_{ik}^+= \max\Big\{ W_{ik} \frac{\sum_{l=1}^n \frac{H_{kl} V_{il}}{(WH)_{il}}  }{\sum_{l=1}^n H_{kl}},\varepsilon\Big\},
 H_{kj}^+= \max \Big\{ H_{kj} \frac{\sum_{l=1}^m \frac{W_{lk} V_{lj}}{(WH)_{lj}}  }{\sum_{l=1}^m W_{lk}},\varepsilon \Big\}. 
\end{equation} 
%which can be rewritten in the matrix form
%$$
%\begin{array}{ll}
%W^+= \max\Big\{ W\circ\, \Big(\big ((WH)^{[-1]} \circ V\big ) H/ (\mathbf{1}_{m \times n}H^\top )\Big),\varepsilon\Big\}, 
%\\
% H^+= \max\Big\{ H \circ\, \Big (W^\top \big ((WH)^{[-1]}\circ V\big) /(W^\top  \mathbf{1}_{m \times n})\Big),\varepsilon \Big\},
%\end{array}$$
%where $'\circ'$,  $'/'$ and $W^{[-1]}$ denote entry-wise multiplication, division and power operations. Since MU is a BSUM, it produces a non-increasing sequence of the objective function.  
 We note that the non-increasing property of  $f(x^k)$ produced by BSUM for Problem~\eqref{eq:compositev2} does not guarantee the convergence of $ \{x^k\}$. To guarantee some convergence for $ \{x^k\}$, we need $\varepsilon > 0$ for the objective function to be directionally differentiable, which is not the case when $\varepsilon = 0$, $(WH)_{ij}=0$ and $V_{ij}>0$ for some $i,j$.

\vspace{-0.1in}
\subsection{Proof of Theorem \ref{thm:BMDconvergence}} \label{app:BMDconvergence}
\textbf{Theorem \ref{thm:BMDconvergence}:} Suppose Assumption \ref{assump:RelativeSmooth} is satisfied. Let $\{x^k\}$ be the sequence generated by Algorithm \ref{alg:BMD}. 
Let also $\Phi(x)= f (x)+ \sum_i I_{\mathcal X_i}(x_i)$, where $I_{\mathcal X_i}$ is the indicator function of $\mathcal X_i$.  We have 

(i) $\Phi \big(x^k\big)$ is non-increasing;

(ii) Suppose $\mathcal X_i \subset {\rm int\, dom} \, \kappa_i$.  If   $\{x^k\}$ is bounded and $\kappa_i(x_i)$ is strongly convex on bounded subsets of $\mathcal X_i$ that contain $\{x_i^k\}$, 
then every limit point of $\{x^k\}$  is a critical point of $\Phi$;

(iii)  If together with the conditions  in (ii) we assume that  $\nabla \kappa_i$ and $\nabla_i f$ are Lipschitz continuous on bounded convex subsets of $\mathcal X_i$ that contain $\{x_i^k\}$, 
then the whole sequence  $\{x^k\}$  converges to a critical point of $\Phi$. 

\begin{proof}
We follow the methodology established in \cite{Bolte2014} that bases on the following theorem to prove the global convergence of BMD. 
\begin{theorem}
\label{thm:globalex}
Let $\Phi: \mathbb{R}^N\to (-\infty,+\infty]$ be a proper and lower semicontinuous function which is bounded from below. Let $\mathcal{A}$ be a generic algorithm which generates a bounded  sequence $\{x^{k}\}$ by $
x^{k+1}\in \mathcal A(x^{k})$, $k=0,1,\ldots$
Assume that the following conditions are satisfied.

(B1) 
\textbf{Sufficient decrease property}: There exists some $\rho_1>0$ such that
\[\rho_1 \|x^{k}-x^{k+1}\|^2 \leq  \Phi(x^{k}) - \Phi(x^{k+1}),   \forall \,k=0,1,\ldots.
\] 

(B2) \textbf{Boundedness of subgradient}:  There exists some $\rho_2>0$ such that
\[
\|w^{k+1}\|\leq \rho_2 \|x^{k}-x^{k+1}\| , w^{(k)}\in \partial \Phi (x^{k}), \forall k=0,1,\ldots
\]

(B3) \textbf{KL property}: $\Phi$ is a KL function.

(B4) \textbf{A continuity condition}: If a subsequence 
$\{x^{k_n}\}$   converges to  $\bar{x}$  then $\Phi(x^{k_n})\to \Phi(\bar{x})$ as $n\to \infty$.

Then $
\sum_{k=1}^\infty \|x^{k}-x^{k+1}\|<\infty
$
and $\{x^{k}\}$ converges to a critical point of $\Phi$. 
\end{theorem} 

We will prove Theorem \ref{thm:BMDconvergence} (iii) by verifying all the conditions in Theorem \ref{thm:globalex} for  $\Phi$. Let us first recall the following three-point property.

\begin{proposition}[Property 1 of \cite{Tseng2008}]\label{prop:threepoint}
Let $ z^+=	\arg\min_{u \in \mathcal X_i} \phi(u)+\mathcal B_{\kappa}(u,z)$, where $\phi$ is a proper convex function, $\mathcal X_i$ is convex and $\mathcal B_{\kappa}$ is the Bregman distance with respect to $\kappa_i$. Then for all $u\in \mathcal X_i$ we have 
\[
\phi (u) + \mathcal B_{\kappa}(u,z) \geq \phi(z^+) + \mathcal B_{\kappa}(z^+,z) + \mathcal B_{\kappa}(u,z^+).
\]
\end{proposition}
\noindent Denote $\Phi_i^k(x_i)=f_i^k(x_i)+I_{\mathcal X_i}(x_i)$. 
We have
\begin{equation*}
\begin{split}
f^k_i(x_i^{k+1})
& \leq f_i^k(x_i^{k} ) + \langle \nabla f_i^k(x_i^{k} ),x_i^{k+1}-x_i^{k}  )  \rangle + L_i^{k} \mathcal B_{\kappa_i} (x_i^{k+1},x_i^{k} )\\
&\leq f_i^k(x_i^{k} ) + \langle \nabla f_i^k(x_i^{k} ),x_i-x_i^{k} )  \rangle + L_i^{k}  \mathcal B_{\kappa_i} (x_i,x_i^{k} ) - L_i^{k}  \mathcal B_{\kappa_i} (x_i,x_i^{k+1} )\\
& \leq  f_i^k(x) + L_i^{k}  \mathcal B_{\kappa_i} (x,x_i^{k} )  - L_i^{k}  \mathcal B_{\kappa_i} (x,x_i^{k+1} ),
\end{split}
\end{equation*}
where the first inequality uses Assumption \ref{assump:RelativeSmooth}, the second inequality uses  Proposition \ref{prop:threepoint} applied for \eqref{eq:BMDstep} and the third uses convexity of $f_i$. 
Note that $I_{\mathcal X_i}(x)=0$ if $x\in \mathcal{X}$ and $I_{\mathcal X_i}(x)=+\infty$ otherwise. Hence, for all $x_i$, 
\begin{equation}
\label{eq:fk}
\Phi^k_i(x_i^{k+1}) \leq \Phi(x_1^{k+1},\ldots,x_{i-1}^{k+1},x_i,x_{i+1}^k,\ldots,x_s^k)+ L_i^{k}\big( \mathcal B_{\kappa_i} (x_i,x_i^{k} )  -  \mathcal B_{\kappa_i} (x_i,x_i^{k+1})\big).
\end{equation}
Choosing $x=x_i^k$ in \eqref{eq:fk}, we have
$ \Phi^k_i\big(x_i^{k+1}\big) \leq  \Phi^k_i(x_i^k) - \underline{L}_i  \mathcal B_{\kappa_i} \big(x_i^k,x_i^{k+1} \big).
$
Therefore, 
\begin{equation}
\label{ieg:decreasing} 
\begin{split}
\Phi \big(x^{k+1}\big) -  \Phi \big(x^{k}\big) &= \sum_{i=1}^s \Phi^k_i\big(x_i^{k+1}\big) - \Phi^k_i\big(x_i^{k}\big) \leq - \sum_{i=1}^s \underline{L}_i  \mathcal B_{\kappa_i} \big(x_i^k,x_i^{k+1} \big).
\end{split}
\end{equation}
From \eqref{ieg:decreasing}, we see that $\Phi \big(x^{k}\big)$ is non-increasing. Theorem \ref{thm:BMDconvergence}(i) is proved. 

As $\{x^k\}$ is assumed to be bounded, then there exists positive constants $A_i$ such that $\|x_i^k \| \leq A_i$, $i=1,\ldots,s$, for all $k$. Suppose a subsequence $x^{k_n} \to x^*$. As $ \mathcal X$ is closed convex set and $x^{k_n} \in \mathcal X$, then $x^*\in  \mathcal X$. Hence, $\Phi(x^{k_n}) \to \Phi(x^*)$, i.e., the continuity condition (B4) is satisfied.

Since $\kappa_i$ is strongly convex on the sets $\{x_i: x_i \in \mathcal X_i, \|x_i^k \| \leq A_i \}$, we have
$\mathcal{B}_{\kappa_i}\big(x_i^k,x_i^{k+1} \big)\geq \sigma_i/2\|x_i^k-x_i^{k+1}\|^2,
$
for some constant $\sigma_i$. Hence, from \eqref{ieg:decreasing} we derive the sufficient decrease property (B1) and $\sum_{k=1}^\infty \|x^k-x^{k+1}\|^2 <+\infty$. 
Consequently,  $\|x^k-x^{k+1}\| \to 0$. Then we also have $x^{k_n+1}\to x^*$. In \eqref{eq:fk} we let $k=k_n\to \infty$ and note that $x_i^*\in \mathcal X_i \subset {\rm  int\, dom}\,  \kappa_i$, then we obtain 
$\Phi(x^*)\leq \Phi(x_1^*,\ldots,x_i,\ldots,x_s^*)$,  $\forall \,x_i$. This means $x_i^*$ is a local minimizer of  $\min_{x_i}  \Phi(x_1^*,\ldots,x_i,\ldots,x_s^*)$. Hence $0\in \partial_i \Phi(x^*)$, for $i=1,\ldots s$. In other words, we have $0\in \partial \Phi(x^*)$, i.e., Theorem \ref{thm:BMDconvergence} (ii) is proved. 

To prove Theorem \ref{thm:BMDconvergence} (iii), it remains to verify the boundedness of subgradients.  From \cite[Proposition 2.1]{Attouch2010} we have
\begin{align*}
\partial \Phi\big(x^{k}\big)=\partial \Phi_1\big(x^{k}\big)\times \partial_s \Phi\big(x^{k}\big)=\{\nabla_1 f\big(x^{k}\big)+ \partial I_{\mathcal X_1}(x_1^k)\} \times \ldots \times \{\nabla_s f\big(x^{k}\big)+ \partial I_{\mathcal X_s}(x_s^k)\}.
\end{align*}
It follows from \eqref{eq:BMDstep} that 
$0\in \nabla f_i^k(x_i^k) + \partial I_{\mathcal X_i} (x_i^{k+1}) + L_i^{k+1} \big(\nabla \kappa_i (x_i^{k+1})-\nabla \kappa_i (x_i^{k})\big)$.  
Hence, 
\begin{equation}
\label{eq:temp1}
\begin{split}
\omega_i^{k+1}&=\nabla_i f\big(x^{k+1}\big)+ \partial I_{\mathcal X_1}(x_1^{k+1})\\
& =\nabla_i f\big(x^{k+1}\big)-\nabla f_i^k(x_i^k) +  L_i^{k} \big(\nabla \kappa_i (x_i^{k})-\nabla \kappa_i (x_i^{k+1})\big) \in \partial \Phi_i\big(x^{k+1}\big).
\end{split}
\end{equation} 
 As $\nabla_i f$ and $\nabla \kappa_i$ are Lipschitz continuous on  $\{x_i: \|x_i^k \| \leq A_i \}$, we derive from \eqref{eq:temp1} that there exists some constant $a_i$ such that $\|\omega_i^{k+1}\|\leq a_1 \|x^{k}-x^{k+1}\|$, for $i=1,\ldots,s$. Then the boundedness of subgradients in (B4) is satisfied. Theorem \ref{thm:BMDconvergence} is proved.
 
  \end{proof}
 \vspace{-0.1in}
\subsection{Proof of Proposition \ref{prop:SN_nonincreasing}} \label{app:proofSNmono} 
\textbf{Proposition \ref{prop:SN_nonincreasing}}:
%The objective sequence $\{f(W^k,H^k)\}$ generated by Algorithm~\ref{alg:SN} is non-increasing.
The objective function for the perturbed KL-NMF problem~\eqref{KLNMF} is non-increasing under the updates of Algorithm~\ref{alg:SN}. 
\begin{proof}
We note that if $g(x)$ is a self-concordant function with constant $\mathbf c$ then $\mathbf c^2\, g(x)$  is a standard self-concordant function. Hence, using the result of \cite[Theorem 6]{Tran2015} (see also Section \ref{sec:self-con}), we derive that the objective function of \eqref{KLNMF} is strictly decreasing when a damp Newton step is used. We now prove the proposition for the case when a full Newton step is used, that is when the gradient  $f'_{W_{ik}}\leq 0$ or $\lambda\leq 0.683802$ and the update of $W_{ik}$ then is $ W_{ik}^+= s = W_{ik} -  \frac{f'_{W_{ik}}}{f''_{W_{ik}}} $.

Consider the case $f'_{W_{ik}}\leq 0$. Let us use $f_{W_{ik}}$ to denote the objective of \eqref{KLNMF} with respect to $W_{ik}$. Considering the function $W_{ik}\mapsto g(W_{ik})=f'_{W_{ik}}$ with $W_{ik}\geq 0$,  we see that $g(W_{ik})$ is a concave function. Hence, we have 
$g(W_{ik}^+)-g(W_{ik}) \leq g'(W_{ik}) (W_{ik}^+-W_{ik}). 
$
This implies that 
$
f'_{W_{ik}^+} \leq f'_{W_{ik}} - f''_{W_{ik}}  \frac{f'_{W_{ik}}}{f''_{W_{ik}}}=0. 
$ 
Since $W_{ik} -W_{ik}^+ =   \frac{f'_{W_{ik}}}{f''_{W_{ik}}} \leq 0$ and $W_{ik} \mapsto f_{W_{ik}}$ is convex, we then obtain 
$f_{W_{ik}} \geq f_{W_{ik}^+} + f'_{W_{ik}^+}(W_{ik}-W_{ik}^+) \geq f_{W_{ik}^+}. 
$ Hence, the objective of \eqref{KLNMF} is non-increasing when we update $W_{ik}$ to $W_{ik}^+$. 

Consider the case $\lambda\leq 0.683802$. Denote $W_{ik}^\alpha=W_{ik} + \alpha d$. When $\alpha=1$ we have $W_{ik}^\alpha=s$, that is when a full proximal Newton step is applied. It follows from \cite[Inequality (58)]{Tran2015}  that 
$f_{W_{ik}^\alpha} \leq f_{W_{ik}} - \alpha \lambda^2 + \omega_*(\alpha \lambda)
$
for $\alpha \in (0,1]$, where $\omega_*(t)=- t - \log(1-t)$.  Hence, when $\alpha=1$ we get
$f_{W_{ik}^\alpha} \leq f_{W_{ik}} - (\lambda^2 +  \lambda + \log (1-\lambda)).$ It is not difficult to see that when $\lambda\leq 0.683802$ the value of $\lambda^2 +  \lambda + \log (1-\lambda)$ is positive. Hence, in this case, a full proximal Newton step also decreases the objective function. 
\end{proof} 
\end{document}